\documentclass[a4paper,10pt]{amsart}
\usepackage{amsmath,amsthm,amssymb}
\usepackage[latin1]{inputenc}
\usepackage[T1]{fontenc}
\usepackage[all]{xy}
\usepackage{mathrsfs}

\numberwithin{equation}{subsection}
\newtheorem{thm}[equation]{Theorem}
\newtheorem*{thma*}{Theorem A}
\newtheorem*{thmb*}{Theorem B}
\newtheorem*{thmc*}{Theorem C}
\newtheorem*{thm*}{Theorem}
\newtheorem{cor}[equation]{Corollary}
\newtheorem{lem}[equation]{Lemma}
\newtheorem{prop}[equation]{Proposition}

\newtheorem{defprop}[equation]{Definition/Proposition}

\theoremstyle{definition}

\newtheorem{rem}[equation]{Remark}
\newtheorem{defn}[equation]{Definition}

\DeclareMathOperator{\h}{H} 
\DeclareMathOperator{\sgn}{sgn} 
 
\newcommand{\triv}{{\mathbf{1}}}
\def\R{\mathbb R}

\def\Z{\mathbb Z}
\def\A{\mathbb A}
\def\Q{\mathbb Q}
\def\C{\mathbb C}
\def\F{\mathbb F}

\def\ira{\stackrel{\sim}{\longrightarrow}}
\def\hra{\hookrightarrow}

\def\ra{\rightarrow}

\def\g{\mathfrak g}
\def\gl{\mathfrak{gl}}
\def\t{\textbf{\emph{t}}}

\def\n{\mathfrak n}

\def\O{\mathcal O}
\def\h{\mathfrak h}

\def\k{\mathfrak k}

\def\<{\langle}
\def\>{\rangle}

\def\GL{{\rm GL}}
\def\GSpin{{\rm GSpin}}
\def\SO{{\rm SO}}
\def\GSp{{\rm GSp}_4}
\def\Orth{{\rm O}}
\def\S{\mathcal{S}} 
\def\G{\mathcal{G}}
\def\T{\mathcal{T}}

\def\Hom{{\rm Hom}}

\usepackage{bm}
\usepackage{upgreek}
\makeatletter
\newcommand{\bfgreek}[1]{\bm{\@nameuse{up#1}}}
\makeatother

\title[Shalika models and critical values]{On the arithmetic of Shalika models and \\
the critical values of $L$-functions for $\GL_{2n}$}
\author{ Harald Grobner \ \and \ A. Raghuram}
\date{\today}

\address{Harald Grobner: Fakult\"at f\"ur Mathematik, University of Vienna\\ Oskar--Morgenstern--Platz 1\\ A-1090 Vienna, Austria.}
\email{harald.grobner@univie.ac.at}
\urladdr{http://homepage.univie.ac.at/harald.grobner}
\address{A. Raghuram: Department of Mathematics\\ Oklahoma State University\\ 401 Mathematical Sciences\\ Stillwater, OK 74078, USA. Current Address: Indian Institute of Science Education and Research, Dr.\ Homi Bhabha Road, Pashan, Pune 411021, India.}
\email{raghuram@iiserpune.ac.in}

\thanks{H.G. is supported by the Austrian Science Fund (FWF) Erwin Schr\"odinger grant, J 3076-N13. A.R. is partially supported by the National Science Foundation (NSF), award number DMS-0856113, and an Alexander von Humboldt Research Fellowship.}\thanks{This is an extended and slightly modified version of an article, which appeared in {\it Amer. J. Math.} {\bf 136} (2014)}

\begin{document}
\maketitle
\centerline{{\it With an appendix by Wee Teck Gan}}
\normalsize
\begin{abstract}
Let $\Pi$ be a cohomological cuspidal automorphic representation of ${\rm GL}_{2n}(\A)$ over a totally real number field $F$. Suppose that $\Pi$ has a Shalika model.
We define a rational structure on the Shalika model of $\Pi_f.$ Comparing it with a rational structure on a realization of $\Pi_f$ in cuspidal cohomology in top-degree, we define certain periods $\omega^{\epsilon}(\Pi_f)$. We describe the behaviour of such top-degree periods upon twisting $\Pi$ by algebraic Hecke characters $\chi$ of $F$. Then we prove an algebraicity result for all the critical values of the standard $L$-functions $L(s, \Pi \otimes \chi);$ here we use the recent work of B. Sun on the non-vanishing of a certain quantity attached to $\Pi_\infty$. As applications, we obtain  algebraicity results in the following cases: Firstly, for the symmetric cube $L$-functions attached to holomorphic Hilbert modular cusp forms; we also discuss the situation for higher symmetric powers. Secondly, for certain
(self-dual of symplectic type) Rankin--Selberg $L$-functions for $\GL_3\times\GL_2;$ assuming Langlands Functoriality, this generalizes to certain Rankin--Selberg $L$-functions of $\GL_n\times\GL_{n-1}$.
Thirdly, for the degree four $L$-functions attached to Siegel modular forms of genus $2$ and full level. Moreover, we compare our top-degree periods with periods defined by other authors. We also show that our main theorem is compatible with conjectures of Deligne and Gross.
\end{abstract}

\setcounter{tocdepth}{1}
\tableofcontents
\section{Introduction}
Let $F$ be a totally real number field and $G=\GL_{2n}/F$, $n\geq 1$, the split general linear group over $F$. Let $\Pi$ be a cuspidal automorphic representation of $G(\A)$ and $\chi$ an algebraic Hecke character. Attached to this data is the standard Langlands $L$-function $L(s,\Pi\otimes\chi)$. The main aim of this paper is to study the algebraicity of the critical values of $L(s,\Pi\otimes\chi)$ for representations $\Pi$ which admit a Shalika model. To that end, we will also investigate the arithmetic of such Shalika models.

Manin and Shimura independently studied the arithmetic of $L$-functions associated to holomorphic Hilbert modular forms, see \cite{manin} and \cite[Thm.\ 4.3]{shimura}; and more generally, Harder \cite{hardermodsym} and Hida \cite{hida-duke} proved algebraicity theorems for critical values of $L$-functions of cuspidal automorphic representations of $\GL_2$. These results relate the $L$-value at hand to a non-zero complex number, called a {\it period}, which essentially captures the transcendental part of the critical $L$-values. Underlying this construction is the fact that the cuspidal automorphic representations considered are of cohomological type, i.e., have non-vanishing cohomology with respect to some finite-dimensional algebraic coefficient module. Then the periods arise from comparing a rational structure on the Whittaker model of the finite part of the cuspidal representation and a realization of the latter in cohomology.
This idea of defining a period attached to cohomological cuspidal representations was pursued by several other authors, among them Kazhdan--Mazur--Schmidt \cite{kazhdan-mazur-schmidt}, Mahnkopf \cite{mahnk}, Raghuram \cite{raghuram-imrn} and Raghuram--Shahidi \cite{raghuram-shahidi-imrn}. All these works have in common that they use the Whittaker model and the {\it lowest} possible degree of cohomology which can carry a cuspidal automorphic representation.
Here, we replace the Whittaker model by what is called the Shalika model -- if there is one -- of a cuspidal automorphic representation $\Pi$ of $G(\A)$; furthermore, we will work with the {\it highest} possible degree of cohomology in which a cuspidal representation may contribute. In particular, this approach gives rise to different periods than the ones considered by the previous authors.

\smallskip

To put ourselves {\it in medias res}, let $\S/F=\Delta\GL_n\cdot {\rm M}_n\subset\GL_{2n}$ be the Shalika subgroup of $G=\GL_{2n}$, $\eta$ an id\`ele class character of $F$ such that $\eta^n$ equals the central character $\omega_\Pi$ of $\Pi$ and $\psi$ a non-trivial additive character of $F\backslash \A$. The latter two characters naturally extend to characters of $\S(\A)$, cf.\  Sect.\ \ref{sect:globalShalikamodels}. A cuspidal automorphic representation $\Pi$ of $G(\A),$ which we do not assume to be unitary, is said to have an $(\eta,\psi)$-Shalika model, if
$$\S^\eta_\psi(\varphi)(g):=\int_{Z_{G}(\A)\S(F)\backslash \S(\A)} (\Pi(g)\cdot\varphi)(s)\eta^{-1}(s)\psi^{-1}(s) ds\neq 0$$
for some $\varphi\in\Pi$ and $g\in G(\A)$. According to Jacquet--Shalika \cite{jacshal}, this is equivalent to a twisted partial exterior square $L$-function
$L^S(s,\Pi,\wedge^2\otimes\eta^{-1})=\prod_{v\notin S} L(s,\Pi_v,\wedge^2\otimes\eta^{-1}_v)$ having a pole at $s=1$, cf.\  Thm.\ \ref{thm:JS}. One may again reformulate this by saying that $\Pi$ has a Shalika model if and only if $\Pi$ is the Asgari--Shahidi transfer of a globally generic, cuspidal automorphic representation of ${\rm GSpin}_{2n+1}(\A)$, see Prop.\ \ref{prop:selfdual}. For the definition of the Shalika model $\S^{\eta}_{\psi}(\Pi)$, see Def.\ \ref{defn:shalika-model}.

If $\Pi$ is cohomological and cuspidal, then we know that its $\sigma$-twist ${}^\sigma\Pi:=\otimes_{v\textrm{ arch.}} \Pi_{\sigma^{-1} v}\otimes(\Pi_f\otimes_{\C,\sigma^{-1}}\C)$
is also cohomological and cuspidal for all $\sigma\in$Aut$(\C);$ see Clozel \cite{clozel}.  We would like to define an action of ${\rm Aut}(\C)$ on Shalika models and hence define rational structures on such models. Toward this we have the following theorem which says that having a Shalika model is an arithmetic property of a cohomological cuspidal automorphic representation $\Pi.$  See  Thm.\ \ref{thm:arithmeticShalika}. The appendix of this article contains a simple and elegant proof of this theorem by Wee Teck Gan. \\

Under our present assumptions, it is known that the rationality field $\Q(\Pi)$ of $\Pi$, i.e., the fixed field of all automorphisms $\sigma\in$ Aut$(\C)$ which leave $\Pi$ invariant, $\Pi\cong{}^\sigma\Pi$, is an algebraic number field. The same holds for $\Q(\Pi,\eta)$, the compositum of the rationality fields of $\Pi$ and $\eta$.
By virtue of Thm.\,\ref{thm:arithmeticShalika}, we are able to define a ``$\sigma$-twisted action'' on the Shalika model $\S^{\eta_f}_{\psi_f}(\Pi_f)$ of the finite part $\Pi_f$ of $\Pi$, and hence obtain a $\Q(\Pi,\eta)$-structure on $\S^{\eta_f}_{\psi_f}(\Pi_f)$. That is, there is a $\Q(\Pi,\eta)$-subspace of $\S^{\eta_f}_{\psi_f}(\Pi_f)$, stable under the action of $G(\A_f)$, which - tensored by $\C$ - retrieves the Shalika model, see Lem.\ \ref{lem:rational-shalika}.

Let $K_\infty=\prod_{v~ {\rm arch.}} {\rm O}(2n)\R^\times$ and $q_0=\dim_\Q(F)\cdot(n^2+n-1)$. Then $q_0$ is the highest degree in which a cuspidal automorphic representation $\Pi$ can have non-vanishing $(\g_\infty,K_\infty^\circ)$-cohomology with respect to some finite-dimensional, irreducible algebraic coefficient system $E^{{\sf v}}_\mu$.
It is known that every character $\epsilon$ of $\pi_0(G_\infty)=K_\infty/K_\infty^\circ$ appears in
$H^{q_0}(\g_\infty,K^\circ_\infty,\Pi_\infty\otimes E^{\sf v}_\mu)$ with multiplicity one. Hence, taking the $\epsilon$-isotypic component gives a one-dimensional space
$H^{q_0}(\g_\infty,K^\circ_\infty,\Pi_\infty\otimes E^{\sf v}_\mu)[\epsilon]\cong\C.$
Fixing a basis vector $[\Pi_\infty]^\epsilon$ of the former cohomology space defines an isomorphism
$$
\Theta_\Pi^\epsilon: \S^{\eta_f}_{\psi_f}(\Pi_f)\ira H^{q_0}(\g_\infty,K^\circ_\infty,\Pi\otimes E^{\sf v}_\mu)[\epsilon].
$$
The right hand side also has a $\Q(\Pi,\eta)$-structure, which originates from a geometric realization of automorphic cohomology. One may normalize $\Theta^\epsilon$ in such a way that it respects the $\Q(\Pi,\eta)$-structures on both sides. This normalization factor is a period which we denote $\omega^\epsilon(\Pi_f)$;  it is well-defined as an element of $\C^\times/\Q(\Pi,\eta)^\times.$ See Definition/Proposition~\ref{defprop}. Call this normalized isomorphism $\Theta_{\Pi,0}^{\epsilon}$.\\

In Sect.\,\ref{sect:twisted} we prove the first main algebraicity result of this paper; see Theorem~\ref{thm:twisted}.
It describes the behaviour of our top-degree periods $\omega^\epsilon(\Pi_f)$ under twisting $\Pi$ by an algebraic Hecke character $\chi$ of $F$. Let $\G(\chi_f)$ be the Gau\ss~sum of $\chi_f$ and $\epsilon_\chi$ the signature of $\chi$, cf.\  Sect.\,\ref{sect:characters}.
Let $\epsilon$ be a character of $K_{\infty}/K_{\infty}^\circ$ with rationality field $\Q(\epsilon)$ and let $\omega^{\epsilon}(\Pi_f)$ be the attached period. Let $\chi$ be an algebraic Hecke character of $F$, and let $\epsilon_{\chi}$ be its signature. For any $\sigma \in {\rm Aut}({\mathbb C})$ we have
$$
\sigma\left(\frac
{\omega^{\epsilon \cdot \epsilon_{\chi}}(\Pi_f\otimes\chi_f)}
{\G(\chi_f)^n \, \omega^{\epsilon}(\Pi_f) }\right)
 \ = \
\left(\frac
{\omega^{\epsilon \cdot \epsilon_{\chi}} ({}^{\sigma}\Pi_f \otimes {}^\sigma\!\chi_f)}
{\G( {}^\sigma\!\chi_f)^n \, \omega^{\epsilon} ({}^{\sigma}\Pi_f)}\right).
$$
This roughly says that the periods of $\Pi_f\otimes\chi_f$ differ from the periods of $\Pi_f$ by the $n$-th power of the Gau\ss~sum of $\chi$ up to an algebraic number in a canonical number field determined by the data at hand. In the context of periods arising from Whittaker models and bottom-degree cohomology, such a theorem was proved by the second author and Shahidi; see the main theorem of \cite{raghuram-shahidi-imrn}. The strength of Thm.\,\ref{thm:twisted} relies on the fact that in order to prove an algebraicity theorem for all the critical values of $L(s,\Pi\otimes\chi)$, it suffices to prove an algebraicity theorem for just {\it one} critical value of the {\it untwisted} $L$-function $L(s,\Pi)$. In the rest of the introduction we show how one can prove such an algebraicity theorem for $L(\tfrac12, \Pi)$, assuming that $s=\tfrac12$ is critical for $L(s,\Pi)$. \\

For a Shalika function $\xi_{\varphi}=\S^\eta_\psi(\varphi)\in\S^\eta_\psi(\Pi)$, following Friedberg--Jacquet \cite{friedjac}, one
may define the Shalika-zeta-integral
$$\zeta(s,\varphi):=\int_{\GL_n(\A)} \S^\eta_\psi(\varphi)\left(\!\!\left(\!\! \begin{array}{ccc}
g_1 &  0\\
0 &  1
\end{array}\!\!\right)\!\!\right) |\det(g_1)|^{s-1/2} dg_1,$$
which may be shown to extend to a meromorphic function in $s\in\C$, cf.\  Prop.\ \ref{prop:FJ}. This also makes  sense locally, i.e., at a place $v$ of $F$.
Under the standing assumption that $\Pi$ is cohomological and cuspidal automorphic and admits an $(\eta, \psi)$-Shalika model, we prove in Sect.\ \ref{sect:vector} that there is a very special vector $\xi^\circ_{\Pi_f}$ in the Shalika model $\S^{\eta_f}_{\psi_f}(\Pi_f)$ which satisfies
\begin{enumerate}
\item $\zeta_v(\frac12,\xi^\circ_{\Pi_v})=L(\frac12,\Pi_v)$ for all unramified finite places $v$,
\item $\zeta_v(\frac12,\xi^\circ_{\Pi_v})=1$ for all ramified finite places $v$.
\end{enumerate}
Moreover, this vector transforms compatibly under the action of Aut$(\C)$, i.e., ${}^\sigma\xi^\circ_{\Pi_v}=\xi^\circ_{{}^\sigma\Pi_v}$. Let $H:=\GL_n\times\GL_n$ which is naturally a subgroup of  $G=\GL_{2n}.$ Next, one uses the result of Friedberg--Jacquet that the period integral along $H(F)\backslash H(\A)$ of a cusp form $\varphi$ is nothing but the Shalika-zeta-integral of $\xi_\varphi.$ Hence, we get an integral representation of the central critical (partial) $L$-value $L^S(\tfrac12,\Pi)$ as a period integral of a cusp form, which in the Shalika model corresponds to our special vector $\xi^\circ_{\Pi_f}$. Our main algebraicity result follows by interpreting this period integral in cohomology. Towards such a cohomological interpretation, consider the real orbifolds
$$\tilde{S}^H_{K_f}:= H(F)\backslash H(\A)/(K^\circ_\infty\cap H_\infty)^\circ \iota^{-1}(K_f)\longrightarrow G(F)\backslash G(\A)/K^\circ_\infty K_f=:S^G_{K_f},$$
where $\iota:H\hookrightarrow G$ denotes the natural embedding of $H$ into $G$ and $K_f$ is an open compact subgroup of $G(\A_f)$. {\it It is a crucial observation that $\dim_\R \tilde{S}^H_{K_f} = q_0$}. This numerical coincidence is a very important ingredient in making  the whole story work. Another important ingredient, which follows from a classical branching law (cf.\  Prop.\ \ref{prop:knapp}), is the observation that $s=\tfrac12$ is critical for $L(s,\Pi)$ if and only if the essentially  trivial representation of $H_\infty$ appears (and then necessarily with multiplicity one) in the representation $E_\mu^{\sf v}$.  Finally, we use a version of Poincar\'e-duality $\int_{\tilde{S}^H_{K_f}}$ (cf.\  Sect.\ \ref{sect:Poinc}) to obtain our main diagram of maps, see Sect.\ \ref{sect:diagram} for details and notation unexplained here:
$$
\xymatrix{
H_c^{q_0}(S^G_{K_f},\mathcal E^{\sf v}_\mu) \ar[r]^{\iota^*} & H_c^{q_0}(\tilde{S}^H_{K_f},\mathcal E^{\sf v}_\mu) \ar[rr]^{\mathcal T^*} & &
H_c^{q_0}(\tilde{S}^H_{K_f},\mathcal E_{(0, -{\sf w})}) \ar[dd]^{\int_{\tilde{S}^H_{K_f}}}\\
H^{q_0}(\mathfrak{g}_{\infty},K_{\infty}^\circ; \Pi\otimes E^{\sf v}_{\mu})[\epsilon_0]^{K_f} \ar@{^{(}->}[u] & & \\
\S^{\eta_f}_{\psi_f}(\Pi_f)^{K_f} \ar[u]^{\Theta_{\Pi,0}^{\epsilon_0}} \ar[rrr]& & & \C}
$$

The ``Main Identity'' proved in Thm.\ \ref{thm:mainid} shows that chasing our special vector $\xi^\circ_{\Pi_f}$ through this diagram essentially computes the
$L$-value $L(\tfrac12,\Pi_f)$. Here, we use recent work of Sun \cite{sun}, which shows the non-vanishing of a quantity $\omega(\Pi_\infty)$, depending only on the choice of generator $[\Pi_\infty]^\epsilon.$ This, together with our theorem on period relations then gives the second main algebraicity result of this paper; see Theorem~\ref{thm:central-value}.
Let $\chi$ be a finite-order Hecke character of $F$ and Crit$(\Pi)=$ Crit$(\Pi\otimes\chi)$ be the set of critical points in $\tfrac12 + \Z$
for the $L$-function $L(s,\Pi\otimes\chi)$ of $\Pi\otimes\chi$. Let $\tfrac12 + m \in$ Crit$(\Pi\otimes\chi)$. Then,
for any $\sigma\in {\rm Aut}(\C)$, we have
$$
\sigma\left(
\frac{L(\tfrac 12+m,\Pi_f \otimes \chi_f)}{\omega^{(-1)^{m+n-1}\epsilon_{\chi}}(\Pi_f) \, \G(\chi_f)^n \, \omega(\Pi_\infty,m)}\right) \ = \
\frac{L(\tfrac 12+m,{}^\sigma\Pi_f \otimes {}^\sigma\!\chi_f)}{\omega^{(-1)^{m+n-1}\epsilon_\chi}({}^\sigma\Pi_f) \, \G({}^\sigma\!\chi_f)^n \,\omega(\Pi_\infty,m)},
$$
where the quantity $\omega(\Pi_\infty,m)$ is defined in Thm.\,\ref{hyp}.  \\

For the case of trivial coefficients (i.e., when $\mu = 0$) for the group $\GL_4$, a weak form of the above theorem is implicit in a construction of $p$-adic $L$-functions due to Ash--Ginzburg \cite{ash-ginzburg}. (There the authors worked over $\overline{\Q}$ -- the algebraic closure of $\Q$ in $\C$ -- instead of the number field $\Q(\Pi,\eta, \chi)$ and needed to assume the non-vanishing of $L(\tfrac12,\Pi_f\otimes\chi_f)$ for some unitary character $\chi$ trivial at infinity.) There are several parts of that paper which are for $\GL_{2n}$, however, to quote them from the introduction of their paper, ``our results are definitive when $n=2$ and $F$ totally real.'' The reader should view our Thm.\,\ref{thm:central-value} as a generalization, as well as a refinement, of some of the results of \cite{ash-ginzburg}.

For a cohomological cuspidal representation $\pi$ of ${\rm GL}_n/{\mathbb Q}$, Mahnkopf \cite{mahnk} was the first to prove a general rationality result for the critical values of the standard $L$-function $L(s, \pi)$. (See \cite[Thm.\ A]{mahnk}.) His rationality result, which is under the assumption of a non-vanishing hypothesis, is formulated in terms of certain periods $\Omega^\pm(\pi)$ attached to $\pi.$ These periods however depend not only on $\pi$, but also on a series of representations
$\pi = \pi_0, \pi_1, \dots ,$ where $\pi_j$ is a representation of ${\rm GL}_{n-2j}/{\mathbb Q}$. Unfortunately, $\pi$ does not canonically determine the $\pi_j$'s; besides, there is no relation between the rationality fields $\Q(\pi)$ and $\Q(\pi_j)$. This raises significant problems in any particular instance; for example, using Langlands Functoriality for the symmetric cube transfer, it seems impossible to apply Mahnkopf's results to prove that the critical values of the symmetric cube $L$-function of the Ramanujan $\Delta$-function, divided by his periods $\Omega^\pm({\rm Sym}^3(\Delta))$, are rational numbers. (See, for example, Mizumoto~\cite{mizumoto} for the critical values of the $L$-function of
${\rm Sym}^3(\Delta)$.)
In comparison, our Thm.\,\ref{thm:central-value}, which is totally independent of Mahnkopf's paper, has the advantage that it is unconditional, and furthermore, it is sufficiently refined to give algebraicity results for critical values of concrete examples like the symmetric cube $L$-functions of Hilbert modular forms, or of the degree four $L$-functions of Siegel modular forms.

In Sect.\ \ref{sec:complementa} we take up various such examples to which Thm.\,\ref{thm:central-value} is applicable. Consider a primitive holomorphic Hilbert modular cusp form ${\bf f}$ of weight $k=(k_v)_{v\in S_\infty}$ and let $\pi({\bf f})$ be the corresponding cuspidal automorphic representation of $\GL_2(\A)$. If ${\bf f}$ is algebraic, we prove that $\Pi:= {\rm Sym}^3(\pi)$ being the Kim--Shahidi symmetric cube transfer of $\pi=\pi({\bf f})\otimes|\cdot|^{k_0/2}$, $k_0=\max k_v$, satisfies all the assumptions made in our Thm.\,\ref{thm:central-value}, cf.\  Prop.\ \ref{prop:hilbertmod}. Hence, we get a new algebraicity theorem for the critical values of such symmetric cube $L$-functions, see Cor.\ \ref{cor:hilbertmod}. The reader should compare this with a previous theorem of Garrett--Harris \cite[Thm.\ 6.2]{garrett-harris} on symmetric cube $L$-functions. In fact, using their paper, we derive Cor.\ \ref{cor:period-relns-garrett-harris}, which compares our top-degree periods with the Petersson inner product of ${\bf f}$. Further, assuming Langlands Functoriality, we get a theorem for all odd symmetric power $L$-functions of ${\bf f}$, see Prop.\ \ref{prop:highersym}. This should be compared with Raghuram \cite[Thm.\ 1.3]{raghuram-imrn}.

Next, we consider Rankin--Selberg $L$-functions for $\GL_3 \times \GL_2$ attached to a pair $(\pi,\tau)$ of unitary cuspidal automorphic representations. This $L$-function is the standard $L$-function of the Kim--Shahidi transfer $\Pi:=\pi \boxtimes \tau$, which is a representation of $\GL_6(\A)$. We show that if $\pi$ is essentially self-dual, $\tau$ is not dihedral and $\pi$ is not a twist of the Gelbart--Jacquet transfer of $\tau$, then $\pi \boxtimes \tau$ is a cuspidal automorphic representation; if $\pi_\infty$ and $\tau_\infty$ are also cohomological and ``sufficiently disjoint'' (a mild condition on their Langlands parameters), then we prove that $\Pi$ is cohomological, too, and finally we verify that $\Pi$ admits a Shalika model, see Prop.\ \ref{prop:gl2gl3}. Hence, we get a new algebraicity theorem for the critical values of $L(s,\pi\times\tau)$, cf.\  Cor.\ \ref{cor:gl2gl3}. We compare this result with Raghuram~\cite[Thm.\ 1.1]{raghuram-imrn} in Cor.\ \ref{cor:gl2gl3-period-relations}, which yields a comparison of our top-degree (Shalika--)periods and the bottom-degree (Whittaker--)periods used in the aforementioned reference. In Sect.\ \ref{sect:glngln-1}, we indicate how these algebraicity theorems may be extended to the case of Rankin--Selberg $L$-functions of $\GL_n\times\GL_{n-1}$ assuming Langlands Functoriality.

As another class of examples, let $\Phi$ be a non-zero genus two cuspidal Siegel modular eigenform of full level. By a recent work of Pitale--Saha--Schmidt \cite{pss}, one knows the existence of the Langlands transfer of $\Phi$ to a cuspidal automorphic representation $\Pi(\Phi)$ of $\GL_4(\A_\Q)$. We check that our Thm.\,\ref{thm:central-value} applies to $\Pi(\Phi)$, giving a new theorem on the critical values of the degree four $L$-function of Siegel modular cusp forms, cf.\  Cor.\ \ref{cor:siegelmod}. This should be compared with Harris \cite[Thm.\ 3.5.5]{harris-occult}.

In Sect.\ \ref{sec:complementa} we also comment on the compatibility of our theorem with Deligne's conjecture on the critical values of motivic $L$-functions. As it stands, it seems impossible to compare our periods with Deligne's motivic periods directly. However, Blasius and Panchishkin have independently computed the behaviour of Deligne's periods upon twisting the motive by characters, and based on this they predict how critical values of automorphic $L$-functions change upon twisting. Our theorem is compatible with their predictions; see Cor.\ \ref{cor:deligne}. We also note that our result is compatible with Gross's conjecture on the order of vanishing of motivic $L$-functions at critical points; see
Cor.\ \ref{cor:gross}. \\

\noindent{\it Acknowledgements: } \small We are grateful to G\"unter Harder for valuable discussions on automorphic cohomology. We thank Ameya Pitale, Abhishek Saha and Ralf Schmidt for showing us their preprint \cite{pss}; Binyong Sun for sending us his preprint \cite{sun}; and Wee Teck Gan for writing the appendix with his proof of arithmeticity of Shalika models. A.R. thanks Paul Garrett for a lot of helpful discussions. We are grateful to the referee for a thorough reading and for several insightful comments. Finally, both H.G. and A.R. thank the Max Planck Institut f\"ur Mathematik, Bonn, for its hospitality; this work was conceived and carried through when both were visitors at the Max Planck Institut.
\normalsize

\section{Notation and conventions}

\subsection{}
Let $F$ be a totally real number field of degree $d=[F:\Q]$ with ring of integers $\O$. For any place $v$ we write $F_v$ for the topological completion of $F$ at $v$. Let $S_\infty$ be the set of archimedean places of $F$. If $v \notin S_{\infty}$, we let $\O_v$ be the local ring of integers of $F_v$ with unique maximal ideal $\wp_v$. Moreover, $\A$ denotes the ring of ad\`eles of $F$ and $\A_f$ its finite part. We use the local and global normalized absolute values and denote each of them by $|\cdot|$. Further, $\mathfrak{D}_F$ stands for the absolute different of $F$, i.e., $\mathfrak{D}_F^{-1} = \{x \in F : Tr_{F/\Q}(x \O) \subset \Z\}$.

\subsection{}\label{sect:GH}
Throughout this paper we let $G:=\GL_{2n}/F$, $n\geq 1$, the split general linear group over $F$. Let $H:=\GL_n\times \GL_n/F$ which is viewed as a subgroup of $G$ consisting of block diagonal matrices. The center of $G$ is denoted $Z_G/F$. If $A$ is any abelian $F$-algebra, $G(A)$ (resp., $H(A)$) stands for the $A$-rational points of $G$ (resp., $H$). In accordance with the usual conventions, we write $G_\infty=\prod_{v\in S_{\infty}} G(F_v)=\GL_{2n}(\R)^d$ (resp., $H_\infty=\prod_{v\in S_{\infty}} H(F_v)=(\GL_{n}(\R)\times \GL_n(\R))^d$). Lie algebras of real Lie groups are denoted by the same letter but in lower case gothics; for example, $\g_\infty=Lie(G_\infty)$, $\g_v=Lie(G(F_v))$, $v\in S_{\infty}$.

\subsection{}\label{sect:highweights}
We fix once and for all a maximal $F$-split torus of $G$, the group of diagonal matrices in $G$. Fixing positivity on the corresponding set of roots in the usual way gives us that the set of tuples $\mu=(\mu_v)_{v\in S_{\infty}}$, $\mu_v=(\mu_{v,1},...,\mu_{v,2n})$ with $\mu_{v,1}\geq ...\geq\mu_{v,2n}$ and $\mu_{v,i}\in\Z$, for all $v\in S_{\infty}$ and $1\leq i\leq 2n$, can be identified with the set of equivalence classes of irreducible finite-dimensional algebraic representations $E_\mu$ of $G_\infty$ (on complex vector spaces) via the highest weight correspondence. It is clear that any such representation $E_\mu$ factors as $E_\mu=\bigotimes_{v\in S_{\infty}} E_{\mu_v}$, where $E_{\mu_v}$ is the irreducible representation of $G(F_v)=G(\R)$ of highest weight $\mu_v$. The representation $E_\mu$ is called {\it essentially self-dual} if all its local factors $E_{\mu_{v}}$ are, i.e., if for all $v\in V_\infty$ there is a $w_v\in\Z$ such that
$$\mu_{v,i}+\mu_{v,2n-i+1}=w_v, \quad\quad 1\leq i\leq n.$$
This is equivalent to saying that $E_{\mu_v}\cong E^{\sf v}_{\mu_v}\otimes\det^{w_v}$. It is called {\it self-dual} if $w_v=0$, i.e., $E_{\mu_v}\cong E^{\sf v}_{\mu_v}$.

\subsection{}\label{sect:gKcoh}
At an archimedean place $v\in S_{\infty}$ we let $K_v$ be the product of a maximal compact subgroup of the real Lie group $G(F_v)=\GL_{2n}(\R)$ and $Z_G(F_v)=Z_G(\R)$. We make the following explicit choice:
$$K_v= \Orth(2n)\R^\times,$$
and set $K_\infty=\prod_{v\in S_\infty} K_v$. By $K^\circ_\infty$ we mean the topological connected component of the identity within $K_\infty$. Hence, locally
$$K^\circ_v= \SO(2n)\R_+.$$
All Lie-group representations $\Pi_\infty=\bigotimes_{v\in S_{\infty}} \Pi_v$ of $G_\infty$ appearing in this paper define a $(\g_\infty,K^\circ_\infty)$-module and
for each $v \in S_{\infty}$ a $(\g_v,K^\circ_v)$-module, which we shall all denote by the same letter as the original Lie group representation.
In particular, this applies to a highest weight representation  $E_\mu=\bigotimes_{v\in S_{\infty}} E_{\mu_v}$. If furthermore $\Pi_\infty=\bigotimes_{v\in S_{\infty}} \Pi_v$ is any $(\g_\infty,K^\circ_\infty)$-module, then we denote by
$H^q(\g_\infty,K^\circ_\infty,\Pi_\infty)$
its space of $(\g_\infty,K^\circ_\infty)$-cohomology in degree $q$, cf.\ Borel--Wallach \cite{bowa}, I.5. A module $\Pi_\infty$ is called {\it cohomological}, if there is a highest weight representation $E_\mu$ as in Sect.\ \ref{sect:highweights}, such that $H^q(\g_\infty,K^\circ_\infty,\Pi_\infty\otimes E^{\sf v}_\mu)\neq 0$ for some degree $q$. It is a basic fact that these cohomology groups obey the K\"unneth-rule, i.e.,
$$H^q(\g_\infty,K^\circ_\infty,\Pi_\infty\otimes E^{\sf v}_\mu)\cong\bigoplus_{\sum_v q_v=q}\bigotimes_{v\in S_{\infty}} H^{q_v}(\g_v,K^\circ_v,\Pi_v\otimes E^{\sf v}_{\mu_v}).$$
Hence, $\Pi_\infty$ is cohomological, if and only if all its local components $\Pi_v$ are, i.e., they have non-vanishing $(\g_v,K^\circ_v)$-cohomology with respect to some local highest weight representation $E^{\sf v}_{\mu_v}$.

\subsection{}\label{sect:sigmatwist}
For $\sigma\in \textrm{Aut}(\C)$, let us define the {\it $\sigma$-twist} ${}^\sigma\!\nu$ of a representation $\nu$ of $G(\A_f)$ (resp., $G(F_v)$, $v\notin S_\infty$) on a complex vector space $W$ as in Waldspurger \cite{waldsp}, I.1: If $W'$ is a $\C$-vector space with a $\sigma$-linear isomorphism $t':W\ra W'$ then we set
$${}^\sigma\!\nu:= t'\circ\nu\circ t'^{-1}.$$
This definition is independent of $t'$ and $W'$ up to equivalence of representations. If $\nu_\infty=\bigotimes_{v\in S_{\infty}}\nu_{v}$ is a representation of $G_\infty$, we let $${}^{\sigma}\!\nu_\infty:=\bigotimes_{v\in S_{\infty}}\nu_{\sigma^{-1} v},$$
interpreting $v\in S_{\infty}$ as an embedding of fields $v:F\hookrightarrow\R$. For $\sigma\in\textrm{Aut}(\C)$, this defines the $\sigma$-twist on a global representation $\nu=\nu_\infty\otimes\nu_f$ of $G(\A)$ be setting
$${}^\sigma\!\nu:={}^{\sigma}\!\nu_\infty\otimes {}^\sigma\!\nu_f.$$
Recall also the definition of the rationality field of a representation from \cite{waldsp}, I.1. If $\nu$ is any of the representations considered above, then let $\mathfrak S(\nu)$ be the group of all automorphisms $\sigma\in \textrm{Aut}(\C)$ such that ${}^\sigma\!\nu\cong\nu $. Then the {\it rationality field} $\Q(\nu)$ is defined as
$$\Q(\nu):=\{z\in\C| \sigma(z)=z \textrm{ for all } \sigma\in\mathfrak S(\nu)\}.$$
As another ingredient we recall that a representation $\nu$ on a $\C$-vector space $W$ is said to be {\it defined over a subfield $\F\subset\C$}, if there is a $\F$-vector subspace $W_\F\subset W$, stable under the given action, and such that the canonical map $W_\F\otimes_\F\C\rightarrow W$ is an isomorphism. In this case, we also say that $(\nu,W)$ has an {\it $\F$-structure}.

\subsection{}\label{sect:cusprep}
We let $\Pi$ be an irreducible cuspidal automorphic representation of $G(\A)$ with central character $\omega_\Pi$, cf.\  \cite{bojac} 4.4--4.6. For convenience we will not distinguish between a cuspidal automorphic representation, its smooth automorphic LF-space completion and its (non-smooth) Hilbert space completion in the $L^2$-spectrum. It is of the form
$$\Pi=\tilde\Pi\otimes |\!\det\!|^t,\quad t\in\C,$$
with $\tilde\Pi$ being a unitary cuspidal automorphic representation of $G(\A)$. It decomposes abstractly into a restricted tensor product of local representations $\Pi=\otimes'_v \Pi_v$ of irreducible admissible representations $\Pi_v=\tilde\Pi_v\otimes|\!\det_v\!|^t$ of $G(F_v)$. Collecting the local representations at the archimedean (resp., non-archimedean) places, we obtain an irreducible admissible representation $\Pi_\infty=\otimes_{v\in S_{\infty}}\Pi_v$ of the real Lie group $G_\infty$ (resp., an irreducible admissible representation $\Pi_f=\otimes'_{v \notin S_{\infty}}\Pi_v$ of the totally disconnected, locally compact group $G(\A_f)$). The finite set of places where $\Pi_f$ ramifies is denoted $S_{\Pi_f}$ and we let $S_{\Pi}=S_{\infty}\cup S_{\Pi_f}$. We assume furthermore that there is an id\`ele class character $\eta: F^\times\backslash\A^\times\ra\C^\times$ such that
$$\eta^n=\omega_\Pi.$$
It is hence of the form $\eta=\tilde\eta\otimes |\cdot|^{2t}$, $\tilde\eta$ being unitary. We will write $S_\eta$ for the set of places where $\eta$ ramifies and define $S_{\Pi,\eta}:=S_\Pi\cup S_\eta$. Further, let
$$
\Q(\Pi,\eta):=\Q(\Pi)\Q(\eta),
$$
the compositum of the rationality field $\Q(\Pi)$ of $\Pi$ and the rationality field $\Q(\eta)$ of $\eta$.

\subsection{}\label{sect:psi}
We fix, once and for all, an additive character $\psi_{\Q}$ of ${\mathbb Q} \backslash {\mathbb A}$, as in Tate's thesis, namely,
$\psi_{\Q}(x) = e^{2\pi i \lambda(x)}$ with the $\lambda$ as defined in \cite[Sect.\ 2.2]{tate-thesis}. In particular,
$\lambda = \sum_{p \leq \infty} \lambda_p$, where $\lambda_{\infty}(t) = -t$ for any $t \in \R$ and $\lambda_p(x_p)$ for any $x_p \in \Q_p$
is the rational number with only $p$-power denominator such that $x_p - \lambda_p(x_p) \in \Z_p$. If we write
$\psi_{\Q} =  \psi_{\R} \otimes \otimes_p \psi_{\Q_p}$,
then $\psi_{\R}(t) = e^{-2\pi it}$ and $\psi_{\Q_p}$ is trivial on $\Z_p$ and nontrivial on $p^{-1}\Z_p$.
Next, we define a character $\psi$ of $F\backslash \A$ by composing $\psi_{\Q}$ with the
trace map from $F$ to $\Q$: $\psi = \psi_{\Q} \circ Tr_{F/\Q}$. If $\psi = \otimes_v \psi_v$, then the local characters are determined analogously.
In particular, if $\mathfrak{D}_F = \prod_{\wp} \wp^{r_{\wp}}$ with the product running over
all prime ideals $\wp$, then the conductor of the local character $\psi_{v}$ is $\wp_v^{-r_{\wp}}$, i.e., $\psi_{v}$ is trivial on $\wp_v^{-r_{\wp}}$ and nontrivial on $\wp_v^{-r_{\wp}-1}.$ Let $S_\psi=\{\wp:\wp\nmid\mathfrak D_F\}$ the set of non-archimedean places of $F$, where $\psi$ ramifies.
Note that $\psi_f$ takes values in the subgroup  $\bfgreek{mu}_{\infty}$ of $\C^\times$ consisting of all roots of unity. This comment will be relevant when we deal with rational structures on Shalika models, cf.\ Sect.\ \ref{sect:AutCact}.

\subsection{}\label{sect:measures}
For $v \notin S_\infty$, let $dh_v=d(h_{1,v},h_{2,v})=dg_{1,v}\times dg_{2,v}$ be the unique local Haar measure on $H(F_v)$ for which the volume of each copy of $\GL_n(\O_v)$ equals $1$. Define $dg_{i,f}:= \prod_{v\notin S_\infty} dg_{i,v}$, $i=1,2$, let $dh_f=d(h_{1,f},h_{2,f})=dg_{1,f}\times dg_{2,f}$ be the corresponding measure on $H(\A_f)$. This choice implies that certain volume terms that will appear will be rational numbers. Observe that the volume of $Z_G(F)\backslash Z_G(\A)/\R_+^d$ is already determined by our choice of $dh_f$ made above. Just for this subsection, let $c$ be this volume. Now, at an archimedean place $v\in S_\infty$, let $dg'_{1,v}$ and $dg_{2,v}$ be the local Haar measures that give the respective copy of $SO(n)$ volume 1 and define $dg_{1,\infty}:=c\cdot\prod_{v\in S_\infty}dg'_{1,v}$ and $dg_{2,\infty}:=\prod_{v\in S_\infty}dg_{2,v}$.
This defines global invariant measures $dg_i:=dg_{i,\infty}\cdot dg_{i,f}$, $i=1,2$, on each copy of $\GL_n(\A)$, well as a global invariant measure $dh = d(h_1,h_2)$ on $H(\A)$ by $dh:=dg_1\times dg_2$.

\section{Shalika models and rational structures}
\subsection{Global Shalika models}\label{sect:globalShalikamodels}
We will now define the notion of a {\it Shalika model} of a cuspidal automorphic representation $\Pi$ as in Sect.\ \ref{sect:cusprep}. Let
$$ \S:=\left\{
s =  \left( \!\!\begin{array}{ccc}
h &  0 \\
0 &  h
\end{array}\!\!\right)
\left( \!\!\begin{array}{ccc}
1 &  X\\
0 &  1
\end{array}\!\!\right) \Bigg|
\begin{array}{l}
h\in \GL_n\\
X\in {\rm M}_n
\end{array}\right\}\subset G.$$
It is traditional to call $\S$ the Shalika subgroup of $G$. The characters $\eta$ and $\psi$ can be extended to a character of $\S(\A)$:
$$s=\left(\!\! \begin{array}{ccc}
h &  0\\
0 &  h
\end{array}\!\!\right)\left(\!\! \begin{array}{ccc}
1 &  X\\
0 &  1
\end{array}\!\!\right) \mapsto (\eta \otimes \psi)(s) := \eta(\det(h))\psi(Tr(X)).$$
We will also denote $\eta(s) = \eta(\det(h))$ and $\psi(s) = \psi(Tr(X))$. For a cusp form $\varphi\in\Pi$ and $g\in G(\A)$ consider the integral
$$
\S^\eta_\psi(\varphi)(g):=\int_{Z_{G}(\A)\S(F)\backslash \S(\A)} (\Pi(g)\cdot\varphi)(s)\eta^{-1}(s)\psi^{-1}(s) ds.
$$
It is well-defined by the cuspidality of the function $\varphi$, cf.\  Jacquet--Shalika \cite{jacshal} 8.1, and hence yields a function
$\mathcal{S}^\eta_\psi(\varphi): G(\A)\ra\C$. It satisfies the transformation law
$$
\S^\eta_\psi(\varphi)(sg) = \eta(s)\cdot\psi(s)\cdot \S^\eta_\psi(\varphi)(g),
$$
for all $g\in G(\A)$ and $s\in \S(\A)$ as above. In particular, we obtain an intertwining of $G(\A)$-modules
$$
\Pi\ra \textrm{Ind}_{\S(\A)}^{G(\A)}[\eta\otimes\psi]
$$
given by $\varphi\mapsto \S^\eta_\psi(\varphi),$ which by the irreducibility of $\Pi$ is either trivial or injective. The following theorem, due to Jacquet--Shalika, gives a necessary and sufficient condition for $\S^\eta_\psi$ being non-zero.

\begin{thm}[Jacquet--Shalika, \cite{jacshal} Thm.\,1, p.\,213]\label{thm:JS}
The following assertions are equivalent:
\begin{enumerate}
\item[(i)] There is a $\varphi\in\Pi$ and $g\in G(\A)$ such that $\S^\eta_\psi(\varphi)(g)\neq 0$.
\item[(ii)] $\S^\eta_\psi$ defines an injection of $G(\A)$-modules
$$
\Pi\hookrightarrow \textrm{\emph{Ind}}_{\S(\A)}^{G(\A)}[\eta\otimes\psi].
$$
\item[(iii)] Let $S$ be any finite set of places containing $S_{\Pi,\eta}$.
The twisted partial exterior square $L$-function
$$
L^S(s,\Pi,\wedge^2\otimes\eta^{-1}):=\prod_{v\notin S} L(s,\Pi_v,\wedge^2\otimes\eta^{-1}_v)
$$
has a pole at $s=1$.
\end{enumerate}
\end{thm}
\begin{proof}
This is proved in \cite{jacshal} for unitary representations and its extension to the non-unitary case is easy.
\end{proof}

\begin{defn}
\label{defn:shalika-model}
If $\Pi$ satisfies any one, and hence all, of the equivalent conditions of Thm.\ \ref{thm:JS}, then we say that $\Pi$ {\it has an $(\eta,\psi)$-Shalika model}, and
we call the isomorphic image $\S^\eta_\psi(\Pi)$ of $\Pi$ under $\S^\eta_\psi$ a {\it global $(\eta,\psi)$-Shalika model} of $\Pi$. We will sometimes suppress the choice of the characters $\eta$ and $\psi$ and the fact that we deal with a global representation (i.e., a representation of $G(\A)$) and simply say that
$\Pi$ has a Shalika model.
\end{defn}

\begin{cor}\label{cor:n=1}
Let $\Pi$ be a  cuspidal automorphic representation of $\GL_2(\A)$ with central character $\omega_\Pi$. Then $\Pi$ has a global $(\omega_\Pi,\psi)$-Shalika model.
\end{cor}
\begin{proof}
For $\GL_2$, the Whittaker model and the Shalika model of a representation $\Pi$ coincide.
\end{proof}

The following proposition gives another equivalent condition for $\Pi$ to have a global Shalika model, which puts this notion into a broader context within the theory of automorphic forms and will be of particular importance in Wee Teck Gan's appendix. We will use the functorial transfer from GSpin$_{2n+1}$ to $\GL_{2n}$, established for unitary globally generic cuspidal automorphic representations by Asgari--Shahidi in \cite[Thm.\ 1.1]{asgari-shahidi-generic} in its weak form and finally in \cite[Cor.\ 5.15]{asgari-shahidi} at every place. Its extension to the non-unitary case, which we are going to use, is given as follows: Every cuspidal automorphic representation $\pi$ of ${\rm GSpin}_{2n+1}(\A)$ is of the form $\pi\cong\tilde\pi\otimes|\!\det\!|^{t/n}$ for a unitary cuspidal automorphic representation $\tilde\pi$ and some $t\in\C$. Further, $\pi$ is globally generic if and only if $\tilde\pi$ is. Now, if $\tilde\Pi$ is the Asgari--Shahidi transfer of such a $\tilde\pi$, then we let $\pi$ transfer to the cuspidal automorphic representation $\Pi:=\tilde\Pi\otimes |\!\det\!|^t$ of $\GL_{2n}(\A)$. With this set-up in place we obtain

\begin{prop}\label{prop:selfdual}
Let $\Pi$ be a cuspidal automorphic representation of $G(\A)=\GL_{2n}(\A)$ with central character $\omega_\Pi$. Then the following assertions are equivalent:
\begin{enumerate}
\item[(i)] $\Pi$ has a global $(\eta,\psi)$-Shalika model for some id\`ele class character $\eta$ satisfying $\eta^n=\omega_\Pi$.
\item[(ii)] $\Pi$ is the transfer of a globally generic cuspidal automorphic representation $\pi$ of ${\rm GSpin}_{2n+1}(\A)$.
\end{enumerate}
In particular, if any of the above equivalent conditions is satisfied, then $\Pi$ is essentially self-dual. The character $\eta$ may be taken to be the central character
$\omega_\pi$ of $\pi$.
\end{prop}
\begin{proof}
Let $\Pi$ be a cuspidal automorphic representation of $G(\A)$ with central character $\omega_\Pi$.
By Thm.\ \ref{thm:JS}, $\Pi$ has an $(\eta,\psi)$-Shalika model for some id\`ele class character $\eta$ satisfying $\eta^n=\omega_\Pi$, if and only if the partial exterior square $L$-function $L^S(s,\Pi,\wedge^2\otimes\eta^{-1})$ has a pole at $s=1$. (Here $S$ is any finite set of places containing $S_{\Pi,\eta}$.) Furthermore, a functorial transfer of a globally generic, cuspidal automorphic representation of ${\rm GSpin}_{2n+1}(\A)$ is essentially self-dual by \cite{asgari-shahidi}, Cor.\ 5.15. Observing that $L^S(s,\Pi,\wedge^2\otimes\eta^{-1})=L^S(s,\tilde\Pi,\wedge^2\otimes\tilde\eta^{-1})$, the equivalence of (i) and (ii) follows now from Hundley--Sayag \cite{hundleysayag}, Thm.\ A, together with Asgari--Shahidi, \cite{asgari-shahidi}, Thm.\ 5.10.(b).
\end{proof}

The following proposition is crucial for much that will follow. It relates the period-integral over $H$ of a cusp form $\varphi$ of $G$ to a certain zeta-integral of the
function $\S^\eta_\psi(\varphi)$ in the Shalika model corresponding to $\varphi$ over one copy of $\GL_n$.

\begin{prop}[Friedberg--Jacquet, \cite{friedjac} Prop.\ 2.3]\label{prop:FJ}
Let $\Pi$ have an $(\eta,\psi)$-Shalika model. For a cusp form $\varphi\in\Pi$, consider the integral
$$
\Psi(s,\varphi):=\int_{Z_G(\A)H(F)\backslash H(\A)} \varphi\left(\!\!\left(\!\! \begin{array}{ccc}
h_1 &  0\\
0 &  h_2
\end{array}\!\!\right)\!\!\right)\Bigg|\frac{\det(h_1)}{\det(h_2)}\Bigg|^{s-1/2}\eta^{-1}(\det(h_2))\, d(h_1,h_2).
$$
Then, $\Psi(s,\varphi)$ converges absolutely for all $s\in\C$. Next, consider the integral
$$
\zeta(s,\varphi):=\int_{\GL_n(\A)} \S^\eta_\psi(\varphi)\left(\!\!\left(\!\! \begin{array}{ccc}
g_1 &  0\\
0 &  1
\end{array}\!\!\right)\!\!\right) |\det(g_1)|^{s-1/2} \, dg_1.
$$
Then, $\zeta(s,\varphi)$ is absolutely convergent for $\Re(s)\gg 0$.
Further, for $\Re(s)\gg 0$, we have
$$
\zeta(s,\varphi)=\Psi(s,\varphi),
$$
which provides an analytic continuation of $\zeta(s,\varphi)$ by setting $\zeta(s,\varphi)=\Psi(s,\varphi)$ for all $s\in\C$.
\end{prop}
\begin{proof}
This is proved in \cite{friedjac} for unitary representations and its extension to the non-unitary case is easy.
\end{proof}

\begin{rem}
The results quoted in this section are valid for any number field $F$, however, we will need them only for a totally real $F$ since
our main results will crucially depend on $F$ being totally real.
\end{rem}

\subsection{Local Shalika models}
Consider a cuspidal automorphic representation $\Pi=\otimes'_v \Pi_v$ of $G(\A)$ as in Sect.\ \ref{sect:cusprep}.

\begin{defn}
For any place $v$ we say that $\Pi_v$ has a {\it local $(\eta_v,\psi_v)$-Shalika model} if there is a non-trivial (and hence injective) intertwining
$$
\Pi_v \hookrightarrow \textrm{Ind}_{\S(F_v)}^{G(F_v)}[\eta_v\otimes\psi_v].
$$
\end{defn}

If $\Pi$ has a global Shalika model, then $\S^\eta_\psi$ defines local Shalika models at every place. The corresponding local intertwining operators are denoted by $\S^{\eta_v}_{\psi_v}$ and their images by $\S^{\eta_v}_{\psi_v}(\Pi_v)$, whence $\S^\eta_\psi(\Pi)=\otimes'_v \S^{\eta_v}_{\psi_v}(\Pi_v)$.
We can now consider cusp forms $\varphi$ such that the function $\xi_\varphi=\S^\eta_\psi(\varphi)\in \S^\eta_\psi(\Pi)$ is factorizable as
$$
\xi_\varphi=\otimes'_v \xi_{\varphi_v},
$$
where
$$
\xi_{\varphi_v}\in \S^{\eta_v}_{\psi_v}(\Pi_v)\subset \textrm{Ind}_{\S(F_v)}^{G(F_v)}[\eta_v\otimes\psi_v].
$$
Prop.\ \ref{prop:FJ} implies that
$$
\zeta_v(s,\xi_{\varphi_v}):=\int_{\GL_n(F_v)} \xi_{\varphi_v}\left(\!\!\left(\!\! \begin{array}{ccc}
g_{1,v} &  0\\
0 &  1_v
\end{array}\!\!\right)\!\!\right) |\det(g_{1,v})|^{s-1/2} dg_{1,v}
$$
is absolutely convergent for $\Re(s)$ sufficiently large. The same remark applies to
$$\zeta_f(s,\xi_{\varphi_f}):=\int_{\GL_n(\A_f)} \xi_{\varphi_f}\left(\!\!\left(\!\! \begin{array}{ccc}
g_{1,f} &  0\\
0 &  1_f
\end{array}\!\!\right)\!\!\right) |\det(g_{1,f})|^{s-1/2} dg_{1,f}=\prod_{v\notin S_\infty} \zeta_v(s,\xi_{\varphi_v}).$$

\subsection{}
The next proposition, also due to Friedberg--Jacquet, relates the ``Shalika-zeta-integral'' with the standard $L$-function of $\Pi$.

\begin{prop}[\cite{friedjac}, Prop.\ 3.1 \& 3.2]\label{prop:FJL-fct}
Assume that $\Pi$ has an $(\eta,\psi)$-Shalika model. Then for each place $v$ and $\xi_{\varphi_v}\in \S^{\eta_v}_{\psi_v}(\Pi_v)$ there is a holomorphic function $P(s,\xi_{\varphi_v})$ such that
$$\zeta_v(s,\xi_{\varphi_v})=L(s,\Pi_v)P(s,\xi_{\varphi_v}).$$
One may hence analytically continue $\zeta_v(s,\xi_{\varphi_v})$ by re-defining it to be $L(s,\Pi_v)P(s,\xi_{\varphi_v})$ for all $s\in\C$. Moreover, for every $s\in\C$ there exists a vector $\xi_{\varphi_v}\in \S^{\eta_v}_{\psi_v}(\Pi_v)$ such that $P(s,\xi_{\varphi_v})=1$. If $v\notin S_{\Pi,\psi}:=S_\Pi\cup S_\psi$, then this vector can be taken to be the spherical vector $\xi_{\Pi_v}\in \S^{\eta_v}_{\psi_v}(\Pi_v)$ normalized by the condition
$$\xi_{\Pi_v}(id_v)=1.$$
\end{prop}
\begin{proof}
This is proved in \cite{friedjac} for unitary representations and its extension to the non-unitary case is easy.
\end{proof}

Prop.\ \ref{prop:FJL-fct} relates the Shalika-zeta-integral to $L$-functions, on the other hand Prop.\ \ref{prop:FJ} relates this integral to a period integral over $H$. As we shall soon see, the period integral over $H$ admits a cohomological interpretation, provided the cuspidal representation $\Pi$ is of cohomological type.

\subsection{An interlude: cohomological cuspidal automorphic representations}\label{sect:cohreps}
We assume from now on that the cuspidal automorphic representation $\Pi$ of $G(\A)$ as defined in Sect.\ \ref{sect:cusprep} is cohomological. By Sect.\ \ref{sect:gKcoh}, this means that there is a  highest weight representation $E_\mu=\otimes_{v\in S_{\infty}} E_{\mu_v}$ of $G_\infty$ such that
$$H^q(\g_\infty,K^\circ_\infty,\Pi\otimes E^{\sf v}_\mu)=H^q(\g_\infty,K^\circ_\infty,\Pi_\infty\otimes E^{\sf v}_\mu)\otimes\Pi_f\neq 0$$
for some degree $q$. Such a highest weight module $E_\mu$ is necessarily essentially self-dual, but even more is true due to the fact that $\Pi$ is cuspidal. Indeed, according to Clozel \cite{clozel}, Lem.\ 4.9, there is a ${\sf w}\in\Z$ such that for each archimedean place $v\in S_{\infty}$ the highest weight $\mu_v=(\mu_{v,1},...,\mu_{v,2n})$ satisfies
$$\mu_{v,i}+\mu_{v,2n-i+1}={\sf w}, \quad\quad 1\leq i\leq n.$$
In other words, $E_{\mu_v}$ differs from its dual by the same integer power $|\!\det\!|^{w_v}=|\!\det\!|^{\sf w}$ of the determinant at each archimedean place $v\in S_\infty$, cf.\  Sect.\ \ref{sect:highweights}. This integer is called the ``purity weight'' of $\Pi$ and if we write $\Pi=\tilde\Pi\otimes|\!\det\!|^t$ as in Sect.\ \ref{sect:cusprep}, then $t=\tfrac{{\sf w}}{2}$.

As $\Pi$ is generic, the archimedean local component $\Pi_\infty$ must be essentially tempered. More precisely, for each archimedean place $v\in S_{\infty}$ let
$$
\ell_{v,i}:=\mu_{v,i}-\mu_{v,2n-i+1}+2(n-i)+1=2\mu_{v,i}+2(n-i)+1-{\sf w}, \quad 1\leq i\leq n,
$$
so $\ell_{v,1}>\ell_{v,2}>...>\ell_{v,n}\geq 1$. Moreover, let $P$ be the parabolic subgroup of $G$ with Levi factor $L=\prod_{i=1}^n \GL_2$ and $D(\ell_v)$ the discrete series representations of $\GL_2(\R)$ of lowest (non-negative) $\Orth(2)$-type $\ell_v+1$. Then one can show, under these assumptions, that

\begin{equation}\label{eq:piinfty}
\Pi_v\cong\textrm{Ind}^{G(\R)}_{P(\R)}[D(\ell_{v,1})|\!\det\!|^{{\sf w}/2}\otimes...\otimes D(\ell_{v,n})|\!\det\!|^{{\sf w}/2}],\quad\quad\forall v\in S_{\infty},
\end{equation}
and so
\begin{eqnarray*}
H^q(\g_\infty,K^\circ_\infty,\Pi_\infty\otimes E^{\sf v}_\mu) & \cong & \bigoplus_{\sum_v q_v=q}\bigotimes_{v\in S_{\infty}} H^{q_v}(\mathfrak{gl}_{2n}(\R),\SO(2n)\R_+,\Pi_v\otimes E^{\sf v}_{\mu_v})\\
& \cong & \bigoplus_{\sum_v q_v=q}\bigotimes_{v\in S_{\infty}} \C^2\otimes\bigwedge^{q_v-n^2}\C^{n-1}.
\end{eqnarray*}
The group $K_\infty/K^\circ_\infty \cong (\Z/2\Z)^d$ acts on this cohomology space. For any character $\epsilon$ of
$K_\infty/K^\circ_\infty$ which we write as:
$$
\epsilon = (\epsilon_1,\dots,\epsilon_d) \in  (\Z/2\Z)^d \cong (K_\infty/K^\circ_\infty)^*,
$$
one obtains a corresponding eigenspace
$$
H^q(\g_\infty,K^\circ_\infty,\Pi_\infty\otimes E^{\sf v}_\mu)[\epsilon]\cong\bigoplus_{\sum_v q_v=q}\bigotimes_{v\in S_{\infty}}\bigwedge^{q_v-n^2}\C^{n-1}.
$$
(As a general reference for the above see Clozel \cite[3]{clozel}. See also \cite[3.1.2]{mahnk} or \cite[5.5]{grobrag}.)
In particular, for any cohomological cuspidal automorphic representation $\Pi$ of $G(\A)$, the corresponding $\epsilon$-eigenspace in the $(\g_\infty,K^\circ_\infty)$-cohomology of $\Pi_\infty\otimes E^{\sf v}_\mu$ in the top degree, i.e., in degree
\begin{equation}\label{eqn:top-degree}
q_0:=d(n^2+n-1)
\end{equation}
is one-dimensional. Observe that $q_0$ only depends on $n$ and the degree $d$ of $F/\Q$.

The next proposition explains the behaviour of cuspidal automorphic representations under $\sigma$-twisting.

\begin{prop}[Clozel \cite{clozel}, Thm.\ 3.13 ]\label{prop:regalg}
Let $\Pi$ be a cuspidal automorphic representation as in Sect.\ \ref{sect:cusprep}. If $\Pi$ is cohomological with respect to a highest weight representation $E^{\sf v}_\mu$, then the $\sigma$-twisted representation ${}^\sigma\Pi$ is also cuspidal automorphic. It is cohomological with respect to ${}^\sigma\! E^{\sf v}_\mu$. Moreover, for every $\epsilon\in (K_\infty/K^\circ_\infty)^*$, the $G(\A_f)$-module $H^{q_0}(\g_\infty,K^\circ_\infty,\Pi\otimes E^{\sf v}_\mu)[\epsilon]$ is defined over the rationality field $\Q(\Pi)$, which in this case is a number field. The same holds for $\Q(\Pi)$ being replaced by $\Q(\Pi,\eta)$.
\end{prop}

\begin{rem}\label{rem:clozel}
This is shown in Clozel \cite{clozel} for regular algebraic cuspidal automorphic representations, cf.\  \cite[3.5]{clozel} for this notion. However, a cuspidal automorphic representation $\Pi$ as in Sect.\ \ref{sect:cusprep} is cohomological if and only if it is regular algebraic, whence we obtain the proposition in the above form. The last assertion on cohomology being defined over $\Q(\Pi)$ and hence also over $\Q(\Pi,\eta)$ -- although well-known to the experts -- is only implicitly proved in \cite{clozel}. For an actual proof one may consider \cite[Thm.\ 8.6]{grobrag}. Here, we note that $\Pi_\infty$ being cohomological forces its central character to be of the form $\omega_{\Pi_\infty}\cong\otimes_{v\in S_\infty}\sgn^{a_v} |\cdot|^{n\cdot{\sf w}}$, for some $a_v\in\{0,1\}$, and so $\eta$ has to be algebraic, whence $\Q(\eta)$ is a number field, too. See also Sect.\ \ref{sect:characters}. The rational structure on cohomology has a purely geometric origin and it is inherited by a $\Q(E^{\sf v}_\mu)$-structure on cuspidal, or better, inner cohomology.
\end{rem}

\subsection{Shalika model versus cohomology}
Let $\Pi$ be a cuspidal automorphic representation of $G(\A)$ as in Sect.\ \ref{sect:cusprep}. On the one hand we can impose the condition that $\Pi$ has a Shalika model,
and on the other hand we can ask for $\Pi$ to be cohomological. Let us observe that these conditions are independent of each other by presenting some examples. An example of {\it Shalika model but not cohomological:}
Take $n=1$. According to Cor.\ \ref{cor:n=1}, any  cuspidal automorphic representation $\Pi$ of $G(\A)=\GL_2(\A)$ has an $(\omega_{\Pi}, \psi)$-Shalika model. But if $\Pi$ is constructed from a Maass-form, then $\Pi_\infty$ cannot be cohomological. An example of {\it cohomological but no Shalika model for $\GL_4$:}
Let $n=2$. Let $\pi_k$ and $\pi_{k'}$ be cuspidal automorphic representations of $\GL_2(\A)$, $F=\Q$, attached to primitive modular cusp forms of weights $k$ and $k'$, respectively. Assume that $k\neq k'$ and both numbers are even. Assume moreover that both modular forms are not of CM-type, i.e., Sym$^2(\pi_k)$ and Sym$^2(\pi_{k'})$, cf.\  Sect.\ \ref{sect:symmtrans}, are cuspidal. Put $\Pi:=(\pi_k\boxtimes\pi_{k'})|\cdot|^{1/2}$. Then $\Pi$ is regular algebraic, i.e., cohomological, cf.\  Rem.\  \ref{rem:clozel}. Furthermore, $\Pi$ is cuspidal by Ramakrishnan \cite[Thm.\ M]{rama}. Further, one may check that the exterior square of $\Pi$ decomposes as an isobaric direct sum
$$\wedge^2\Pi \cong \left({\rm Sym}^2(\pi_k) \otimes\omega_{\pi_{k'}}\right)|\cdot| \boxplus \left({\rm Sym}^2(\pi_{k'}) \otimes\omega_{\pi_{k}}\right)|\cdot|.$$
See, for example, Asgari--Raghuram \cite{asgari-raghuram}, Prop.\ 3.1. Hence, $\wedge^2\Pi$ has no one-dimensional isobaric summand, and by Thm.\ \ref{thm:JS}, $\Pi$ cannot have a Shalika model for any $\eta$.
Another example of {\it cohomological but no Shalika model} but now for $\GL_6$:
Let $n=3$. In \cite[Thm.\ 5.1]{ramwong} Ramakrishnan and Wang gave an example of a cohomological unitary cuspidal automorphic representation $\Pi$ of $G(\A)=\GL_6(\A)$, $F=\Q$, which is not essentially self-dual. Hence, Prop.\ \ref{prop:selfdual} implies that $\Pi$ does not admit an $(\eta,\psi)$-Shalika model for any
$\eta$.

\subsection{Shalika models and $\sigma$-twisting}\label{sect:shalikasigma}
Having observed in the last subsection that having a Shalika model is independent of having non-zero $(\g_\infty,K_\infty^\circ)$-cohomology with respect to some finite-dimensional, algebraic coefficient system, one can still ask the question if having a Shalika-model is an invariant under $\sigma$-twisting. In other words, we may ask, if having a Shalika model is an {\it arithmetic} property of a cohomological cuspidal automorphic representation. We begin with a useful observation about the character $\eta$:

\begin{lem}\label{lem:eta}
Let $\Pi$ be a cohomological cuspidal automorphic representation of $G(\A)$ which admits an $(\eta,\psi)$-Shalika model. Then, for all $v \in S_\infty$ we have
$$
\eta_v \ = \ {\rm sgn}^{\sf w} |\ |^{\sf w}.
$$
\end{lem}
\begin{proof}
See the proof of Thm.\ 5.3 in Gan--Raghuram~\cite{gan-raghuram}.
\end{proof}

By Prop.\ \ref{prop:regalg}, ${}^\sigma\Pi$ is also a cuspidal automorphic representation of $G(\A)$ for all $\sigma\in {\rm Aut}(\C)$. It makes sense to ask if $\Pi$ having an $(\eta,\psi)$-Shalika model implies that ${}^\sigma\Pi$ has a $({}^\sigma\!\eta,\psi)$-Shalika model for all $\sigma\in {\rm Aut}(\C)$. For $n=1$ this is obvious in view of Cor.\ \ref{cor:n=1}. If $n\geq 2$, the situation is more complicated, nevertheless, we have the following theorem.

\begin{thm}\label{thm:arithmeticShalika}
Let $\Pi$ be a cohomological cuspidal automorphic representation of $G(\A)$ which admits an $(\eta,\psi)$-Shalika model. Then ${}^\sigma\Pi$ has a $({}^\sigma\!\eta,\psi)$-Shalika model for all $\sigma\in {\rm Aut}(\C)$.
\end{thm}

\begin{proof}
See the appendix for Wee Teck Gan's proof of this theorem; this is elaborated further in Thm.\ 5.3 in Gan--Raghuram~\cite{gan-raghuram}.
\end{proof}

\subsection{An action of Aut$(\C)$ on Shalika functions}\label{sect:AutCact}
Henceforth, we take $\Pi$ to be a cohomological cuspidal automorphic representation of $G(\A)$ which has an $(\eta,\psi)$-Shalika model $\S^{\eta}_\psi(\Pi)$.

Our goal is to define a $\Q(\Pi,\eta)$-structure $\S^{\eta_f}_{\psi_f}(\Pi_f)_{\Q(\Pi,\eta)}$ on the Shalika model of $\Pi_f$.
The main ingredient towards this will be a certain ``twisted action'' of Aut$(\C)$ on $\textrm{Ind}_{\S(\A_f)}^{G(\A_f)}[\eta_f \otimes \psi_f]$, which we shall now define. Recall that $\psi_f$, being the finite part of a unitary additive character, takes values in $\bfgreek{mu}_{\infty} \subset \C^\times$, cf.\  Sect.\ \ref{sect:psi}. This suggests that we consider the cyclotomic character:
$$
\begin{array}{ccccccc}
{\rm Aut}(\C)& \longrightarrow & {\rm Gal}({\Q}(\bfgreek{mu}_{\infty})/{\Q}) & \longrightarrow & \widehat{{\Z}}^\times \cong \prod_p \Z_p^\times & \hookrightarrow &
\prod_{p} \prod_{v|p} \O_v^\times, \\
\sigma  & \longmapsto & \sigma |_{{\Q}(\bfgreek{mu}_{\infty})}& \longmapsto & t_{\sigma} & \longmapsto & t_{\sigma}
\end{array}
$$
where the last inclusion is the one induced by the diagonal embedding of
${\Z}_p$ into $\prod_{v|p} \O_v$. The element $t_{\sigma}$ at the end may hence
be thought of as an element of ${\A}_f^\times$.
Let $\t^{-1}_\sigma$ denote the diagonal matrix
$$\t^{-1}_\sigma:={\rm diag}(\underbrace{t_{\sigma}^{-1},...,t_{\sigma}^{-1}}_{n}, \underbrace{1,...,1}_{n}),$$ regarded as an
element of $G(\A_f)$. For $\sigma \in {\rm Aut}({\C})$ and $\xi\in \S^{\eta_f}_{\psi_f}(\Pi_f)$, we define the function ${}^{\sigma}\xi$ by
\begin{equation}
\label{eqn:aut-c-action}
{}^{\sigma}\xi(g_f) = \sigma(\xi(\t^{-1}_\sigma\cdot g_f)),
\end{equation}
$g_f \in G({\A}_f)$. Note that this action makes sense locally, by replacing
$\t^{-1}_\sigma$ by $\t^{-1}_{\sigma,v}$. We see that $\xi\mapsto {}^{\sigma}\xi$ is a $\sigma$-linear $G({\A}_f)$-equivariant
isomorphism $\tilde\sigma :  \textrm{Ind}_{\S(\A_f)}^{G(\A_f)}[\eta_f \otimes \psi_f] \to \textrm{Ind}_{\S(\A_f)}^{G(\A_f)}[{}^{\sigma}\eta_f \otimes \psi_f]$
and the same holds locally at any finite place $v$.

\subsection{A certain $\Q(\Pi,\eta)$-structure on $\S^{\eta_f}_{\psi_f}(\Pi_f)$}

\begin{lem}\label{lem:rational-shalika}
Let $\Pi$ be a cohomological cuspidal automorphic representations of $G(\A)$ which has an $(\eta,\psi)$-Shalika model $\S^{\eta}_\psi(\Pi)$. Then, $\tilde\sigma\left(\S^{\eta_f}_{\psi_f}(\Pi_f)\right)=\S^{{}^\sigma\!\eta_f}_{\psi_f}({}^\sigma\Pi_f)$ for all $\sigma\in {\rm Aut}(\C)$. For any finite extension $\F/{\Q}(\Pi,\eta)$ we have an $\F$-structure on $\S^{\eta_f}_{\psi_f}(\Pi_f)$ by taking invariants:
$$
\S^{\eta_f}_{\psi_f}(\Pi_f)_\F: = \S^{\eta_f}_{\psi_f}(\Pi_f)^{{\rm Aut}({\C}/\F)}.
$$
\end{lem}

\begin{proof}
Using Thm.\ \ref{thm:arithmeticShalika}, the first part of the lemma may be proved in analogy to the case of Whittaker models. See Harder \cite[p.80]{hardermodsym}, Mahnkopf \cite[p.594]{mahnk} or Raghuram--Shahidi \cite[Lem.\ 3.2]{raghuram-shahidi-imrn}. See also Remark \ref{rem:unique} below.

In order to prove the remaining assertions of the lemma, consider the vector $\xi_{\Pi_f} = \otimes_{v\notin S_\infty}\xi_{\Pi_v}$, where $\xi_{\Pi_v}$ is a new vector (called ``essential vector'' in Jacquet--Piatetski-Shapiro--Shalika \cite{jac-ps-shalika-mathann}.) That means that $\xi_{\Pi_v}$ is right invariant by a suitable open compact subgroup of $G(F_v)$ which gives rise to a one-dimensional space of invariant vectors. At each finite place $v$, the vector $\xi_{\Pi_v}$ is unique up to scalars and if $\eta_v$ is unramified at all $v\in S_{\Pi_f}$ (in particular, if $\eta_v=\triv_v$ at the finitely many, non-archimedean places, where $\Pi$ ramifies) we may fix a choice of $\xi_{\Pi_v}$ by assuming that $\xi_{\Pi_v}(id_v)=1$, see \cite{note}, Cor.\ 8. In this case, we obtain by an easy calculation using (\ref{eqn:aut-c-action}) that
$$
{}^\sigma\xi_{\Pi_v}=\xi_{{}^\sigma\Pi_v}.
$$
In particular, any $\sigma \in {\rm Aut}({\C}/\Q(\Pi,\eta))$ fixes $\xi_{\Pi_v}$, and hence fixes the global new vector $\xi_{\Pi_f}$. As a consequence, the well-known strategy used for establishing rational structures on Whittaker models, \cite{hardermodsym, mahnk, raghuram-shahidi-imrn}, carries over the our situation: For this, let $\S^{\eta_f}_{\psi_f}(\Pi_f)_{\Q(\Pi,\eta)}$ be the $\Q(\Pi,\eta)$-span of the $G({\A}_f)$-orbit of $\xi_{\Pi_f}$. Then the canonical map
$$
\S^{\eta_f}_{\psi_f}(\Pi_f)_{\Q(\Pi,\eta)} \otimes_{\Q(\Pi,\eta)} {\C} \to \S^{\eta_f}_{\psi_f}(\Pi_f)
$$
is an isomorphism. Indeed, as $\S^{\eta_f}_{\psi_f}(\Pi_f)_{\Q(\Pi,\eta)}\neq 0$, surjectivity follows from the irreducibility of $\S^{\eta_f}_{\psi_f}(\Pi_f)$, and injectivity follows exactly as in the proof of \cite[Lemme I.1.1]{waldsp}. The action of ${\rm Aut}({\C}/\Q(\Pi,\eta))$ on $\S^{\eta_f}_{\psi_f}(\Pi_f)$ may then be identified with the action of ${\rm Aut}({\C}/\Q(\Pi,\eta))$ on $\S^{\eta_f}_{\psi_f}(\Pi_f)_{\Q(\Pi,\eta)} \otimes_{\Q(\Pi,\eta)} {\C}$, where it acts on the second factor. We deduce that
$$
\S^{\eta_f}_{\psi_f}(\Pi_f)_{\Q(\Pi,\eta)} = \S^{\eta_f}_{\psi_f}(\Pi_f)^{{\rm Aut}({\C}/\Q(\Pi,\eta))}.
$$
Now, if $\F$ is a finite extension of $\Q(\Pi,\eta)$ then, in the above isomorphism, one can identify
$$
\S^{\eta_f}_{\psi_f}(\Pi_f)_\F = \S^{\eta_f}_{\psi_f}(\Pi_f)_{\Q(\Pi,\eta)} \otimes_{\Q(\Pi,\eta)} \F
$$
with
$\S^{\eta_f}_{\psi_f}(\Pi_f)^{{\rm Aut}({\C}/\F)}$. This proves the lemma for unramified $\eta_v$, $v\in S_{\Pi_f}$. If there is a place $v\in S_{\Pi_f}$, where $\eta$ is ramified, too, then the well-known strategy from the Whittaker case does apply any more. Still, the lemma holds true by Hilbert's Thm.\ 90, as it is explained for instance in \cite{gro-ron}, Sect.\ 3.3.1.
\end{proof}

\begin{rem}\label{rem:unique}
For the first assertion of Lem.\ \ref{lem:rational-shalika} we used uniqueness of local, non-archimedean Shalika models. This has been established by Jacquet--Rallis \cite{jacquet-rallis} and Nien \cite{nien} for trivial $\eta$ and in general by Chen--Sun in \cite{chen_sun}, Thm.\ A.
\end{rem}

\subsection{A specific choice of a compatible vector in the Shalika model}
Assume that $\Pi$ is a cohomological cuspidal automorphic representations of $G(\A)$ which admits an $(\eta,\psi)$-Shalika model. Let $S_{\Pi,\psi}=S_\Pi\cup S_\psi$ as in Prop.\ \ref{prop:FJL-fct}, resp.\ $S_{\Pi_f,\psi}:=S_{\Pi_f}\cup S_\psi$. We shall now fix once and for all a particular vector
$$
\xi^\circ_{\Pi_f} \ = \
\otimes'_{v\notin S_{\infty}}  \xi^\circ_{\Pi_v} \ \in \
\S^{\eta_f}_{\psi_f}(\Pi_f),
$$
which transforms compatibly under Aut$(\C)$, i.e., ${}^\sigma\xi^\circ_{\Pi_f}=\xi^\circ_{{}^\sigma\Pi_f}$ and which has the following properties:
\begin{enumerate}
\item $\zeta_v(\frac12,\xi^\circ_{\Pi_v})=L(\frac12,\Pi_v)$ for all $v\notin S_{\Pi,\psi}$,
\item $\zeta_v(\frac12,\xi^\circ_{\Pi_v})=1$ for all $v \in S_{\Pi_f,\psi}$.
\end{enumerate}
\noindent
We divide our discussion into two parts.

\subsubsection{}
\label{sec:unramified}
First, we consider the unramified case. So, let $v\notin S_{\Pi,\psi}$.
According to Prop.\ \ref{prop:FJL-fct}, the normalized spherical vector $\xi_{\Pi_v}$ has the property $\zeta_v(\frac12,\xi_{\Pi_v})=L(\frac12,\Pi_v)$. Furthermore, as one easily sees, ${}^\sigma\xi_{\Pi_v}=\xi_{{}^\sigma \Pi_v}$ for every $\sigma\in {\rm Aut}(\C)$. Therefore, we let
$$
\xi^\circ_{\Pi_v} \ := \
\xi_{\Pi_v}  \ = \ \mbox{normalized spherical vector},
$$
for $v\notin S_{\Pi,\psi}$.

\subsubsection{}
\label{sec:ramified}
Next, consider $v\in S_{\Pi_f,\psi}$. Now the situation is slightly more complicated. Firstly, we observe that by virtue of Prop.\ \ref{prop:FJL-fct} and the non-vanishing of the local $L$-function $L(s,\Pi_v)$ at $s=\tfrac12$, there exists a vector $\xi^\circ_{\Pi_v}\in \S^{\eta_v}_{\psi_v}(\Pi_v)$ such that $\zeta_v(\tfrac12,\xi^\circ_{\Pi_v})=1$. We fix such a local Shalika-functional and put $\xi^\circ_{{}^\sigma\Pi_v}:={}^\sigma\xi^\circ_{\Pi_v}$ for $\sigma\in $Aut$(\C)$. Observe that $\xi^\circ_{{}^\sigma\Pi_v}\in \S^{{}^\sigma\!\eta_v}_{\psi_v}({}^\sigma\Pi_v)$ by Lem.\ \ref{lem:rational-shalika}. Hence, for our purpose it is enough to show 

\begin{lem}\label{lem:rationality-newvector}
Let $v\in S_{\Pi_f,\psi}$. With the above notation, $\zeta_v(\tfrac12, \xi^\circ_{{}^\sigma\Pi_v})=1.$
\end{lem}
\begin{proof}
Let $\Re(s)\gg 0$ and denote by $q$ the order of the residue field $\O_v/\wp_v$. Then, for any $\xi_v\in \S^{\eta_v}_{\psi_v}(\Pi_v)$
$$
\zeta_v(s+\tfrac12,\xi_v)=\int_{\GL_n(F_v)} \xi_v\left(\!\!\left(\!\! \begin{array}{ccc}
g_{1,v} &  0\\
0 &  1_v
\end{array}\!\!\right)\!\!\right) |\det(g_{1,v})|^{s} dg_{1,v}
$$
is a formal Laurent series in the variable $q^{-s}$, its $k$-th coefficients being
$$c_k^{\eta_v, \psi_v}(\xi_v):=\int_{\substack{g_{1,v}\in\GL_n(F_v)\\ |\det(g_{1,v})|=q^{-k}}}\xi_v\left(\!\!\left(\!\! \begin{array}{ccc}
g_{1,v} &  0\\
0 &  1_v
\end{array}\!\!\right)\!\!\right) dg_{1,v}.$$
The latter integral vanishes for $k\ll 0$ and becomes a finite sum, because of the support of Shalika functionals, cf.\ \cite{friedjac}, p.\ 111. Hence, by a change of variables in the integral, we obtain 
\begin{equation}\label{eq:c}
\sigma(c_k^{\eta_v, \psi_v}(\xi_v))=c_k^{{}^\sigma\!\eta_v, \psi_v}({}^\sigma\xi_v).
\end{equation} 
The linear span of the zeta-integrals $\zeta_v(s+\tfrac12,\xi_v)$, $\xi_v\in \S^{\eta_v}_{\psi_v}(\Pi_v)$, is a fractional ideal $\mathcal I^{\eta_v, \psi_v}(\Pi_v)$ of $\C[q^{-s},q^s]$, cf.\ \cite{erez_mao}, p.\ 27. Hence, \eqref{eq:c} implies that 
\begin{equation}\label{eq:d}
\mathcal I^{\eta_v, \psi_v}(\Pi_v)^\sigma = \mathcal I^{{}^\sigma\!\eta_v, \psi_v}({}^\sigma\Pi_v),
\end{equation}
where we let $\sigma\in $ Aut$(\C)$ act on Laurent polynomials by applying it to the coefficients. Recall the holomorphic function $P(s,\xi_v)$ from Prop.\ \ref{prop:FJL-fct}. It is a consequence of \cite{erez_mao}, Cor.\ 5.2, that $P(s+\tfrac12,\xi_v)$ is in fact a Laurent {\it polynomial} in $q^{\pm s}$, i.e., an element of $\C[q^{-s},q^s]$. So, we may (without any issues of convergence) simply insert $s=0$ into the two functions $\sigma(P(s+\tfrac12,\xi_v))$ and $P(s+\tfrac12,\xi_v)^\sigma$ and obtain 
\begin{equation}\label{eq:e}
\sigma(P(\tfrac12,\xi_v)) = P(\tfrac12,\xi_v)^\sigma.
\end{equation}
Specifying $\xi_v=\xi^\circ_{\Pi_v}$ and invoking analytic continuation
\begin{eqnarray*}
\zeta_v(\tfrac12, \xi^\circ_{{}^\sigma\Pi_v}) & = & L(\tfrac12,{}^\sigma\Pi_v) P(\tfrac12, \xi^\circ_{{}^\sigma\Pi_v})\\
& \underset{\textrm{\cite{clozel} Lem.\ 4.6}}{=} & \sigma(L(\tfrac12,\Pi_v)) P(\tfrac12, {}^\sigma\xi^\circ_{\Pi_v}) \\
& \underset{\eqref{eq:d}}{=} & \sigma(L(\tfrac12,\Pi_v)) P(\tfrac12, \xi^\circ_{\Pi_v})^\sigma \\
& \underset{\eqref{eq:e}}{=}  & \sigma(L(\tfrac12,\Pi_v)) \sigma( P(\tfrac12, \xi^\circ_{\Pi_v})) \\
& = & \sigma(L(\tfrac12,\Pi_v) P(\tfrac12, \xi^\circ_{\Pi_v})) \\
& = & \sigma(\zeta_v(\tfrac12, \xi^\circ_{\Pi_v})) \\
& = & \sigma(1) \\
& = & 1.
\end{eqnarray*}
\end{proof}

\subsubsection{In summary}\label{sect:vector}  Putting \ref{sec:unramified} and \ref{sec:ramified} together, we have a very special vector
$$\xi^\circ_{\Pi_f}=\otimes'_{v \notin S_{\infty}}\xi^\circ_{\Pi_v}\in \S^{\eta_f}_{\psi_f}(\Pi_f),$$
which satisfies
\begin{enumerate}
\item $\zeta_v(\frac12,\xi^\circ_{\Pi_v})=L(\frac12,\Pi_v)$ for all $v\notin S_{\Pi,\psi}$,
\item $\zeta_v(\frac12,\xi^\circ_{\Pi_v})=1$ for all $v\in S_{\Pi_f,\psi},$
\end{enumerate}
\noindent and
$${}^\sigma\xi^\circ_{\Pi_f} = \otimes'_{v \notin S_{\infty}}{}^\sigma\xi^\circ_{\Pi_v} =
\otimes'_{v \notin S_{\infty}}\xi^\circ_{{}^\sigma\Pi_v} = \xi^\circ_{{}^\sigma\Pi_f},$$
for all $\sigma\in\textrm{Aut}(\C)$.

\section{Automorphic cohomology groups and top-degree periods}
\subsection{}\label{sect:generator}
Continuing with the notation and assumptions of the previous sections, let $\Pi$ be a cohomological cuspidal automorphic representation of $G(\A)$ which admits an $(\eta,\psi)$-Shalika model. Recall from Sect.\ \ref{sect:cohreps} that in top-degree $q_0=d(n^2+n-1)$, and for any $\epsilon = (\epsilon_v)_{v\in S_\infty} \in (K_{\infty}/K_{\infty}^{\circ})^*\cong (\Z/2\Z)^d$, the
$\epsilon$-eigenspace of $(\g_\infty,K^\circ_\infty)$-cohomology is one-dimensional:
$$\dim H^{q_0}(\g_\infty,K^\circ_\infty,\Pi_\infty\otimes E^{\sf v}_\mu)[\epsilon]=1.$$
The same therefore holds for $\Pi_\infty$ being replaced by its local Shalika model $\S^{\eta_\infty}_{\psi_\infty}(\Pi_\infty)$. We will now fix once and for all a generator of these one-dimensional cohomology spaces (i.e., for all $\epsilon\in (K_{\infty}/K_{\infty}^{\circ})^*$ at once). To this end, observe that
$$
H^{q_0}(\g_\infty,K^\circ_\infty,\S^{\eta_\infty}_{\psi_\infty}(\Pi_\infty)\otimes E^{\sf v}_\mu)[\epsilon]
\ \subseteq \
\left(\bigwedge^{q_0} \left(\g_\infty/\k_\infty\right)^*\otimes \S^{\eta_\infty}_{\psi_\infty}(\Pi_\infty)\otimes E^{\sf v}_\mu\right)^{K^\circ_\infty},
$$
cf.\  \cite[II.3.4]{bowa}. So, we may choose a generator of $H^{q_0}(\g_\infty,K^\circ_\infty,\S^{\eta_\infty}_{\psi_\infty}(\Pi_\infty)\otimes E^{\sf v}_\mu)[\epsilon]$ of the form
$$
[\Pi_\infty]^\epsilon:=\sum_{\underline i=(i_1,...,i_{q_0})}\sum_{\alpha=1}^{\dim E_\mu} X^*_{\underline i}\otimes\xi^\epsilon_{\infty,\underline i, \alpha}\otimes e^{\sf v}_\alpha,
$$
where the following data has been fixed:
\begin{enumerate}
\item A basis $\{X_j\}$ of $\g_\infty/\k_\infty$, which fixes the dual-basis $\{X_j^*\}$ for $\left(\g_\infty/\k_\infty\right)^*$.
For $\underline i=(i_1,...,i_{q_0})$, let $X^*_{\underline i}=X^*_{i_1}\wedge...\wedge X^*_{i_{q_0}}\in \bigwedge^{q_0} \left(\g_\infty/\k_\infty\right)^*.$
\item Elements $e^{\sf v}_1,...,e^{\sf v}_{\dim E_\mu}$ making up a $\Q(E^{\sf v}_\mu)$-basis of $E^{\sf v}_\mu$.
\item To each $\underline i$ and $\alpha$, $\xi^\epsilon_{\infty,\underline i, \alpha}=\otimes_{v\in S_\infty} \xi^{\epsilon_v}_{v,\underline i, \alpha}\in \S^{\eta_\infty}_{\psi_\infty}(\Pi_\infty)=\otimes_{v\in S_\infty} \S^{\eta_v}_{\psi_v}(\Pi_v)$.
\end{enumerate}
We may and will assume that $\{X_j\}$ is the extension of a fixed basis $\{Y_j\}$ of $\h_\infty/(\h_\infty\cap\k_\infty)$ via the block-diagonal embedding $\iota:H\hra G$. Finally, recall that $\sigma\in {\rm Aut}(\C)$ acts on objects at infinity which are parameterized by $S_{\infty}$ by permuting the archimedean places. This induces an action of ${\rm Aut}(\C)$ on the cohomology class at infinity:
$$\tilde\sigma\left([\Pi_\infty]^\epsilon\right):=[{}^\sigma\Pi_\infty]^{\epsilon}.$$
(Observe that $({}^\sigma\xi^\epsilon_{\infty,\underline i, \alpha})_v=\xi^{\epsilon_v}_{\sigma^{-1}\circ v,\underline i, \alpha}$)

\subsection{The map $\Theta^\epsilon$ and the definition of the period $\omega^\epsilon(\Pi_f)$}
\label{sec:theta-epsilon}
The choice of the generator $[\Pi_\infty]^\epsilon$ of $H^{q_0}(\g_\infty,K^\circ_\infty,\S^{\eta_\infty}_{\psi_\infty}(\Pi_\infty)\otimes E^{\sf v}_\mu)[\epsilon]$ fixes an isomorphism
$\Theta_\Pi^\epsilon$ of $G(\A_f)$-modules exactly as in \cite[3.3]{raghuram-shahidi-imrn}; it is the composition of
the three isomorphisms:
\begin{eqnarray*}
\S^{\eta_f}_{\psi_f}(\Pi_f) & \ira &
\S^{\eta_f}_{\psi_f}(\Pi_f) \otimes
H^{q_0}(\g_\infty,K^\circ_\infty, \S^{\eta_{\infty}}_{\psi_{\infty}}(\Pi_\infty) \otimes E^{\sf v}_\mu)[\epsilon] \\
& \ira &
H^{q_0}(\g_\infty,K^\circ_\infty, \S^{\eta}_{\psi}(\Pi) \otimes E^{\sf v}_\mu)[\epsilon] \\
& \ira &
H^{q_0}(\g_\infty,K^\circ_\infty, \Pi \otimes E^{\sf v}_\mu)[\epsilon],
\end{eqnarray*}
where the first map is $\xi_{\varphi_f} \mapsto \xi_{\varphi_f} \otimes [\Pi_{\infty}]^\epsilon$;
the second map is the obvious one; and the third map is the map induced in cohomology
by $(\S^\eta_\psi)^{-1}$. More concretely,
if we denote by
$\xi_{\underline i,\alpha}:=\xi_{\infty,\underline i, \alpha}\otimes\xi_{\varphi_f}$ and $\varphi_{\underline i,\alpha}:= (\S^\eta_\psi)^{-1}(\xi_{\underline i,\alpha})$, then
$$\Theta^\epsilon(\xi_{\varphi_f})=\sum_{\underline i=(i_1,...,i_{q_0})}\sum_{\alpha=1}^{\dim E_\mu} X^*_{\underline i}\otimes\varphi_{\underline i, \alpha}\otimes e^{\sf v}_\alpha.$$

Recall from Prop.\ \ref{prop:regalg} that the space $H^{q_0}(\g_\infty,K^\circ_\infty,\Pi \otimes E^{\sf v}_\mu)[\epsilon]$ has a certain $\Q(\Pi,\eta)$-structure. In particular, it is defined over a number field. Since we are dealing with an irreducible representation of $G(\A_f)$, such $\Q(\Pi,\eta)$-structures are unique up to homotheties, i.e., up to multiplication with non-zero complex numbers, cf.\ \cite[Lem.\ I.1]{waldsp} or \cite[Prop. 3.1]{clozel}. This leads us to the following

\begin{defprop}[The periods]
\label{defprop}
Let $\Pi$ be a cuspidal automorphic representation of $G(\A)$ which is cohomological with respect to a highest weight representation $E^{\sf v}_\mu$. Assume furthermore that $\Pi$ admits an $(\eta,\psi)$-Shalika model. Let $\epsilon$ be a character of $K_{\infty}/K_{\infty}^\circ$ and let $[\Pi_\infty]^\epsilon$ be a generator of the one-dimensional vector space $H^{q_0}(\g_\infty,K^\circ_\infty,\S^{\eta_\infty}_{\psi_\infty}(\Pi_\infty)\otimes E^{\sf v}_\mu)[\epsilon]$. Then there is a non-zero complex number $\omega^{\epsilon}(\Pi_f)=\omega^\epsilon(\Pi_f,[\Pi_\infty]^\epsilon)$, such that the normalized map
\begin{equation}\label{eq:Theta_0}
\Theta^\epsilon_{\Pi,0}:=\omega^\epsilon(\Pi_f)^{-1}\cdot\Theta_\Pi^\epsilon
\end{equation}
is ${\rm Aut}({\mathbb C})$-equivariant, i.e., the following diagram commutes:
$$
\xymatrix{
\S^{\eta_f}_{\psi_f}(\Pi_f) \ar[rrrr]^{\Theta^\epsilon_{\Pi,0}}\ar[d]_{\tilde\sigma} & & & &
H^{q_0}(\mathfrak{g}_{\infty},K_{\infty}^\circ, \Pi\otimes E^{\sf v}_{\mu})[\epsilon]
\ar[d]^{\tilde\sigma} \\
\S^{{}^\sigma\!\eta_f}_{\psi_f}({}^\sigma\Pi_f)
\ar[rrrr]^{\Theta^{\epsilon}_{{}^\sigma\Pi,0}}
& & & &
H^{q_0}(\mathfrak{g}_{\infty},K_{\infty}^\circ, {}^\sigma\Pi\otimes {}^\sigma\! E^{\sf v}_\mu)[\epsilon]
}
$$
In particular, $\Theta^\epsilon_0$ maps the $\Q(\Pi,\eta)$-structure $\S^{\eta_f}_{\psi_f}(\Pi_f)_{\Q(\Pi,\eta)}$, defined in Lem.\ \ref{lem:rational-shalika},
onto the $\Q(\Pi,\eta)$-structure of $H^{q_0}(\mathfrak{g}_{\infty},K_{\infty}^\circ, \Pi\otimes E^{\sf v}_{\mu})[\epsilon]$.
The complex number $\omega^{\epsilon}(\Pi_f)$ is well-defined only up to multiplication by invertible elements of the number field ${\mathbb Q}(\Pi,\eta)$.
\end{defprop}

\begin{proof}
Having fixed a $\Q(\Pi,\eta)$-structure on $\S^{\eta_f}_{\psi_f}(\Pi_f)$ and on $H^{q_0}(\mathfrak{g}_{\infty},K_{\infty}^\circ, \Pi\otimes E^{\sf v}_{\mu})[\epsilon]$ above, which are both unique up to homotheties, we see that there is a non-zero complex number $\omega^\epsilon(\Pi_f)$ such that $\Theta^\epsilon_{\Pi,0}=\omega^\epsilon(\Pi_f)^{-1}\cdot\Theta_\Pi^\epsilon$ maps the one $\Q(\Pi,\eta)$-structure onto the other. Now, recall the normalized new vector $\xi_{\Pi_f}$ from the proof of Lem.\ \ref{lem:rational-shalika}. By what we just said, for every $\sigma\in\textrm{Aut}(\C)$
$$\tilde\sigma\left(\Theta^\epsilon_{\Pi,0}\left(\xi_{\Pi_f}\right)\right)\quad\textrm{and}\quad\Theta^{\epsilon}_{{}^\sigma\Pi,0}\left(\tilde\sigma\left(\xi_{\Pi_f}\right)\right)$$
are both new vectors in the same $\Q({}^\sigma\Pi,{}^\sigma\!\eta)$-structure of $H^{q_0}(\mathfrak{g}_{\infty},K_{\infty}^\circ, {}^\sigma\Pi\otimes {}^\sigma\! E^{\sf v}_\mu)[\epsilon]$. Hence, these two vectors only differ by an element in $\Q({}^\sigma\Pi,{}^\sigma\!\eta)^\times$ and so, by adjusting
$\omega^{\epsilon}({}^\sigma \Pi_f)$ accordingly, we may assume that
$$\tilde\sigma\left(\Theta^\epsilon_{\Pi,0}\left(\xi_{\Pi_f}\right)\right)=\Theta^{\epsilon}_{{}^\sigma\Pi,0}\left(\tilde\sigma\left(\xi_{\Pi_f}\right)\right).$$
Since $\xi_{\Pi_f}$ generates $\S^{\eta_f}_{\psi_f}(\Pi_f)$ as a $G(\A_f)$-representation over $\C$, cf.\ the proof of Lem.\ \ref{lem:rational-shalika}, this implies that the above diagram commutes. Thus, the assertion.
\end{proof}

\begin{rem}
Note that once we have chosen $\omega^{\epsilon}(\Pi_f)$, requiring the commutativity of the above diagram actually pins down $\omega^{\epsilon}({}^\sigma \Pi_f)$. Further, if we change $\omega^{\epsilon}(\Pi_f)$ to $\kappa \, \omega^{\epsilon}(\Pi_f)$ with a $\kappa \in {\mathbb Q}(\Pi, \eta)^\times$ then the period
$\omega^{\epsilon}({}^\sigma\Pi_f)$ changes to ${\sigma}(\kappa)\cdot\omega^{\epsilon}({}^\sigma\Pi_f).$
\end{rem}

\section{Behaviour of periods upon twisting by characters}\label{sect:twisted}

The purpose of this section is to study the behaviour of the periods $\omega^{\epsilon}(\Pi_f)$ upon twisting $\Pi$ by any algebraic Hecke character
$\chi$ of $F$. This is a generalization to the context at hand of the main theorem of Raghuram--Shahidi \cite{raghuram-shahidi-imrn}.

\subsection{Preliminaries on twisting characters}\label{sect:characters}
Before we state and prove the main result of this section, we need some preliminaries on Hecke characters. By a Hecke character $\chi$ of $F$, we mean a continuous homomorphism $\chi: F^\times\backslash \A^\times \to {\mathbb C}^\times$. By an algebraic Hecke character, we mean a Hecke character $\chi$ whose component at infinity, denoted $\chi_{\infty}$, is algebraic in the sense of Clozel \cite[1.2.3]{clozel}; these are the Gr\"o\ss encharakters of type $A_0$ of A. Weil. Note that $\chi$ being algebraic implies that $\chi=\tilde\chi |\cdot|^b$ with $\tilde\chi$ a finite-order Hecke character and $b\in\Z$. In particular, at each archimedean place $v\in S_\infty$, $\chi_v(x)={\rm sgn}(x)^{a_v}|x|^{b}$, for $x\in\R^\times$, $a_v\in\{0,1\}$ and $b\in\Z$. We define the {\it signature} of an algebraic Hecke character $\chi$ of $F$ to be
$$\epsilon_{\chi}:= \left((-1)^{a_v+b}\right)_{v \in S_{\infty}} \in \{\pm 1\}^d.$$
We will think of $\epsilon_{\chi}$ as a character of $K_{\infty}/K_{\infty}^{\circ}$. We let ${\mathbb Q}(\chi)$ denote the rationality field of $\chi$. Since $\chi$ is algebraic, ${\mathbb Q}(\chi)$ is a number field. We have $\Q(\chi)=\Q(\textrm{Im}(\tilde\chi_f))$ and $\epsilon_\chi=\epsilon_{{}^\sigma\!\chi}$ for all $\sigma\in$ Aut$(\C)$. We obtain the following lemma.

\begin{lem}\label{lem:twists}
Let $\chi=\chi_\infty\otimes\chi_f$ be an algebraic Hecke character of $F$ and $\Pi$ a cuspidal automorphic representation of $G(\A)$. If $\Pi$ is cohomological with respect to the highest weight module $E^{\sf v}_\mu$, then the twisted cuspidal automorphic representation $\Pi \otimes \chi$ is cohomological with respect to the highest weight module $E^{\sf v}_{\mu-b}:=E^{\sf v}_\mu\otimes\otimes_{v\in S_\infty}\det^{-b}$. If $\Pi$ has an $(\eta,\psi)$-Shalika model, then $\Pi \otimes \chi$ has an $(\eta\chi^2,\psi)$-Shalika model. In particular, the period $\omega^{\epsilon}(\Pi_f \otimes \chi_f)$ is defined.
\end{lem}
\begin{proof}
The first part of the lemma being clear, we only prove the second assertion. Therefore, let $\varphi_\chi\in\Pi\otimes\chi$. It is of the form $\varphi_\chi(g)=\varphi(g)\chi(\det(g))$ for $\varphi\in\Pi$, $g\in G(\A)$. Now, a direct calculation shows that $\S^{\eta\chi^2}_\psi(\varphi_\chi)(g)=\S^{\eta}_\psi(\varphi)(g)\cdot\chi(\det(g))$ for $g\in G(\A)$. So, $\Pi \otimes \chi$ has an $(\eta\chi^2,\psi)$-Shalika model, if $\Pi$ has an $(\eta,\psi)$-Shalika model, cf.\ the proof of Thm.\ \ref{thm:JS}.
\end{proof}

The point being made in the first assertion of Lem.\ \ref{lem:twists} is that we are making a compatible choice of generators $[\Pi_\infty]^\epsilon$, i.e., given $\Pi$, $\chi$ and $\epsilon$, a choice of $[\Pi_\infty]^\epsilon$ pins down a choice $[\Pi_\infty\otimes\chi_\infty]^{\epsilon\cdot\epsilon_\chi}$.

Following Weil \cite[VII, Sect.\ 7]{weil}, we define the Gau\ss ~sum of $\chi_f$ as follows: We let $\mathfrak{c}$ stand for the conductor ideal of $\chi_f$. Let $y = (y_v)_{v \notin S_\infty} \in {\mathbb A}_f^{\times}$ be such that ${\rm ord}_v(y_v) = -{\rm ord}_v(\mathfrak{c}) -{\rm ord}_v(\mathfrak{D}_F)$.  The Gau\ss ~sum of $\chi_f$ is defined as $\mathcal{G}(\chi_f,\psi_f,y) = \prod_{v \notin S_\infty} \mathcal{G}(\chi_v,\psi_v,y_v)$, where the local Gau\ss ~sum $\mathcal{G}(\chi_v,\psi_v,y_v)$ is defined as
$$
\mathcal{G}(\chi_v,\psi_v,y_v) = \int_{\mathcal{O}_v^{\times}} \chi_v(u_v)^{-1}\psi_v(y_vu_v)\, du_v.
$$
For almost all $v$, where everything in sight is unramified, we have $\mathcal{G}(\chi_v,\psi_v,y_v)=1$, and for all $v$ we have $\mathcal{G}(\chi_v,\psi_v,y_v) \neq 0$. (See, for example, Godement \cite[Eq. 1.22]{godement}.) Note that, unlike Weil, we do not normalize the Gau\ss ~sum to make it have absolute value one and we do not have any factor at infinity. Suppressing the dependence on $\psi$ and $y$, we denote $\mathcal{G}(\chi_f,\psi_f,y)$ simply by $\mathcal{G}(\chi_f)$.

\subsection{The main theorem on period relations}\label{sect:twistedcharacters}

\begin{thm}
\label{thm:twisted}
Let $F$ be a totally real number field and $\Pi$ be a cohomological cuspidal automorphic representation of $G(\A)=\GL_{2n}(\A)$ which admits an $(\eta,\psi)$-Shalika model. Let $\epsilon$ be a character of $K_{\infty}/K_{\infty}^\circ$ and we let $\omega^{\epsilon}(\Pi_f)$ be the period as in Definition~\ref{defprop}. Let $\chi$ be an algebraic Hecke character of $F$, and let $\epsilon_{\chi}$ be its signature. We have the following relations:
\begin{enumerate}
\item\label{enum:twisted1} For any $\sigma \in {\rm Aut}({\mathbb C})$ we have
$$
\sigma\left(\frac
{\omega^{\epsilon \cdot \epsilon_{\chi}}(\Pi_f\otimes\chi_f)}
{\G(\chi_f)^n \, \omega^{\epsilon}(\Pi_f) }\right)
 \ = \
\left(\frac
{\omega^{\epsilon \cdot \epsilon_{\chi}} ({}^{\sigma}\Pi_f \otimes {}^\sigma\!\chi_f)}
{\G( {}^\sigma\!\chi_f)^n \, \omega^{ \epsilon} ({}^{\sigma}\Pi_f)}\right).
$$
\item Let ${\mathbb Q}(\Pi, \eta, \chi)$ be the compositum of the number fields
${\mathbb Q}(\Pi,\eta)$ and $\Q(\chi)$. We have
$$
\omega^{\epsilon \cdot \epsilon_{\chi}}(\Pi_f\otimes\chi_f) \
\sim_{{\mathbb Q}(\Pi, \eta, \chi)} \
\G(\chi_f)^n\, \omega^{\epsilon}(\Pi_f).
$$
By ``$\sim_{{\mathbb Q}(\Pi, \eta, \chi)}$'' we mean up to multiplication by an element of ${\mathbb Q}(\Pi, \eta, \chi)$.
\end{enumerate}
\end{thm}

Note that $(2)$ follows from $(1)$ by the definition of the rationality fields of $\Pi,$ $\eta$ and $\chi$. The proof of \eqref{enum:twisted1} is basically the same as the proof of Raghuram--Shahidi \cite[Thm.\ 4.1]{raghuram-shahidi-imrn} but suitably adapted to the situation at hand. This entails an analysis of the following diagram of maps. (Here, we have abbreviated the $(\g_\infty,K^\circ_\infty)$-cohomology of a module $M$ simply by $H^q(M)$.) Observe that this diagram is not commutative. Indeed, the various complex numbers involved in (1) measure the failure of commutativity of this diagram.

\begin{equation}\label{eqn:cube}
\tiny
\xymatrix{
 & \S^{\eta_f}_{\psi_f}(\Pi_f)\ar[rr]^-{\Theta^{\epsilon}_{\Pi_f}}\ar[d]^{\S_{\chi_f}}\ar[ddl]_{\tilde\sigma} & &
H^{q_0}(\Pi \otimes E^{\sf v}_{\mu})[\epsilon]
\ar[d]^{(A_{\chi}\otimes \triv_{E^{\sf v}_{\mu}})^*}\ar[ddl]_{\tilde\sigma} \\
 & \S^{\eta_f\chi_f^2}_{\psi_f}(\Pi_f \otimes \chi_f) \ar[rr]^-{\Theta_{\Pi_f \otimes \chi_f}^{\epsilon \cdot \epsilon_{\chi}}} \ar[ddl]_{\tilde\sigma} & &
H^{q_0}(\Pi\otimes \chi\otimes E^{\sf v}_{\mu-b})[\epsilon \epsilon_{\chi}] \ar[ddl]_{\tilde\sigma} \\
\S^{{}^{\sigma}\!\eta_f}_{\psi_f}({}^{\sigma}\Pi_f) \ar[rr]^-{\Theta_{{}^{\sigma}\Pi_f}^{\epsilon}}
\ar[d]_{\S_{{}^{\sigma}\!\chi_f}} & &
H^{q_0}({}^{\sigma}\Pi \otimes  {}^{\sigma}\!E^{\sf v}_{\mu})[\epsilon]
\ar[d]_{(A_{{}^{\sigma}\!\chi} \otimes \triv_{{}^{\sigma}\!E^{\sf v}_{\mu}})^*} & \\
\S^{{}^{\sigma}\!(\eta_f\chi_f^2)}_{\psi_f}({}^{\sigma}\Pi_f \otimes {}^{\sigma}\!\chi_f)
\ar[rr]^-{\Theta_{{}^{\sigma}\Pi_f \otimes  {}^{\sigma}\!\chi_f}^{\epsilon \cdot  \epsilon_{\chi}}}  & &
H^{q_0}({}^{\sigma}\Pi \otimes {}^{\sigma}\!\chi \otimes {}^{\sigma}\!E^{\sf v}_{\mu-b})[\epsilon  \epsilon_{\chi}] &
}
\end{equation}\normalsize

Here, the maps $\S_{\chi_f}$ and $A_{\chi}$ are defined as follows. If $\xi_f\in\S^{\eta_f}_{\psi_f}(\Pi_f)$, then
$$
\S_{\chi_f}(\xi_f)(g_f) := \chi_f({\rm det}(g_f)) \xi_f(g_f)
$$
for $g_f \in G({\mathbb A}_f)$. It is easy to see that $\S_{\chi_f}$ maps
$\S^{\eta_f}_{\psi_f}(\Pi_f)$ onto $\S^{\eta_f\chi_f^2}_{\psi_f}(\Pi_f \otimes \chi_f)$.
Similarly, for any automorphic form $\varphi$ of $G({\mathbb A})$ we define $A_{\chi}(\varphi)$ by
$$
A_{\chi}(\varphi)(g) := \chi({\rm det}(g)) \varphi(g)
$$ for $g \in G({\mathbb A})$.
It is easy to see that $A_{\chi}$ maps $\Pi$ onto $\Pi \otimes \chi$.
The identity map on the vector space $E^{\sf v}_{\mu}$ is denoted $\triv_{E^{\sf v}_{\mu}}$. We denote $(A_{\chi} \otimes \triv_{E^{\sf v}_{\mu}})^*$ the map induced by
$A_{\chi} \otimes \triv_{E^{\sf v}_{\mu}}$ in cohomology.

Before we may prove Thm.\ \ref{thm:twisted}, we need the following result.

\begin{prop}
\label{prop:shalika-rational}
Let $\Pi$ be a cohomological cuspidal automorphic representation of $G(\A)=\GL_{2n}(\A)$ which admits an $(\eta,\psi)$-Shalika model. Let $\chi$ be an algebraic Hecke character of $F$. For any $\sigma \in {\rm Aut}({\mathbb C})$ we have
\begin{eqnarray*}
\tilde\sigma \circ \S_{\chi_f}
& = & \sigma(\chi_f(t_{\sigma}^{-n}))\, \S_{{}^{\sigma}\!\chi_f} \circ \tilde\sigma \\
& = & \left(\frac{\sigma(\G(\chi_f))} {\G({}^{\sigma}\!\chi_f)} \right)^{-n} \, \S_{{}^{\sigma}\!\chi_f} \circ \tilde\sigma.
\end{eqnarray*}
\end{prop}

\begin{proof}
Consider the diagram
$$
\xymatrix{
\S^{\eta_f}_{\psi_f}(\Pi_f) \ar[d]^{\S_{\chi_f}} \ar[rr]^{\tilde\sigma}  & &
\S^{{}^\sigma\!\eta_f}_{\psi_f}({}^{\sigma}\Pi_f) \ar[d]^{\S_{{}^{\sigma}\!\chi_f}}  \\
\S^{\eta_f\chi_f^2}_{\psi_f}(\Pi_f\otimes\chi_f) \ar[rr]^{\tilde\sigma} & &
\S^{{}^\sigma\!(\eta_f\chi_f^2)}_{\psi_f}({}^{\sigma}\Pi_f\otimes {}^{\sigma}\!\chi_f)
}
$$
and chase an element of $\S^{\eta_f}_{\psi_f}(\Pi_f)$ both ways; note that $t_{\sigma}^{-n}$ is the determinant of the matrix $\textbf{\emph{t}}_{\sigma}^{-1}$ that was used to define the ``twisted action'' of ${\rm Aut}(\C)$ on Shalika models, cf.\  \ref{eqn:aut-c-action}. This gives the first equality. The second follows from a standard calculation which shows that
$\sigma(\G(\chi_f)) = \sigma(\chi_f(t_{\sigma})) \G({}^{\sigma}\!\chi_f).$
\end{proof}

\begin{proof}[Proof of Thm.\ \ref{thm:twisted}]
We recall that it suffices to show (1). Therefore, we compute the composition of maps in the diagram (leading from the top left corner in the back to the bottom right corner in front)
$$(A_{{}^\sigma\chi}\otimes\triv_{{}^\sigma E^{\sf v}_\mu})^*\circ \tilde\sigma\circ \Theta^\epsilon_{\Pi_f}$$
in two ways. Firstly, we have
\begin{eqnarray*}
(A_{{}^\sigma\chi}\otimes\triv_{{}^\sigma E^{\sf v}_\mu})^*\circ \tilde\sigma\circ \Theta^\epsilon_{\Pi_f} & \underset{\textrm{Def./Prop.\ \ref{defprop}}}{=} &
(A_{{}^\sigma\chi}\otimes\triv_{{}^\sigma E^{\sf v}_\mu})^*\circ \left(\frac{\sigma(\omega^{\epsilon}(\Pi_f))}{\omega^{\epsilon}({}^\sigma\Pi_f)}\right)\Theta^{\epsilon}_{{}^\sigma\Pi_f} \circ\tilde\sigma\\
& \underset{\textrm{\cite{raghuram-shahidi-imrn} Prop. 4.6}}{=} & \left(\frac{\sigma(\omega^{\epsilon}(\Pi_f))}{\omega^{\epsilon}({}^\sigma\Pi_f)}\right)\Theta^{\epsilon\cdot \epsilon_\chi}_{{}^\sigma\Pi_f\otimes{}^\sigma\chi_f} \circ\S_{{}^\sigma\chi_f}\circ\tilde\sigma.
\end{eqnarray*}
The latter reference to Raghuram--Shahidi \cite{raghuram-shahidi-imrn}, Prop.\ 4.6 may be used to prove that
$$(A_{{}^\sigma\chi}\otimes\triv_{{}^\sigma E^{\sf v}_\mu})^*\circ \Theta^{\epsilon}_{{}^\sigma\Pi_f} =\Theta^{\epsilon\cdot \epsilon_\chi}_{{}^\sigma\Pi_f\otimes{}^\sigma\chi_f} \circ\S_{{}^\sigma\chi_f}$$
by simply replacing the Whittaker models in \cite{raghuram-shahidi-imrn} Prop.\ 4.6 formally by Shalika models. \\
Secondly, we obtain
\begin{eqnarray*}
(A_{{}^\sigma\chi}\otimes\triv_{{}^\sigma E^{\sf v}_\mu})^*\circ \tilde\sigma\circ \Theta^\epsilon_{\Pi_f} & \underset{\textrm{\cite{raghuram-shahidi-imrn} Prop. 4.5}}{=} & \tilde\sigma\circ (A_{\chi}\otimes\triv_{E^{\sf v}_\mu})^*\circ \Theta^\epsilon_{\Pi_f}\\
& \underset{\textrm{\cite{raghuram-shahidi-imrn} Prop. 4.6}}{=} & \tilde\sigma\circ \Theta^{\epsilon\cdot\epsilon_\chi}_{\Pi_f\otimes\chi_f}\circ\S_{\chi_f}\\
& \underset{\textrm{Def./Prop.\ \ref{defprop}}}{=} & \left(\frac{\sigma(\omega^{\epsilon\cdot\epsilon_\chi}(\Pi_f\otimes\chi_f))}{\omega^{\epsilon\cdot\epsilon_\chi}({}^\sigma\Pi_f\otimes{}^\sigma\chi_f)}\right)
\Theta^{\epsilon\cdot\epsilon_\chi}_{{}^\sigma\Pi_f\otimes{}^\sigma\chi_f}\circ\tilde\sigma\circ\S_{\chi_f}\\
& \underset{\textrm{Prop.\ \ref{prop:shalika-rational}}}{=} & \left(\frac{\sigma(\omega^{\epsilon\cdot\epsilon_\chi}(\Pi_f\otimes\chi_f))}{\omega^{\epsilon\cdot\epsilon_\chi}({}^\sigma\Pi_f\otimes{}^\sigma\chi_f)}\right)
\Theta^{\epsilon\cdot\epsilon_\chi}_{{}^\sigma\Pi_f\otimes{}^\sigma\chi_f}\left(\frac{\sigma(\G(\chi_f))} {\G({}^{\sigma}\!\chi_f)} \right)^{-n} \, \S_{{}^{\sigma}\!\chi_f} \circ \tilde\sigma
\end{eqnarray*}
Comparing the two computations for the composition $(A_{{}^\sigma\chi}\otimes\triv_{{}^\sigma E^{\sf v}_\mu})^*\circ \tilde\sigma\circ \Theta^\epsilon_{\Pi_f}$, we obtain that
$$\left(\frac{\sigma(\omega^{\epsilon}(\Pi_f))}{\omega^{\epsilon}({}^\sigma\Pi_f)}\right)=
\left(\frac{\sigma(\omega^{\epsilon\cdot\epsilon_\chi}(\Pi_f\otimes\chi_f))}{\omega^{\epsilon\cdot\epsilon_\chi}({}^\sigma\Pi_f\otimes{}^\sigma\chi_f)}\right)
\left(\frac{\sigma(\G(\chi_f))} {\G({}^{\sigma}\!\chi_f)} \right)^{-n}.$$
This implies the result.
\end{proof}

\subsection{Finite--order characters}\label{sect:finiteorder}
Finally, recall that if a Hecke character $\chi$ of $F$ is of finite-order, then $\chi_\infty\cong\otimes_{v\in S_\infty}\sgn^{a_v}$, for some $a_v\in\{0,1\}$. Hence, by the description of the archimedean component of a cohomological cuspidal automorphic representation $\Pi$, see Sect.\ \ref{sect:cohreps}, we obtain $\Pi_\infty\cong\Pi_\infty\otimes\chi_\infty$ for all finite-order Hecke characters $\chi$ of $F$.

\section{The main identity: a cohomological interpretation of the central critical value}

In this section, we will first determine the critical points of $L(s, \Pi)$ for a cuspidal automorphic representation $\Pi$ of $G(\A)=\GL_{2n}(\A)$ which is cohomological with respect to the highest weight representation $E^{\sf v}_\mu$ and of purity weight ${\sf w}$, cf.\  Sect.\ \ref{sect:crit}. As a next step, we consider compactly supported cohomology attached to certain geometric spaces $S^G_{K_f}$ and $\tilde{S}^H_{K_f}$ defined using the groups $G$ and $H$ with values in a sheaf $\mathcal E^{\sf v}_\mu$ constructed from the highest weight representation $E_\mu$; this is the content of Sect.\ \ref{sect:SGSH}. Using a classical branching law, we see that the representation
$\triv \otimes \det^{-{\sf w}}$ of $H_\infty$ appears with multiplicity one in $E^{\sf v}_{\mu}$ if and only if $s=\tfrac12$ is critical for $L(s,\Pi)$. Assuming this to be the case, we obtain a morphism from the cohomology $H^q_c(\tilde{S}^H_{K_f},\mathcal E^{\sf v}_\mu)$ to the sheaf cohomology $H^q_c(\tilde{S}^H_{K_f},\mathcal E_{(0, {\sf -w})})$, where
$\mathcal E_{(0, {\sf -w})}$ is the sheaf constructed from $\triv \otimes \det^{-{\sf w}}$, cf.\  Sect.\ \ref{sect:S_H}. Finally, we recall Poincar\'e duality for $H^q_c(\tilde{S}^H_{K_f},\mathcal E_{(0, {\sf -w})})$ in Sect.\ \ref{sect:Poinc}. The cohomological interpretation of the central critical value is best illustrated by the main diagram in \ref{sect:diagram}, which leads to the main identity in Thm.\ \ref{thm:mainid}.

\subsection{Critical points}\label{sect:crit}
We will now determine the critical points of $L(s,\Pi)$ for a cuspidal automorphic representation $\Pi$ of $G(\A)=\GL_{2n}(\A)$ which is cohomological with respect to the highest weight representation $E^{\sf v}_\mu$. These points can be read off from the coefficient system $E_{\mu}$. According to Sect.\ \ref{sect:cohreps}, we may write a weight $\mu$ as $\mu = (\mu_v)_{v \in S_{\infty}}$ with each $\mu_v$ being of the form
$$
\mu_v = (\mu_{v,1}, \dots, \mu_{v,n}, \mu_{v,n+1},\dots, \mu_{v,2n}) =
(\mu_{v,1}, \dots, \mu_{v,n}, {\sf w}-\mu_{v,n},\dots, {\sf w}-\mu_{v,1}),
$$
where $\mu_{v,1} \geq \cdots \geq \mu_{v,n}$ and ${\sf w}\in\Z$ is the purity weight of $\Pi$.

\begin{prop}\label{prop:crit}
Let $\Pi$ be a cuspidal automorphic representation of $G(\A)=\GL_{2n}(\A)$ which is cohomological with respect to the highest weight
representation $E^{\sf v}_\mu$. Then the set of critical points for $L(s, \Pi)$ is given by
$$
{\rm Crit}(\Pi)=\left\{ \tfrac12 + m\in\tfrac12+\Z \ | -\mu_{v,n} \leq m \leq -\mu_{v,n+1} \quad \forall v\in S_\infty\right\}.
$$
In particular, $s=\frac 12$ is critical if and only if $\mu_{v,n} \geq 0 \geq \mu_{v,n+1},$ for all $v\in S_\infty$.
\end{prop}
\begin{proof}
Let us recall the definition of a point being critical. For an automorphic $L$-function $L(s, \pi)$ of degree $k$, a point $s_0 \in \C$ is critical if
$s_0 \in \frac{k-1}{2} + \Z$ and if both $L(s, \pi_{\infty})$ and $L(1-s,\pi_{\infty}^{\sf v})$ are regular at $s = s_0$, i.e, if both the $L$-factors at infinity on either side of the functional equation are holomorphic at $s_0$. This definition is due to Deligne \cite[Prop.\-Def.\ 2.3]{deligne} for motivic $L$-functions; for automorphic $L$-functions of motivic type one may read off the definition we just gave after accounting for the shift by $\frac{k-1}{2}$ coming from the so-called motivic normalization; see Clozel \cite[Conj.\ 4.5]{clozel}.
The proof is now an exercise using the local Langlands correspondence (LLC) for $\GL_{2n}(\R)$ which allows us to lay our hands on the $L$-factors at infinity. We refer the reader to Knapp \cite{knapp-llc} for all the details on LLC that we use. In our situation of a cohomological cuspidal automorphic representation $\Pi$ of $G(\A)=\GL_{2n}(\A)$  we have an $L$-function of degree $k=2n$ and so the critical points are all half-integers of the form $\frac12 + m \in \frac12 + \Z$.

Recall from \eqref{eq:piinfty} that for each $v \in S_{\infty}$ we have
$$
\Pi_v\cong\textrm{Ind}^{G(\R)}_{P(\R)}[D(\ell_{v,1})|\!\det\!|^{{\sf w}/2}\otimes...\otimes D(\ell_{v,n})|\!\det\!|^{{\sf w}/2}],
$$
where $\ell_{v,j} := 2(\mu_{v,j}+n-j)+1-{\sf w}$, $1\leq i\leq n$. From \cite{knapp-llc} we get that the $L$-factor attached to $\Pi_{\infty}$ is of the form
$$
L(s, \Pi_{\infty})  \ \approx \
\prod_{v\in S_\infty} \prod_{j=1}^n \Gamma\left(s + \frac{{\sf w}+\ell_{v,j}}{2}\right),
$$
where, by ``$\approx$'', we mean up to multiplication by non-zero constants and exponential functions (which are holomorphic and non-vanishing everywhere) which are irrelevant to compute the critical points. By definition of $\frac12 + m $ being critical we want both
\begin{eqnarray*}
L(s, \Pi_{\infty})|_{s = \frac12 + m}  &  \approx &  \prod_{v\in S_\infty} \prod_{j=1}^n \Gamma\left(m+ \mu_{v,j}+ n - j +1 \right), \\
L(1-s, \Pi_{\infty}^{\sf v})|_{s = \frac12 + m}  &  \approx &  \prod_{v\in S_\infty}\prod_{j=1}^n \Gamma\left(-m-{\sf w}+ \mu_{v,j}+ n - j +1 \right)
\end{eqnarray*}
to be regular values. Here we used that $\Pi^{\sf v}_\infty\cong\Pi_\infty|\!\det\!|^{-{\sf w}}$. Using the fact that $\Gamma(s)$ has poles only at non-positive integers and is non-vanishing everywhere, we deduce that
$$
-\mu_{v,j} - n + j  \ \leq \ m \ \leq \ \mu_{v,j} + n - j - {\sf w}
$$
for all $v \in S_{\infty}$ and all $j$ with  $1 \leq j \leq n.$ As $\mu_{v,j}-{\sf w}=-\mu_{v,2n-j+1}$ and $\mu_{v,1}\geq...\geq\mu_{v,2n}$ the proposition follows.
\end{proof}

\begin{rem}
By Sect.\ \ref{sect:finiteorder}, Crit$(\Pi)=$Crit$(\Pi\otimes\chi)$ for any finite-order Hecke character $\chi$ of $F$.
\end{rem}

\begin{rem}
One may also phrase the statement of the proposition in a motivic language: The representation $\Pi$ conjecturally corresponds to a motive $M$, and via the $L$-factors at infinity one can write down the Hodge-pairs for $M$, i.e., pairs of integers $(p,q)$ such that the Hodge number $h^{p,q}(M) \neq 0$. The critical strip is entirely a function of these Hodge pairs; see Harder--Raghuram \cite[Sect.\ 3]{harder-raghuram}.
\end{rem}

\subsection{Spaces $S^G_{K_f}$ and $\tilde{S}^H_{K_f}$ and the map $\iota^*$}\label{sect:SGSH}

Let $K_f\subset G(\A_f)$ be an open compact subgroup and consider the ``locally symmetric space'' for $G$ with level structure $K_f$ defined as:
$$S^G_{K_f}:= G(F)\backslash G(\A)/K^\circ_\infty K_f.$$
Recall the group $H=\GL_n\times \GL_n$, which is viewed as a block diagonal subgroup in $G$. As in Sect.\ \ref{sect:generator} we denote this embedding of $F$-algebraic groups by
$\iota: H\hookrightarrow G.$ Consider also the space
$$\tilde{S}^H_{K_f}:= H(F)\backslash H(\A)/(K^\circ_\infty\cap H_\infty)^\circ \iota^{-1}(K_f),$$
which is a real orbifold of dimension $\dim_\R\tilde{S}^H_{K_f}= d(n^2+n-1)=q_0.$ The numerical coincidence
\bigskip

\begin{center}
\framebox{$\dim_\R\tilde{S}^H_{K_f}= d(n^2+n-1) = q_0 = \textrm{top non-vanishing degree for cuspidal cohomology for } \GL_{2n}(\A)$}
\end{center}

\bigskip

\noindent is a crucial ingredient in proving the main identity and the reason why we only consider totally real number fields $F$.  Moreover, it is a well-known result of A.\ Borel and G.\ Prasad, see, e.g., Ash \cite{ash}, Lem.\ 2.7, that the natural inclusion
$$ \iota : \tilde{S}^H_{K_f} 
\hookrightarrow S^G_{K_f}$$
is a proper map.

Let $E^{\sf v}_\mu$ be (the dual of) a highest weight representation of $G_\infty$ as in Sect.\ \ref{sect:highweights}. It defines a sheaf $\mathcal E^{\sf v}_\mu$ on $S^G_{K_f}$, by letting $\mathcal E^{\sf v}_\mu$ be the sheaf with espace \'etal\'e $G(\A)/K^\circ_\infty K_f\times_{G(F)}E^{\sf v}_\mu$ with the discrete topology on $E^{\sf v}_\mu$. (See also Harder, \cite[(1.1.1)]{harderGL2} for a direct definition of this sheaf.) The sheaf-cohomology with compact support,
$$H^q_c(S^G_{K_f},\mathcal E^{\sf v}_\mu),$$
is a module for the Hecke algebra $\mathcal H^G_{K_f}:=C_c^\infty(G(\A_f)/\!\!/K_f,\C)$. Similarly, the sheaf-cohomology with compact support
$$H^q_c(\tilde{S}^H_{K_f},\mathcal E^{\sf v}_\mu)$$
is a module for the Hecke algebra $\mathcal H^H_{K_f}:=C_c^\infty(H(\A_f)/\!\!/\iota^{-1}(K_f),\C)$, where for brevity we also wrote $\mathcal E^{\sf v}_\mu$ for the pulled back sheaf on $\tilde{S}^H_{K_f}$. Since $\iota$ is a proper map, we obtain a well-defined morphism in cohomology
\begin{equation}\label{eq:iota}
H^q_c(S^G_{K_f},\mathcal E^{\sf v}_\mu)\stackrel{\iota^*}{\longrightarrow}
H^q_c(\tilde{S}^H_{K_f},\mathcal E^{\sf v}_\mu).
\end{equation}
If $\sigma\in {\rm Aut}(\C)$, we let ${}^\sigma\!\mathcal E^{\sf v}_\mu$ be the sheaf constructed from ${}^\sigma\! E^{\sf v}_\mu$. Then there are $\sigma$-linear isomorphisms
$$\sigma_G^*: H^q_c(S^G_{K_f},\mathcal E^{\sf v}_\mu)\ira H^q_c(S^G_{K_f},{}^\sigma\!\mathcal E^{\sf v}_\mu) \quad\textrm{ and }\quad \sigma^*_H: H^q_c(\tilde{S}^H_{K_f},\mathcal E^{\sf v}_\mu)\ira H^q_c(\tilde{S}^H_{K_f},{}^\sigma\!\mathcal E^{\sf v}_\mu),$$
cf.\  \cite{clozel}, p.128, and again a well-defined morphism in cohomology
$$H^q_c(S^G_{K_f},{}^\sigma\!\mathcal E^{\sf v}_\mu)\stackrel{\iota_\sigma^*}{\longrightarrow} H^q_c(\tilde{S}^H_{K_f},{}^\sigma\!\mathcal E^{\sf v}_\mu).$$
The next lemma is a consequence of Clozel \cite{clozel}, p.122-123:

\begin{lem}\label{lem:Qstructures}
Let $E_\mu$ be a highest weight representation of $G_\infty$ such that there is a cuspidal automorphic representation $\Pi$ of $G(\A)$ as in Sect.\ \ref{sect:cusprep}, which is cohomological with respect to $E^{\sf v}_\mu$. Then the $\mathcal H^G_{K_f}$-module $H^q_c(S^G_{K_f},\mathcal E^{\sf v}_\mu)$ and the $\mathcal H^H_{K_f}$-module $H^q_c(\tilde{S}^H_{K_f},\mathcal E^{\sf v}_\mu)$ are defined over $\Q(\Pi,\eta)$. Moreover, for all $\sigma\in {\rm Aut}(\C)$ the following diagram commutes:
$$
\xymatrix{H^q_c(S^G_{K_f},\mathcal E^{\sf v}_\mu) \ar[rr]^{\iota^*} \ar[d]^{\sigma^*_G}_\cong & & H^q_c(\tilde{S}^H_{K_f},\mathcal E^{\sf v}_\mu)\ar[d]^{\sigma^*_H}_\cong  \\
H^q_c(S^G_{K_f},{}^\sigma\!\mathcal E^{\sf v}_\mu) \ar[rr]^{\iota_\sigma^*} && H^q_c(\tilde{S}^H_{K_f},{}^\sigma\!\mathcal E^{\sf v}_\mu)}
$$
In particular, the map $\iota^*$ is a $\Q(\Pi,\eta)$-rational map, i.e., preserves the chosen $\Q(\Pi,\eta)$-structures on both sides.
\end{lem}

\subsection{The map $\mathcal T^*$}\label{sect:S_H}

\begin{prop}\label{prop:knapp}
Let $E_\mu$ be a highest weight representation of $G_\infty$ such that there is a cuspidal automorphic representation $\Pi$ of $G(\A)$ which is cohomological with respect to $E^{\sf v}_\mu$ with purity weight ${\sf w}\in\Z$. Assume furthermore that $s=\tfrac12$ is critical for $L(s,\Pi)$, or, in other words that
$$\mu_{v,n}\geq 0 \geq\mu_{v,n+1} \quad\forall v\in S_\infty.$$
Let $E_{(0, -{\sf w})}:= \triv \otimes \det^{-{\sf w}}$ be the $H(\C)$-representation, where the first block of
$H(\C) = \GL_n(\C) \times \GL_n(\C)$ acts as the trivial representation $\triv$, and the second block by
multiplication by $\det^{-{\sf w}}$. Then we have
$$
\dim {\rm Hom}_{H(\C)}(E^{\sf v}_{\mu_v}, E_{(0, -{\sf w})})=1, \quad\forall v\in S_\infty.
$$
\end{prop}

\begin{proof}
This follows from Knapp \cite[Thm.\ 2.1]{knapp} as we now briefly explain. Since $s=\tfrac12$ is critical we get the following condition on the weight $\mu$:
$$\mu_{v,1} \geq \cdots \geq \mu_{v,n} \geq 0 \geq \mu_{v,n+1} \geq \cdots \mu_{v,2n} \quad\forall v\in S_\infty.$$
Since the same argument works for all $v \in S_\infty$, let us suppress the symbol $v$. Let us tentatively denote
$\lambda = (\mu_1,\dots,\mu_n)$. Then, by purity, the weight $\lambda' := (\mu_{n+1},\dots,\mu_{2n}) = {\sf w} + \lambda^{\sf v}$ where $\lambda^{\sf v} = (-\mu_n,\dots,-\mu_1)$. Write $H(\C) = H_1 \times H_2$ with $H_1$ the left-top diagonal block of $\GL_n(\C)$ and $H_2$ the right bottom.  Then, Knapp's theorem says that under the above condition on the weight $\mu$, as a representation of
$H_1,$ we have $E_\mu^{H_2} \ \simeq \ E_\lambda \otimes E_{\lambda'},$ where the action of $H_1 = \GL_n(\C)$ on the right hand side is the diagonal action. We also have
$E_\lambda \otimes E_{\lambda'} \simeq E_\lambda \otimes ({\rm det}^{\sf w} \otimes E_{\lambda}^{\sf v}).$
Since the trivial representation appears with multiplicity one in $E_\lambda \otimes E_{\lambda}^{\sf v}$ we conclude that ${\rm det}^{\sf w}$ appears with multiplicity one in $E_\mu^{H_2}$, as a representation of $H_1$. This means that $E_{({\sf w},0)}$ appears with multiplicity one in $E_\mu$ as a representation of $H(\C)$. After dualizing we get
$\dim {\rm Hom}_{H(\C)}(E^{\sf v}_{\mu}, E_{(-{\sf w},0)})=1.$
Now consider the matrix $J \in \GL_{2n}(\C)$ given by $J_{i,j} = \delta_{i, 2n-j+1}$, and inner conjugating by $J$ we conclude $
\dim{\rm Hom}_{H(\C)}(E^{\sf v}_{\mu_v}, E_{(0, -{\sf w})})=1, \ \forall v\in S_\infty.$
\end{proof}

We will henceforth assume that $s=\tfrac12$ is critical for $L(s,\Pi)$ and fix a non-zero homomorphism
$$\mathcal T=\otimes_{v\in S_{\infty}}\mathcal T_v\in\bigotimes_{v\in S_{\infty}}\textrm{Hom}_{H(\C)}(E^{\sf v}_{\mu_v},
E_{(0, -{\sf w})}).$$
The map $\T$ also gives rise to a map in cohomology
\begin{equation}\label{eq:T}
H^{q_0}_c(\tilde{S}^H_{K_f},\mathcal E^{\sf v}_\mu)\stackrel{\mathcal T^*}{\longrightarrow} H^{q_0}_c(\tilde{S}^H_{K_f},\mathcal E_{(0, -{\sf w})}).
\end{equation}
For $\sigma\in {\rm Aut}(\C)$ we let $\mathcal T^*_\sigma$ be the map
$$H^{q_0}_c(\tilde{S}^H_{K_f},{}^\sigma\!\mathcal E^{\sf v}_\mu)\stackrel{\mathcal T_\sigma^*}{\longrightarrow} H^{q_0}_c(\tilde{S}^H_{K_f},{}^\sigma\!\mathcal E_{(0, -{\sf w})}),$$
induced from $\mathcal T_\sigma = \otimes_{v\in S_{\infty}} \mathcal T_{\sigma^{-1} v}$. Then, the following lemma is immediate:

\begin{lem}\label{lem:T^*}
For all $\sigma\in {\rm Aut}(\C)$ the following diagram commutes:
$$
\xymatrix{H^q_c(\tilde{S}^H_{K_f},\mathcal E^{\sf v}_\mu) \ar[rr]^{\mathcal T^*} \ar[d]^{\sigma^*_H}_\cong & & H^q_c(\tilde{S}^H_{K_f},\mathcal E_{(0, -{\sf w})})\ar[d]^{\sigma^*_H}_\cong  \\
H^q_c(\tilde{S}^H_{K_f},{}^\sigma\!\mathcal E^{\sf v}_\mu) \ar[rr]^{\mathcal T_\sigma^*} && H^q_c(\tilde{S}^H_{K_f},{}^\sigma\!\mathcal E_{(0, -{\sf w})}).}
$$
In particular, the map $\mathcal T^*$ is a $\Q(\Pi,\eta)$-rational map. (This is true for all $q$ but we only need it for $q = q_0.$)
\end{lem}

\subsection{Poincar\'e Duality for $\tilde{S}^H_{K_f}$}\label{sect:Poinc}
Let $\Pi$ again be a cuspidal automorphic representation of $G(\A)$ as in Sect.\ \ref{sect:cusprep}, which is cohomological with respect to a highest weight module $E^{\sf v}_\mu$ of purity weight ${\sf w}\in\Z$ and assume that $s=\tfrac 12$ is critical for $L(s,\Pi)$. We may choose $K_f\subset G(\A_f)$ to be such that its pull-back $\iota^{-1}(K_f)\subset H(\A_f)$ is a direct product $\iota^{-1}(K_f)=K^H_{1,f}\times K^H_{2,f}$, with each factor sitting inside the corresponding copy of $\GL_n(\A_f)$. We additionally assume that the open compact subgroup $K_f$ of $G(\A_f)$ is small enough such that $\eta_f$ becomes trivial on $\det(K^H_{2,f})$. Then, it is easy to see that the character $\triv\times\eta^{-1}$ of $H(\A)$ defines a cohomology class $[\eta]\in H^0(\tilde{S}^H_{K_f},\mathcal E_{(0,{\sf w})})$.

Let $\mathcal X$ be the set of all connected components of $\tilde{S}^H_{K_f}$. Using the map induced by
${\rm det} \times {\rm det} : H(\A) \to \A^\times \times \A^\times$ on $\tilde{S}^H_{K_f}$
one can see that
$\mathcal{X}$ is finite (cf.\  Borel \cite{bor1}, Thm.\ 5.1) and denote by $\tilde S^H_{K_f,x}$ the connected component corresponding to $x\in\mathcal X$. Each of them looks like a quotient of $H^\circ_\infty/(K_\infty^\circ\cap H_\infty)^\circ$ by a discrete subgroup of $H(F)$. Recall the ordered basis $\{Y_j\}$ of $\h_\infty/(\k_\infty\cap \h_\infty),$ from Sect.\ \ref{sect:generator}, which fixes the dual basis of $(\h_\infty/(\k_\infty\cap \h_\infty))^*.$ This choice of basis determines a choice of an orientation on $H^\circ_\infty/(K_\infty^\circ\cap H_\infty)^\circ$, whence on each connected component $\tilde S^H_{K_f,x}$ and so also on $\tilde S^H_{K_f}$. Now, Poincar\'e duality between
$H^{q_0}_c(\tilde{S}^H_{K_f},\mathcal E_{(0, -{\sf w})})$ and $H^0(\tilde{S}^H_{K_f},\mathcal E_{(0, {\sf w})})$ gives rise to a surjection

\begin{equation}\label{eq:poincare}
\begin{array}{cccccl}
H^{q_0}_c(\tilde{S}^H_{K_f},\mathcal E_{(0, -{\sf w})}) & \stackrel{\cong}{\longrightarrow} & \bigoplus_{x\in\mathcal X} \C & \twoheadrightarrow & \C &\\
\theta & \longmapsto & \left( \int_{\tilde{S}^H_{K_f,x}} \theta\wedge [\eta]\right)_{x\in\mathcal X} & \longmapsto & \sum_{x\in\mathcal X} \int_{\tilde{S}^H_{K_f,x}} \theta\wedge [\eta] & =\int_{\tilde{S}^H_{K_f}} \theta\wedge [\eta],
\end{array}
\end{equation}
where we remind ourselves that the orientations on all the connected components $\tilde S^H_{K_f,x}$ have been compatibly chosen.
The map $\theta\mapsto \int_{\tilde{S}^H_{K_f}} \theta\wedge [\eta]$ given by Poincar\'e duality is rational, i.e., we have the following lemma:

\begin{lem}\label{lem:intrational}
For all $\sigma\in$ \emph{Aut}$(\C)$ and for all $\theta\in H^{q_0}_c(\tilde{S}^H_{K_f},\mathcal E_{(0, -{\sf w})})$
$$\sigma\left(\int_{\tilde{S}^H_{K_f}} \theta\wedge [\eta]\right)=\int_{\tilde{S}^H_{K_f}} \sigma^*_H(\theta)\wedge [{}^\sigma\!\eta].$$
\end{lem}

\subsection{The main diagram}\label{sect:diagram}
We now have all the ingredients to talk about the strategy behind the main identity which gives a cohomological interpretation of the central critical value. Therefore, let $\Pi$ be a cuspidal automorphic representation of $G(\A)$ as in Sect.\ \ref{sect:cusprep}, which is cohomological with respect to a highest weight module $E^{\sf v}_\mu$ of purity weight ${\sf w}\in\Z$. Assume that $s=\tfrac 12$ is critical for $L(s,\Pi)$. In terms of the coefficient system $E_\mu$ this means that $\mu_{v,n}\geq 0 \geq\mu_{v,n+1} \quad\forall v\in S_\infty; $
see Prop.\ \ref{prop:crit}. Moreover, we suppose that $\Pi$ admits an $(\eta,\psi)$-Shalika model.
Recall the open compact subgroup $K_f\subset G(\A_f)$ from the end of Sect.\ \ref{sect:Poinc}. By making it even smaller, we can assure that it satisfies the following conditions:
\begin{enumerate}
\item $\eta_f$ is trivial on $\det(K^H_{2,f})$
\item $\xi^\circ_{\Pi_f}\in \S^{\eta_f}_{\psi_f}(\Pi_f)^{K_f}$.
\end{enumerate}
From now on, we fix the choice of such an open compact subgroup $K_f=K_f(\Pi_f)$. Furthermore, we fix the character
$\epsilon_0:=((-1)^{n-1},...,(-1)^{n-1})\in (K_\infty/K^\circ_\infty)^*\cong(\Z/2\Z)^d$. It only depends on the parity of $n$. Recalling the maps $\Theta_{\Pi,0}^{\epsilon_0}$
from \eqref{eq:Theta_0}, $\iota^*$ from \eqref{eq:iota}, $\mathcal T^*$ from \eqref{eq:T} and $\int_{\tilde{S}^H_{K_f}}$ from \eqref{eq:poincare},
we have the following diagram of rational maps:
$$
\xymatrix{
H_c^{q_0}(S^G_{K_f},\mathcal E^{\sf v}_\mu) \ar[r]^{\iota^*} & H_c^{q_0}(\tilde{S}^H_{K_f},\mathcal E^{\sf v}_\mu) \ar[rr]^{\mathcal T^*} & &
H_c^{q_0}(\tilde{S}^H_{K_f},\mathcal E_{(0,-{\sf w})}) \ar[dd]^{\int_{\tilde{S}^H_{K_f}}}\\
H^{q_0}(\mathfrak{g}_{\infty},K_{\infty}^\circ; \Pi\otimes E^{\sf v}_{\mu})[\epsilon_0]^{K_f} \ar@{^{(}->}[u] & & \\
\S^{\eta_f}_{\psi_f}(\Pi_f)^{K_f} \ar[u]^{\Theta_{\Pi,0}^{\epsilon_0}} \ar[rrr]& & &\C}
$$
Here we notice that the special choice of $\epsilon=\epsilon_0$ is necessary in order to obtain a cohomology class $\Theta^{\epsilon_0}_0(\xi_{\varphi_f})$
in $H^{q_0}(\mathfrak{g}_{\infty},K_{\infty}^\circ, \Pi\otimes E^{\sf v}_{\mu})[\epsilon_0]^{K_f}$ which is compatible with the choice of the orientation on the
various connected components $\tilde{S}^H_{K_f,x}$, $x\in\mathcal X$, of $\tilde{S}^H_{K_f}$.

We start with a rational vector $\xi_{\varphi_f}\in\S^{\eta_f}_{\psi_f}(\Pi_f)^{K_f}$ and chase it through the above diagram, i.e., we will compute
$$\int_{\tilde{S}^H_{K_f}}\mathcal T^*\iota^*\Theta_{\Pi,0}^{\epsilon_0}(\xi_{\varphi_f})\wedge [\eta],$$
and see that it is essentially the required $L$-value. See Thm.\ \ref{thm:mainid} below. In this computation we will need the following non-vanishing theorem of B. Sun \cite{sun}.

\subsection{A non-vanishing result}\label{sect:hypo}
 Let $Y_1^*,...,Y^*_{q_0}$ be the basis of $(\h_\infty/(\k_\infty\cap\h_\infty))^*$ chosen in Sect.\ \ref{sect:generator} and recalled in Sect.\ \ref{sect:Poinc} above. Then, for each $\underline i=(i_1,...,i_{q_0})$, there is a well-defined complex number {\bf s}$(\underline i)\in \C$ such that
 $$
 \iota^*(X^*_{\underline i})= {\bf s}(\underline i)(Y_1^*\wedge...\wedge Y^*_{q_0}).
 $$
 Let $\Pi$ be a cuspidal automorphic representation of $G(\A)$ which is cohomological with respect to a highest weight module $E^{\sf v}_\mu$ and has an $(\eta,\psi)$-Shalika model. Recall our choice of a generator $[\Pi_\infty]^{\epsilon_0}$ for the one-dimensional space
$H^{q_0}(\g_\infty,K^\circ_\infty,\S^{\eta_\infty}_{\psi_\infty}(\Pi_\infty)\otimes E^{\sf v}_\mu)[\epsilon_0]$. For a moment, let $\tfrac12$ be critical for $L(s,\Pi)$, so $\mathcal T$ exists. For every such $\Pi$ we define
$$
c(\Pi_\infty):=\sum_{\underline i=(i_1,...,i_{q_0})}\sum_{\alpha=1}^{\dim E_\mu} \textrm{\emph{\textbf{s}}}(\underline i)\,\mathcal T(e^{\sf v}_\alpha)\, \zeta_\infty(\tfrac12,\xi^{\epsilon_0}_{\infty, \underline i, \alpha}).
$$
Now, drop the assumption that $\tfrac12$ is critical for $L(s,\Pi)$, but take an arbitrary critical point $\tfrac12+m\in$ Crit$(\Pi)$. Then, consider the representation $\Pi(m):=\Pi\otimes|\!\det\!|^m$. It is a cuspidal automorphic representation which is cohomological with respect to $E^{\sf v}_{\mu+m}$. Observe that the set of critical points is shifted by $-m$, i.e., Crit$(\Pi(m))=$ Crit$(\Pi)-m$, and so by the choice of $m$, $\tfrac12$ is critical for $\Pi(m)$. Hence, by Prop.\ \ref{prop:knapp}, there is a non-trivial homomorphism $\mathcal T^{(m)}\in\bigotimes_{v\in S_{\infty}}\textrm{Hom}_{H(\C)}(E^{\sf v}_{\mu_v+m},E_{(0, -{\sf w}-2m)})$ and we are in the situation considered above. We define
\begin{equation}\label{eq:cPim}
c(\Pi_\infty,m):=c(\Pi(m)_\infty).
\end{equation}
Note that in this notation $c(\Pi_\infty,0)=c(\Pi_\infty)$, if $\tfrac12$ is critical for $L(s,\Pi)$, and by our consistent choice of generators $[\Pi(m)_\infty]^{\epsilon_0\cdot (-1)^m}$, cf.\  Lem.\ \ref{lem:twists},
$$c(\Pi_\infty,m)=\sum_{\underline i=(i_1,...,i_{q_0})}\sum_{\alpha=1}^{\dim E_\mu} \textbf{s}(\underline i)\,\mathcal T^{(m)}(e^{\sf v}_\alpha)\, \zeta_\infty(\tfrac12+m,\xi^{\epsilon_0}_{\infty, \underline i, \alpha}).$$
There is the following theorem:

\begin{thm}[Sun \cite{sun}]\label{hyp}
For all $\tfrac12+m\in {\rm Crit}(\Pi)$,
$$c(\Pi_\infty,m)=\sum_{\underline i=(i_1,...,i_{q_0})}\sum_{\alpha=1}^{\dim E_\mu} \textrm{\emph{\textbf{s}}}(\underline i)\,\mathcal T^{(m)}(e^{\sf v}_\alpha)\, \zeta_\infty(\tfrac12+m,\xi^{\epsilon_0}_{\infty, \underline i, \alpha})\neq 0.$$
We denote its inverse by $\omega(\Pi_\infty,m)$. So, $\omega(\Pi(m)_\infty)=\omega(\Pi_\infty,m)$.
\end{thm}
\begin{proof}
Let $\tfrac12 +m$ be a critical point for $L(s,\Pi)$. Without loss of generality, we may (and will) suppose that $m=0$, because $\Pi$ was assumed to be a general (i.e., not necessarily unitary) cuspidal automorphic representation. Moreover, observe that one may see as in the the proof of \cite[Prop.\ 3.24]{ragtan} that $c(\Pi_\infty,m) = \prod_{v \in S_\infty} c(\Pi_v, m),$ where each local factor $c(\Pi_v, m)$ is defined similarly. Hence, we may finish the proof by showing that $c(\Pi_v,0)$ is non-zero for all $v\in S_\infty$.\\
So, assume that $s=\tfrac12$ is critical for $L(s,\Pi)$ and let $v\in S_\infty$ be an arbitrary archimedean place. For sake of simplicity, we drop the subscript ``$v$'' now everywhere, so, e.g., $\Pi=\Pi_v$, $G=G_v=GL_{2n}(\R)$, $H=H_v=GL_n(\R)\times GL_n(\R)$ and analogous notation is used for other local archimedean objects. Define $\chi_1:=\triv$, $\chi_2:=\eta$ and let $\chi:=\chi_1\otimes\chi_2=\triv\times\eta=\triv\otimes\det^{\sf w}$ be the corresponding character of $H$. Then, the local archimedean zeta-integral at $v$ defines a non-zero homomorphism
$$\zeta(\tfrac12,.)\in\Hom_H(\Pi,\chi).$$
This follows from Prop. 3.1.5 and the fact that $s=\tfrac12$ is critical. Hence, $\zeta(\tfrac12,.)$ can be taken as the $\varphi_\chi$ in Sun's Thm.\ C, \cite{sun}. Now, recall our choice of $\mathcal T\in\Hom_{H(\C)}(E^{\sf v}_\mu\otimes E_{(0,-{\sf w})})$ from Sect.\ 6.3. Putting $w_1:=0$ and $w_2:=-{\sf w}$, we may take $\T$ to be the non-zero homomorphism $\varphi_{w_1,w_2}$ from Sun's Thm.\ C. Here observe that the condition that $s=\tfrac12$ is critical is enough for Sun's Thm.\ B to hold in this particular situation, namely where $w_1=0$ and $w_2=-{\sf w}$: In fact, Sun has to assume that $\tfrac12+w_1$ and $\tfrac12+w_2$ are both critical, in order for his Lem.\ 2.3 to hold. But the assertion of this lemma is automatic, if $s=\tfrac12$ is critical, by our Prop.\ 6.3.1. \\
In summary, we obtain by Sun's Thm.\ C, that the map
$$C:\Hom(\Lambda^{q_0}\g/\k,\Pi\otimes E^{\sf v}_\mu)\longrightarrow \Hom(\Lambda^{q_0}\h/(\k\cap\h),\chi\otimes E_{(0,-{\sf w})})$$
$$f\mapsto C(f):=(\zeta(\tfrac12,.)\otimes\T)\circ f\circ \wedge^{q_0}j_{2n}$$
is non-zero on the one-dimensional sub-space $H^{q_0}(\g,K^\circ,\Pi\otimes E^{\sf v}_\mu)[\epsilon_0].$ Here, $j_{2n}$ is Sun's notation for the embedding $\h/(\h\cap\k)\hookrightarrow\g/\k$. By the one-dimensionality of the latter cohomology space, it is hence non-zero on our choice of a generator
$$[\Pi]^{\epsilon_0}=\sum_{\underline i=(i_1,...,i_{q_0})}\sum_{\alpha=1}^{\dim E_\mu} X^*_{\underline i}\otimes\xi^{\epsilon_0}_{\underline i, \alpha}\otimes e^{\sf v}_\alpha,$$ being view as an element of $\Hom_{K^\circ}(\Lambda^{q_0}\g/\k,\Pi\otimes E^{\sf v}_\mu)[\epsilon_0]$. But, then, $C$ computes
$$C([\Pi]^{\epsilon_0})=(\zeta(\tfrac12,.)\otimes\T)\circ [\Pi]^{\epsilon_0}\circ \wedge^{q_0}j_{2n}=\sum_{\underline i=(i_1,...,i_{q_0})}\sum_{\alpha=1}^{\dim E_\mu} \textrm{\emph{\textbf{s}}}(\underline i)\,\mathcal T(e^{\sf v}_\alpha)\, \zeta(\tfrac12,\xi^{\epsilon_0}_{\underline i, \alpha})=c(\Pi,0).$$

\end{proof}

\begin{rem}
For $n=1$, the numbers $c(\Pi_\infty,m)$ are known by an explicit calculation; see Raghuram-Tanabe \cite[Prop.\ 3.24]{ragtan}. Ultimately, one expects that one may always choose $[\Pi_\infty]^{\epsilon_0}$ such that $\omega(\Pi_\infty,m)$ is a power of $2\pi i$. Moreover, since any $\sigma \in {\rm Aut}(\C)$ acts on $\Pi_\infty$ by permuting the local components and $c(\Pi_\infty,m) = \prod_{v \in S_\infty} c(\Pi_v, m),$ we deduce that $c({}^\sigma\Pi_\infty,m) = c(\Pi_\infty,m)$, and $\omega({}^\sigma\Pi_\infty,m) = \omega(\Pi_\infty,m).$
\end{rem}

\subsection{The main identity}\label{sect:mainid}

\begin{thm}[Main Identity]\label{thm:mainid}
Let $\Pi$ be a cohomological cuspidal automorphic representation of $G(\A)$ such that $s=\tfrac12$ is critical for $L(s,\Pi)$. Assume that $\Pi$ admits an $(\eta,\psi)$-Shalika model and let $\epsilon_0$ and $K_f=K_f(\Pi_f)$ be chosen as in \ref{sect:diagram}. Then
$$
\int_{\tilde{S}^H_{K_f}}\mathcal T^*\iota^*\Theta_{\Pi,0}^{\epsilon_0}(\xi^\circ_{\Pi_f})\wedge [\eta] \ =\
\frac{L(\tfrac 12,\Pi_f)}{\omega^{\epsilon_0}(\Pi_f)\omega(\Pi_\infty)}\cdot
\frac{1}{vol(\iota^{-1}(K_f))\prod_{v\in S_{\Pi_f,\psi}}L(\tfrac 12,\Pi_v)},
$$
where $\xi^\circ_{\Pi_f}$ is the compatible vector in the Shalika model chosen as in \ref{sect:vector}.
\end{thm}

\begin{proof}
In order to have the notation ready at hand, let $E^{\sf v}_\mu$ be the highest weight representation with respect to which $\Pi$ is cohomological and say that $\Pi$ is of purity weight ${\sf w}\in\Z$. For each $\underline i=(i_1,...,i_{q_0})$ and $\alpha$ let us write $\varphi^\circ_{\underline i,\alpha}:=(\S^{\eta}_{\psi})^{-1}(\xi_{\infty,\underline i,\alpha}\otimes\xi^\circ_{\Pi_f})$. Then we obtain
\begin{eqnarray*}
\int_{\tilde{S}^H_{K_f}}\mathcal T^*\iota^*\Theta_{\Pi,0}^{\epsilon_0}(\xi^\circ_{\Pi_f})\wedge [\eta] & = & \omega^{\epsilon_0}(\Pi_f)^{-1}\sum_{\underline i=(i_1,...,i_{q_0})}\sum_{\alpha=1}^{\dim E_\mu} \int_{\tilde{S}^H_{K_f}} \iota^*(X^*_{\underline i})[\eta]\cdot\varphi^\circ_{\underline i, \alpha}|_{H(\A)} \mathcal T(e^{\sf v}_\alpha)\\
& = & vol(\iota^{-1}(K_f))^{-1}c^{-1}\omega^{\epsilon_0}(\Pi_f)^{-1}\sum_{\underline i,\alpha} \textrm{\textbf{s}}(\underline i)\mathcal T(e^{\sf v}_\alpha)\int_{H(F)\backslash H(\A)/\R_+^d} [\eta]\cdot\varphi^\circ_{\underline i, \alpha}|_{H(\A)} dh,
\end{eqnarray*}
where the last equality is due to the choice of the measure, cf.\  Sect.\ \ref{sect:measures}, and the right $K_f$-invariance of $\xi^\circ_{\Pi_f}$. For an individual summand index by $\underline i$ and $\alpha$, we obtain more explicitly
$$
\int_{H(F)\backslash H(\A)/\R_+^d} [\eta]\cdot\varphi^\circ_{\underline i, \alpha}|_{H(\A)}d(h_1,h_2)$$
$$
=\int_{H(F)\backslash H(\A)/\R_+^d} [\eta]([h_1,h_2])\cdot\varphi^\circ_{\underline i, \alpha}\left(\!\!\left( \!\begin{array}{ccc}
h_1 &  0\\
0 &  h_2
\end{array}\!\!\right)\!\!\right) d(h_1,h_2)
$$
$$
=\int_{H(F)\backslash H(\A)/\R_+^d}\eta^{-1}(\det(h_2))  \cdot
\varphi^\circ_{\underline i, \alpha}\left(\!\! \left(\!\!\begin{array}{ccc}
h_1 &  0\\
0 &  h_2
\end{array}\!\!\right)\!\!\right) d(h_1,h_2)
$$
$$
= \int_{Z_G(\A)H(F)\backslash H(\A)} \int_{Z_G(F)\backslash Z_G(\A)/\R_+^d}\left(\varphi^\circ_{\underline i, \alpha}\left(\!\!\left(\!\! \begin{array}{ccc}
h_1 &  0\\
0 &  h_2
\end{array}\!\!\right)\cdot z\right)\cdot \eta^{-1}(\det(h_2\cdot z))dz\right) d(h_1,h_2),
$$
where $h_j=(h_{j,\infty},h_{j,f})\in \GL_n(\A)$, $j=1,2$, and $z={\rm diag}(a,...,a)\in Z_G(F)\backslash Z_G(\A)/\R_+^d$. Furthermore, this equals
$$
\int_{Z_G(\A)H(F)\backslash H(\A)} \int_{Z_G(F)\backslash Z_G(\A)/\R_+^d}\left(\underbrace{\omega_\Pi(z)\eta(a)^{-n}}_{=1}\cdot\varphi^\circ_{\underline i, \alpha}\left(\!\!\left(\!\! \begin{array}{ccc}
h_1 &  0\\
0 &  h_2
\end{array}\!\!\right)\!\!\right)\cdot \eta^{-1}(\det(h_2))dz\right) d(h_1,h_2)
$$
whence the integrand is $Z_G(\A)$-invariant and we are left with
$$
vol(F^\times\backslash\A^\times/\R_+^d)\cdot\int_{Z_G(\A)H(F)\backslash H(\A)} \varphi^\circ_{\underline i, \alpha}\left(\!\!\left(\!\! \begin{array}{ccc}
h_1 &  0\\
0 &  h_2
\end{array}\!\!\right)\!\!\right)\cdot \eta^{-1}(\det(h_2))\, d(h_1,h_2).
$$
Recalling that $c=vol(F^\times\backslash\A^\times/\R_+^d)$, cf.\ Sect.\ \ref{sect:measures}, we may therefore finish the proof by showing that
$$
\frac{L(\tfrac 12,\Pi_f)}{\omega(\Pi_\infty)\prod_{v\in S_{\Pi_f,\psi}}L(\tfrac 12,\Pi_v)}= \sum_{\underline i,\alpha}\textrm{\textbf{s}}(\underline i)\mathcal T(e^{\sf v}_\alpha)\int_{Z_G(\A)H(F)\backslash H(\A)} \varphi^\circ_{\underline i, \alpha}\left(\!\!\left(\!\! \begin{array}{ccc}
h_1 &  0\\
0 &  h_2
\end{array}\!\!\right)\!\!\right)\cdot \eta^{-1}(\det(h_2))\, d(h_1,h_2).
$$
Recall from Prop.\ \ref{prop:FJ} that for $Re(s)\gg0$ there is the equality
$$
\int_{Z_G(\A)H(F)\backslash H(\A)} \varphi^\circ_{\underline i, \alpha}\left(\!\!\left(\!\! \begin{array}{ccc}
h_1 &  0\\
0 &  h_2
\end{array}\!\!\right)\!\!\right)\Bigg|\frac{\det(h_1)}{\det(h_2)}\Bigg|^{s-1/2}\eta^{-1}(\det(h_2))\, d(h_1,h_2)
$$
$$=\int_{\GL_n(\A)} \S^\eta_\psi(\varphi^\circ_{\underline i,\alpha})\left(\!\!\left(\!\! \begin{array}{ccc}
g_1 &  0\\
0 &  1
\end{array}\!\!\right)\!\!\right) |\det(g_1)|^{s-1/2}\, dg_1$$
$${\small
=\left(\int_{\GL_n(\A_f)} \xi^\circ_{\Pi_f}\left(\!\!\left(\!\! \begin{array}{ccc}
g_{1,f} &  0\\
0 &  1
\end{array}\!\!\right)\!\!\right) |\det(g_{1,f})|^{s-1/2}dg_{1,f}\right) \cdot \left(\int_{\GL_n(\R)^d}\xi^{\epsilon_0}_{\infty,\underline i,\alpha}\left(\!\!\left(\!\! \begin{array}{ccc}
g_{1,\infty} &  0\\
0 &  1
\end{array}\!\!\right)\!\!\right)|\det(g_{1,\infty})|^{s-1/2}dg_{1,\infty}\!\right)
}$$
$$
= \zeta_f(s,\xi^\circ_{\Pi_f})\cdot \zeta_\infty(s,\xi^{\epsilon_0}_{\infty,\underline i,\alpha}),
$$
which after analytic continuation is valid for all $s\in\C$. The last factor $\zeta_\infty(s,\xi^{\epsilon_0}_{\infty,\underline i,\alpha})$ is a meromorphic function in $s$, but since $s=\frac 12$ is critical for $L(s,\Pi)$, the archimedean factor $\zeta_\infty(\frac12,\xi^{\epsilon_0}_{\infty,\underline i,\alpha})$ is finite for all $\underline i$ and $\alpha$, see Prop.\ \ref{prop:FJL-fct}. According to our special choice of the vector $\xi^\circ_{\Pi_f}$ in Sect.\ \ref{sect:vector} we see that at $s=\frac 12$ the last expression equals
$$\frac{L(\tfrac 12,\Pi_f)}{\prod_{v\in S_{\Pi_f,\psi}}L(\tfrac 12,\Pi_v)}\cdot\zeta_\infty(\tfrac 12,\xi^{\epsilon_0}_{\infty,\underline i,\alpha}).$$
The result follows since
$c(\Pi_\infty)=\omega(\Pi_\infty)^{-1}=\sum_{\underline i,\alpha}\textrm{\textbf{s}}(\underline i)\,\mathcal T(e^{\sf v}_\alpha)\,\zeta_\infty(\frac12,\xi^{\epsilon_0}_{\infty,\underline i,\alpha})$.

\end{proof}

\section{Algebraicity results for all critical $L$-values}
\subsection{}
Before stating the main theorem of this article, let us record a preliminary lemma which says that local $L$-values at a critical point transform rationally under $\sigma$-twisting.

\begin{lem}
\label{lem:localvalues}
For a finite place $v$ of $F$, let $\Pi_v$ be (any) irreducible admissible representations of $\GL_{2n}(F_v)$. Then
$$
\sigma(L(\tfrac12, \Pi_v)) = L(\tfrac12, {}^{\sigma}\Pi_v).
$$
\end{lem}
\begin{proof}
This can be showed exactly as in the proof of Raghuram \cite[Prop.\ 3.17]{raghuram-imrn}.
\end{proof}

Recall the periods $\omega^{\epsilon}(\Pi_f)$ from Def./Prop.\ \ref{defprop}, the Gau\ss~sum $\G(\chi_f)$ from Sect.\ \ref{sect:characters} and the non-zero quantities $\omega(\Pi_\infty,m)$ from Sect.\ \ref{sect:hypo}. We now prove the main theorem of this paper on the algebraicity of all the critical values
$L(\frac 12+m,\Pi_f \otimes \chi_f).$

\begin{thm}\label{thm:central-value}
Let $F$ be a totally real number field and $G=\GL_{2n}/F$, $n\geq 1$. Let $\Pi$ be a cuspidal automorphic representation of $G(\A)$, which is cohomological with respect to a highest weight representation $E^{\sf v}_\mu$ of $G_\infty$ and which admits an $(\eta,\psi)$-Shalika model. Let $\chi$ be a Hecke character of $F$ of finite-order and \emph{Crit}$(\Pi)=$\emph{Crit}$(\Pi\otimes\chi)\subset\tfrac12 + \Z$ be the set of critical points for the $L$-function $L(s,\Pi\otimes\chi)$ of $\Pi\otimes\chi$. Then for all critical points $\tfrac12 + m \in$ \emph{Crit}$(\Pi)$ the following assertions hold:
\begin{enumerate}
\item For every $\sigma\in {\rm Aut}(\C)$,
$$
\sigma\left(
\frac{L(\tfrac 12+m,\Pi_f \otimes \chi_f)}{\omega^{(-1)^{n+m-1}\epsilon_{\chi}}(\Pi_f) \, \G(\chi_f)^n \, \omega(\Pi_\infty,m)}\right) \ = \
\frac{L(\tfrac 12+m,{}^\sigma\Pi_f \otimes {}^\sigma\!\chi_f)}{\omega^{(-1)^{n+m-1}\epsilon_\chi}({}^\sigma\Pi_f) \, \G({}^\sigma\!\chi_f)^n \,\omega(\Pi_\infty,m)}.
$$
\item
$$
L(\tfrac 12+m,\Pi_f\otimes\chi_f) \ \sim_{\Q(\Pi,\eta, \chi)} \
\omega^{(-1)^{n+m-1}\epsilon_\chi}(\Pi_f) \, \G(\chi_f)^n  \, \omega(\Pi_\infty,m),
$$
where ``$\sim_{\Q(\Pi,\eta,\chi)}$'' means up to multiplication by an element in the number field $\Q(\Pi,\eta,\chi)$.
\end{enumerate}
\end{thm}

\begin{proof}
Note that (1) implies (2) by definition of the rationality fields. For convenience of the reader we divide the proof of (1) into three steps.\\

\emph{Step 1: Assume $\tfrac12$ is critical for $L(s,\Pi)$ (i.e., $m=0$) and $\chi=\triv$}\\
In this case the sign $\epsilon_{\chi}  = (+1,\dots,+1)$ is the trivial sign character and $\omega(\Pi_\infty,m)=\omega(\Pi_\infty)$. Let $\sigma\in {\rm Aut}(\C)$, let $K_f=K_f(\Pi_f)$ be chosen for $\Pi$ as in Sect.\ \ref{sect:diagram} and recall our special choice of the vector $\xi^\circ_{\Pi_f}$ from Sect.\ \ref{sect:vector}. The Main Identity, cf.\  Thm.\ \ref{thm:mainid}, implies that
\begin{eqnarray*}
\sigma\left(\int_{\tilde{S}^H_{K_f}}\mathcal T^*\iota^*\Theta_{\Pi,0}^{\epsilon_0}(\xi^\circ_{\Pi_f})\wedge [\eta]\right) & = & \sigma\left(\frac{L(\tfrac 12,\Pi_f)}{\omega^{\epsilon_0}(\Pi_f)\,\omega(\Pi_\infty)}\cdot \frac{1}{vol(\iota^{-1}(K_f))\prod_{v\in S_{\Pi_f,\psi}}L(\tfrac 12,\Pi_v)}\right)\\
& = & \sigma\left(\frac{L(\tfrac 12,\Pi_f)}{\omega^{\epsilon_0}(\Pi_f)\,\omega(\Pi_\infty)}\right)\cdot \frac{1}{vol(\iota^{-1}(K_f))\prod_{v\in S_{\Pi_f,\psi}}\sigma(L(\tfrac 12,\Pi_v))},
\end{eqnarray*}
where the last line follows from the fact that the volume appearing in the formula is a rational number by the choice of the measure, cf.\  Sect.\ \ref{sect:measures}. On the other hand,

\begin{eqnarray*}
\sigma\left(\int_{\tilde{S}^H_{K_f}}\mathcal T^*\iota^*\Theta_{\Pi,0}^{\epsilon_0}(\xi^\circ_{\Pi_f})\wedge [\eta]\right) & \underset{\textrm{Lem.\ \ref{lem:intrational}}}{=} & \int_{\tilde{S}^H_{K_f}}\sigma^*_H\left(\mathcal T^*\iota^*\Theta_0^{\epsilon_0}(\xi^\circ_{\Pi_f})\right)\wedge [{}^\sigma\!\eta] \\
& \underset{\textrm{Lem.\ \ref{lem:T^*}}}{=} & \int_{\tilde{S}^H_{K_f}}\mathcal T_\sigma^*\sigma^*_H\left(\iota^*\Theta_{\Pi,0}^{\epsilon_0}(\xi^\circ_{\Pi_f})\right)\wedge [{}^\sigma\!\eta] \\
& \underset{\textrm{Lem.\ \ref{lem:Qstructures}}}{=} & \int_{\tilde{S}^H_{K_f}}\mathcal T_\sigma^*\iota_\sigma^*\sigma^*_G\left(\Theta_{\Pi,0}^{\epsilon_0}(\xi^\circ_{\Pi_f})\right)\wedge [{}^\sigma\!\eta] \\
& \underset{\textrm{Lem.\ \ref{defprop}}}{=} & \int_{\tilde{S}^H_{K_f}}\mathcal T_\sigma^*\iota_\sigma^*\Theta_{\Pi,0}^{\epsilon_0}\left(\tilde\sigma\left(\xi^\circ_{\Pi_f}\right)\right)\wedge [{}^\sigma\!\eta] \\
& \underset{\textrm{Def. of $\tilde\sigma$}}{=} & \int_{\tilde{S}^H_{K_f}}\mathcal T_\sigma^*\iota_\sigma^*\Theta_{\Pi,0}^{\epsilon_0}\left({}^\sigma\xi^\circ_{\Pi_f}\right)\wedge [{}^\sigma\!\eta] \\
& \underset{\textrm{Sect.\ \ref{sect:vector}}}{=} & \int_{\tilde{S}^H_{K_f}}\mathcal T_\sigma^*\iota_\sigma^*\Theta_{\Pi,0}^{\epsilon_0}\left(\xi^\circ_{{}^\sigma\Pi_f}\right)\wedge [{}^\sigma\!\eta].
\end{eqnarray*}
By Prop.\ \ref{prop:regalg}, ${}^\sigma\Pi$ is again a cohomological cuspidal automorphic representation and by Thm.\ \ref{thm:arithmeticShalika} it admits a $({}^\sigma\!\eta,\psi)$-Shalika model. One immediately checks that ${}^\sigma\Pi$ also satisfies the hypotheses of the Main Identity, Thm.\ \ref{thm:mainid}, and so, applying it to ${}^\sigma\Pi$, we see that the last line equals
$$\frac{L(\tfrac 12,{}^\sigma\Pi_f)}{\omega^{\epsilon_0}({}^\sigma\Pi_f)\,\omega({}^\sigma\Pi_\infty)}\cdot \frac{1}{vol(\iota^{-1}(K_f))\prod_{v\in S_{{}^\sigma\Pi_f,\psi}}L(\tfrac 12,{}^\sigma\Pi_v)}.$$
As $\omega({}^\sigma\Pi_\infty)=\omega(\Pi_\infty)$ and $S_{\Pi_f} = S_{{}^\sigma\Pi_f}$, Lem.\ \ref{lem:localvalues} finishes the proof in this case.\\

\emph{Step 2: $m$ is arbitrary and $\chi=\triv$}\\
Now, let $\tfrac12+m\in$ Crit$(\Pi)$ be an arbitrary critical value of $\Pi$. Consider the representation $\Pi(m):=\Pi\otimes|\!\det\!|^m$ as in Sect.\ \ref{sect:hypo}. It is a cuspidal automorphic representation which is cohomological. Observe that the set of critical points is shifted by $-m$, i.e., Crit$(\Pi(m))=$ Crit$(\Pi)-m$, and so by the choice of $m$, $\tfrac12$ is critical for $\Pi(m)$ and $c(\Pi(m)_\infty)=c(\Pi_\infty,m)\neq 0$ (i.e., $\omega(\Pi(m)_\infty)=\omega(\Pi_\infty,m)$ exists) by Thm. \ref{hyp}. Furthermore, we may take $K_f(\Pi(m)_f)=K_f(\Pi_f)$, since $|\!\det(K_f(\Pi_f))\!|^m\equiv 1$. Therefore, we can apply the result proved in step one to $\Pi(m)$ and obtain

$$ \sigma\left(\frac{L(\tfrac 12,\Pi(m)_f)}{\omega^{\epsilon_0}(\Pi(m)_f)\,\omega(\Pi(m)_\infty)}\right)   =  \frac{L(\tfrac 12,{}^\sigma\Pi(m)_f)}{\omega^{\epsilon_0}({}^\sigma\Pi(m)_f)\,\omega(\Pi(m)_\infty)} =  \frac{L(\tfrac 12+m,{}^\sigma\Pi_f)}{\omega^{\epsilon_0}({}^\sigma\Pi(m)_f)\,\omega(\Pi_\infty,m)}.$$

For the last equation, observe that ${}^\sigma(|\!\det\!|^m)=\sigma(|\!\det\!|)^m=|\!\det\!|^m$. Applying Thm.\ \ref{thm:twisted} on period relations to the algebraic Hecke character $|\cdot|^{m}$ associated to the twist $|\!\det\!|^m$ and keeping in mind that $\G(|\cdot|_f^{m})=1$ gives the theorem in this case.\\

\emph{Step 3: $m$ and $\chi$ are arbitrary}\\
Finally, let $\chi$ be any finite-order Hecke character of $F$. Applying step two to the twisted representation $\Pi \otimes \chi$, gives
$$
\sigma\left(
\frac{L(\tfrac 12+m,\Pi_f \otimes \chi_f)}{\omega^{(-1)^{n+m-1}}(\Pi_f\otimes \chi_f)\, \omega(\Pi_\infty,m)}\right) \ = \
\frac{L(\tfrac 12+m,{}^\sigma\Pi_f \otimes {}^\sigma\!\chi_f)}{\omega^{(-1)^{n+m-1}}({}^\sigma\Pi_f\otimes {}^\sigma\!\chi_f)\, \omega(\Pi_\infty,m)},
$$
for any critical value $\tfrac12+m\in{\rm Crit}(\Pi)={\rm Crit}(\Pi\otimes\chi)$ and $\sigma\in$ Aut$(\C)$. Here we observe that $\Pi_\infty\cong \Pi_\infty\otimes\chi_\infty$, since $\chi$ is of finite-order, see Sect.\ \ref{sect:finiteorder}. The result now follows from Thm.\ \ref{thm:twisted}.
\end{proof}

\section{Complementa}
\label{sec:complementa}

In this section we give several families of examples to which our main result (Thm.\ \ref{thm:central-value}) on $L$-values applies.
We also comment on the compatibility of this theorem with Deligne's conjecture on the critical values of motivic $L$-functions. Our theorem is also weakly compatible with a conjecture of Gross on the order of vanishing of a motivic $L$-function at a critical point.

\subsection{The symmetric cube $L$-functions for Hilbert modular forms}\label{sec:symmetric}
\subsubsection{} In this section we want to construct a family of examples for Thm.\ \ref{thm:central-value} starting from Hilbert modular forms. For a $d$-tuple $k=(k_v)_{v\in S_\infty}\in\Z^d$ with $k_v\geq 1$ and an integral ideal $\n\subseteq\mathcal O$ of $F$, let $\mathcal M_k(\n,\tilde\omega)$ be the space of holomorphic Hilbert modular forms of level $\n$ and character $\tilde\omega$. The subspace of cuspidal holomorphic Hilbert modular forms is denoted $S_k(\n,\tilde\omega)$. To each primitive cusp form ${\bf f}\in S_k(\n,\tilde\omega)$ we can associate a cuspidal automorphic representation $\pi({\bf f})$ of $\GL_2(\A)$ with archimedean component
$$\pi({\bf f})_\infty\cong \bigotimes_{v\in S_\infty} D(k_v-1).$$
All details concerning this construction may be found in Raghuram--Tanabe, \cite[Sect.\ 4]{ragtan} to which we refer. We write $\omega_{\pi({\bf f})}$ for the central character of $\pi({\bf f})$. Let $k_0$ (resp., $k^0$) be the maximum (resp., the minimum) of all $k_v$, $v\in S_\infty$ and set
$$\pi:=\pi({\bf f})\otimes |\cdot|^{k_0/2}.$$
If $k_v\geq 2$ and $k_v\equiv k_{w}\pmod 2$ for all $v,w\in S_\infty$, then $\pi$ is cohomological, cf.\  \cite[Thm.\ 8.3]{ragtan}.

\subsubsection{}\label{sect:symmtrans}
Let Sym$^a$ be the $a$-th symmetric power of the standard representation of $\GL_2(\C)$. By the Local Langlands Correspondence, see Harris--Taylor \cite{harris} and Henniart \cite{henniart} for the non-archimedean places and Langlands \cite{lang2} for the archimedean places, Sym$^a(\pi):=\bigotimes'_v {\rm Sym}^a(\pi_v)$ is a well-defined irreducible admissible representation of $\GL_{a+1}(\A)$. According to Kim--Shahidi, \cite[Thm.\ 6.1]{kimshahidiANN}, the symmetric cube Sym$^3(\pi)$ of $\pi$, is known to be an automorphic representation.

\subsubsection{} With this notation in place we obtain the following proposition.

\begin{prop}\label{prop:hilbertmod}
Let $k=(k_v)_{v\in S_\infty}\in\Z^d$ with $k_v\geq 2$ and $k_v\equiv k_{w}\pmod 2$ for all $v,w\in S_\infty$. Let ${\bf f}\in S_k(\n,\tilde\omega)$ be a primitive holomorphic Hilbert modular cusp form and let $\pi({\bf f})$ be the corresponding cuspidal automorphic representation of $\GL_2(\A)$. Assume that $\pi({\bf f})$ is not dihedral and denote by
$\pi=\pi({\bf f})\otimes |\cdot|^{k_0/2}$. Then the symmetric cube transfer
$$\Pi={\rm Sym}^3(\pi)$$
of $\pi$ is a cuspidal automorphic representation of $\GL_4(\A)$ which is cohomological with respect to a highest weight module $E^{\sf v}_\mu,$ and admits an $(\eta,\psi)$-Shalika model with $\eta=\omega_{\pi({\bf f})}^3|\cdot|^{3k_0}$.
\end{prop}
\begin{proof}
For convenience we divide the proof in three parts.
\smallskip

\emph{$\Pi$ is cuspidal:}
As we assumed that $\pi({\bf f})$ is not dihedral, ${\rm Sym}^3(\pi({\bf f}))$ is cuspidal unless $\pi({\bf f})$ is of tetrahedral type, cf.\  Kim--Shahidi \cite[Thm.\ 6.1]{kimshahidiANN}. But for $\pi({\bf f})$ to be tetrahedral, it is necessary that the local Langlands parameter of $\pi({\bf f})_\infty$ has finite image. Using Knapp,
\cite{knapp-llc}, we obtain that at each $v\in S_\infty$ this local Langlands parameter is ${\rm Ind}^{W_\R}_{\C^\times}[(\tfrac{z}{\overline z})^{\frac{k_v-1}{2}}]$, which has infinite image unless $k_v=1$ for all $v\in S_\infty$. Since we assume that $k_v\geq 2$, this shows that ${\rm Sym}^3(\pi({\bf f}))$ is cuspidal, and hence
$\Pi={\rm Sym}^3\left(\pi({\bf f})\otimes |\cdot|^{k_0/2}\right)\cong {\rm Sym}^3(\pi({\bf f}))\otimes |\cdot|^{3k_0/2},$
is also cuspidal.
\smallskip

\emph{$\Pi$ is cohomological:}
For each $v\in S_\infty$, let
$$\mu_v=\left(\frac{3(k_v-2)+3k_0}{2},\frac{k_v-2+3k_0}{2},\frac{-(k_v-2)+3k_0}{2},\frac{-3(k_v-2)+3k_0}{2}\right)=(k_v-2)\rho_4+\tfrac{3k_0}{2},$$
where $\rho_4$ is half the sum of positive roots of $\GL_4(\R)$. Then, $\Pi_v$ is cohomological with respect to $E^{\sf v}_{\mu_v}$ for all $v\in S_\infty$ by Raghuram--Shahidi \cite[Thm.\ 5.5]{ragsha} and so, using the K\"unneth rule, $\Pi$ is cohomological with respect to the highest weight module $E^{\sf v}_\mu$ with $\mu=(\mu_v)_{v\in S_\infty}$.
\smallskip

\emph{$\Pi$ has a $(\eta, \psi)$-Shalika model:}
Let $\eta=(\omega_{\pi({\bf f})}|\cdot|^{k_0})^3$. Then one easily checks that $\eta^2=\omega_{\Pi}$, so $\eta$ is an id\`ele class character as considered in Sect.\ \ref{sect:cusprep}. We now show that $\Pi$ has an $(\eta,\psi)$-Shalika model. By Thm.\ \ref{thm:JS}, this amounts to proving that the partial $L$-function $L^S(s,\Pi,\wedge^2\otimes\eta^{-1})$ has a pole at $s=1$, $S$ being any finite set of places containing $S_{\Pi,\eta}$. We obtain
\begin{eqnarray*}
\wedge^2(\Pi) & \cong & \wedge^2\left({\rm Sym}^3\left(\pi({\bf f})\otimes |\cdot|^{k_0/2}\right)\right)
\cong \wedge^2\left({\rm Sym}^3\left(\pi({\bf f})\right)\otimes |\cdot|^{3k_0/2}\right) \\
& \cong & \wedge^2\left({\rm Sym}^3\left(\pi({\bf f})\right)\right)\otimes |\cdot|^{3k_0}
\cong  \left({\rm Sym}^4(\pi({\bf f}))\otimes\omega_{\pi({\bf f})}\boxplus\omega^3_{\pi({\bf f})}\right)\otimes |\cdot|^{3k_0},
\end{eqnarray*}
where ``$\boxplus$'' denotes the isobaric direct sum of automorphic representations. Observe that we have used Kim \cite[Sect.\ 7]{kim-jams} in order to obtain the last line. This shows that
$\wedge^2(\Pi)\otimes\eta^{-1} \cong \left({\rm Sym}^4(\pi({\bf f}))\otimes\omega^{-2}_{\pi({\bf f})}\right)\boxplus\triv.$
Now, Kim \cite[Thm.\ 7.3.2]{kim-jams} together with Jacquet--Shalika \cite[Thm.\ (1.3)]{jacquet-shalika-invent} shows that $L^S(s,\Pi,\wedge^2\otimes\eta^{-1})$ has a pole at $s=1$ as desired.
\end{proof}

\begin{cor}\label{cor:hilbertmod}
Thm.\ \ref{thm:central-value} can be applied to any representation $\Pi={\rm Sym}^3(\pi)$ as in Prop.\ \ref{prop:hilbertmod}. In particular, if $\chi$ is a finite-order Hecke character of $F$, then for all integers $m$ with $-\tfrac{k^0-2}{2}\leq m +\tfrac{3k_0}{2}\leq \tfrac{k^0-2}{2}$ the following assertions hold:
\begin{enumerate}
\item For every $\sigma\in {\rm Aut}(\C)$,
$$
\sigma\left(
\frac{L(\tfrac 12+m,\Pi_f \otimes \chi_f)}{\omega^{(-1)^{m+1}\epsilon_{\chi}}(\Pi_f) \, \G(\chi_f)^2 \, \omega(\Pi_\infty,m)}\right) \ = \
\frac{L(\tfrac 12+m,{}^\sigma\Pi_f \otimes {}^\sigma\!\chi_f)}{\omega^{(-1)^{m+1}\epsilon_\chi}({}^\sigma\Pi_f) \, \G({}^\sigma\!\chi_f)^2 \,\omega(\Pi_\infty,m)}.
$$
\item
$$
L(\tfrac 12+m,\Pi_f\otimes\chi_f) \ \sim_{\Q(\Pi,\eta, \chi)} \
\omega^{(-1)^{m+1}\epsilon_\chi}(\Pi_f) \, \G(\chi_f)^2  \, \omega(\Pi_\infty,m),
$$
where ``$\sim_{\Q(\Pi,\eta,\chi)}$'' means up to multiplication by an element in the number field $\Q(\Pi,\eta,\chi)$.
\end{enumerate}
\end{cor}
\begin{proof}
This is a direct consequence of Prop.\ \ref{prop:hilbertmod} and Prop.\ \ref{prop:crit}, which implies that the set of critical points for $\Pi$ is given by $${\rm Crit}(\Pi)={\rm Crit}(\Pi\otimes\chi)=\{\tfrac12+m\in\tfrac12+\Z|-\tfrac{k^0-2}{2}\leq m +\tfrac{3k_0}{2}\leq \tfrac{k^0-2}{2}\}.$$
\end{proof}

The critical values of symmetric cube $L$-functions of a Hilbert modular form ${\bf f}$ have been studied by Garrett and Harris \cite[Thm.\ 6.2]{garrett-harris}.
They analyzed the critical values of triple product $L$-functions and obtained those for symmetric cube $L$-functions as a by-product.
One can use the symmetric cube $L$-values as an anchor to deduce certain relations between the top-degree periods of this paper and the Petersson norm of ${\bf f}$.
We record such a period relation in the corollary below. For simplicity we work with an elliptic modular form of even weight $k$,
but the reader should be aware that a similar, but far more tedious, exercise can be carried through in the Hilbert modular setting.

\begin{cor}[Period Relations I]
\label{cor:period-relns-garrett-harris}
Let ${\bf f} \in S_k(\n,\tilde\omega)$ be a primitive holomorphic elliptic modular cusp form. Assume (for simplicity) that $k$ is even.
Let $\Pi = {\rm Sym}^3( \pi({\bf f})).$ Put $k' = \tfrac{k}{2}-1$ and $\epsilon = \epsilon_{k'} = (-1)^{k'+1}.$
We have
$$
\omega^{\epsilon}(\Pi_f) \,  \omega^{\epsilon}(\pi({\bf f})_f)^2
\ \sim_{\Q(\pi({\bf f}))} \
 c(\Pi_\infty, k') \,
 \G(\tilde\omega)^{-8} \,
\langle {\bf f}, {\bf f} \rangle_{\rm Blasius}^3,
$$
where $\langle {\bf f}, {\bf f} \rangle_{\rm Blasius}$ is the Petersson norm of ${\bf f}$ normalized as in Blasius \cite{blasius-appendix}.
\end{cor}

\begin{proof}
Consider the triple product $L$-function $L(s, {\bf f} \times {\bf f} \times {\bf f})$ as defined in \cite[Introduction]{garrett-harris}. We have
\begin{eqnarray*}
L(s, {\bf f} \times {\bf f} \times {\bf f}) & = &
L(s - \tfrac{3(k-1)}{2}, \pi({\bf f})_f \times \pi({\bf f})_f \times \pi({\bf f})_f) \\
& = & L(s - \tfrac{3(k-1)}{2}, \Pi_f)\cdot L(s - \tfrac{3(k-1)}{2}, \pi({\bf f})_f \otimes\tilde\omega)^2.
\end{eqnarray*}

Put $s = 2(k-1)$, which is the right--most critical point of $L(s, {\bf f} \times {\bf f} \times {\bf f})$; see \cite[(6.4.1)]{garrett-harris}. Using (6.4.2) and (6.4.3) of \cite{garrett-harris} we obtain
$$
L(s, {\bf f} \times {\bf f} \times {\bf f}) \ \sim_{\Q(\pi({\bf f}))} \
(2 \pi i)^{2(k-1)}  \, \G(\tilde\omega)^{-6} \,
\langle {\bf f}, {\bf f} \rangle_{\rm Blasius}^3.
$$
(Note that the Gauss sum $G(\tilde\omega)$ of \cite{garrett-harris} is our $\G(\tilde\omega^{-1}) \sim_{\Q(\tilde\omega)} \G(\tilde\omega)^{-1}.$)

On the other hand, by the above Corollary for $\Pi = {\rm Sym}^3(\pi({\bf f}))$, and noting that $s - \tfrac{3(k-1)}{2}$ at $s = 2(k-1)$ is nothing but $\tfrac12 + k'$, we get
$$
L(\tfrac12+k', \Pi_f) \sim_{\Q(\pi({\bf f}))}  \omega^{\epsilon}(\Pi_f) \, \omega(\Pi_\infty,k').
$$
By using Shimura's classical theorem on the critical values for $L(s - \tfrac{3(k-1)}{2}, \pi({\bf f})_f \otimes  \tilde\omega)$, in the form stated in \cite[Corollary 1.3]{ragtan}, we obtain furthermore
$$
L(\tfrac12+k', \pi({\bf f})_f \otimes\tilde\omega)^2
\ \sim_{\Q(\pi({\bf f}))} \
 (2 \pi i)^{2(k-1)}\,
 \omega^{\epsilon}(\pi({\bf f})_f)^2 \,
 \G(\tilde\omega)^2.
$$
The corollary now follows keeping in mind that $\omega(\Pi_\infty,k')^{-1} = c(\Pi_\infty, k')$, cf.\  Sect.\ \ref{sect:hypo}.
\end{proof}

By working with the critical point $m = 2k-3$, which is next to the right-most critical point of $L(s, {\bf f} \times {\bf f} \times {\bf f})$, one may deduce a similar relation for $\omega^{-\epsilon}(\Pi_f) \,  \omega^{-\epsilon}(\pi({\bf f})_f)^2$. We leave the details to the reader.

\subsubsection{Higher symmetric powers}\label{sect:highersym}
The entire discussion may be generalized to higher symmetric powers. Let $\Pi_r:={\rm Sym}^{2r+1}(\pi)$ for $r\geq 2$. Criteria for $\Pi_r$ being cuspidal automorphic are not yet known in all generality; however, there are some very interesting results due to Ramakrishnan \cite{ram-cuspidal}. Assuming Langlands Functoriality, $\Pi_r$ should at least always be automorphic.

\begin{prop}\label{prop:highersym}
Let $k=(k_v)_{v\in S_\infty}\in\Z^d$ with $k_v\geq 2$ and $k_v\equiv k_{w}\pmod 2$ for all $v,w\in S_\infty$. Let ${\bf f}\in S_k(\n,\tilde\omega)$ be a primitive holomorphic Hilbert modular cusp form, $\pi({\bf f})$ the corresponding cuspidal automorphic representation of $\GL_2(\A)$ and write $\pi=\pi({\bf f})\otimes |\cdot|^{k_0/2}$. The odd symmetric power transfer
$$\Pi_r={\rm Sym}^{2r+1}(\pi), \quad\quad r\geq 2,$$
of $\pi$ is an irreducible admissible representation of $\GL_{2(r+1)}(\A)$ which is cohomological with respect to a highest weight module $E^{\sf v}_\mu$. Assume furthermore that $\Pi_r$ is cuspidal automorphic and ${\rm Sym}^{4(r-a)}(\pi({\bf f}))$, $0\leq a \leq r$, is an isobaric direct sum of unitary cuspidal automorphic representations. Then, $\Pi_r$ admits an $(\eta,\psi)$-Shalika model with $\eta=\omega_{\pi({\bf f})}^{2r+1}|\cdot|^{(2r+1)k_0}$. In particular, Thm.\ \ref{thm:central-value} can be applied to the odd symmetric power transfer $\Pi_r={\rm Sym}^{2r+1}(\pi)$, $r\geq 2$.
\end{prop}
\begin{proof}
For each $v\in S_\infty$, let
$$\mu_v=(k_v-2)\rho_{2(r+1)}+\tfrac{(2r+1)k_0}{2},$$
where $\rho_{2(r+1)}$ is half the sum of positive roots of $\GL_{2(r+1)}(\R)$. Then, $\Pi_v$ is cohomological with respect to $E^{\sf v}_{\mu_v}$ for all $v\in S_\infty$ by Raghuram--Shahidi \cite[Thm.\ 5.5]{ragsha} and so, using the K\"unneth rule, $\Pi$ is cohomological with respect to the highest weight module
$E^{\sf v}_\mu$ with $\mu=(\mu_v)_{v\in S_\infty}$. A similar calculation as in the proof of Prop.\ \ref{prop:hilbertmod} shows that
$$\wedge^2(\Pi_r)=\wedge^2\left({\rm Sym}^{2r+1}\left(\pi\right)\right)\cong \left(\boxplus_{a=0}^{r}{\rm Sym}^{4(r-a)}(\pi({\bf f}))\otimes\omega_{\pi({\bf f})}^{2a+1}\right)\otimes |\cdot|^{(2r+1)k_0}.$$
Hence, by Jacquet--Shalika \cite[Thm.\ (1.3)]{jacquet-shalika-invent} and Thm.\ \ref{thm:JS}, $\Pi_r$ has an $(\eta,\psi)$-Shalika model with $\eta=\omega_{\pi({\bf f})}^{2r+1}|\cdot|^{(2r+1)k_0}$.
\end{proof}

\subsection{Rankin--Selberg $L$-functions for $\GL_3\times\GL_2$}\label{sec:gl3xgl2}
\subsubsection{}\label{sect:81}
We describe another class of examples where our theorem applies, and this concerns Rankin--Selberg $L$-functions for $\GL_3 \times \GL_2$ via transfer to $\GL_6$.
Let $F$ be totally real as before, and $\pi=\otimes_v'\pi_v$ (resp., $\tau=\otimes_v'\tau_v$) be a unitary cuspidal automorphic representation of $\GL_3(\A)$ (resp. $\GL_2(\A)$). For each place $v$ of $F$, let $\pi_v\boxtimes\tau_v$ be the irreducible admissible representation of $\GL_6(F_v)$ attached to $\pi_v\otimes\tau_v$ via the Local Langlands Correspondence (\cite{harris}, \cite{henniart}, \cite{lang2}). Then, $\Pi:=\pi\boxtimes\tau$ is an irreducible admissible representation of $\GL_6(\A)$, which by Kim--Shahidi \cite[Thm.\ 5.1]{kimshahidiANN}, is automorphic. Recall also the symmetric square transfer, cf.\  Sect.\ \ref{sect:symmtrans}. By Gelbart--Jacquet \cite{gelbart-jacquet} it assigns to each unitary cuspidal automorphic $\tau$ as above an automorphic representation Sym$^2(\tau)$ of $\GL_3(\A)$.\\

Let $v\in S_\infty$. We say that $\pi_v$ is cohomological with respect to a highest weight module $E^{\sf v}_{\mu_v}$ of $\GL_3(\R)$ if $H^q(\gl_3(\R), {\rm O}(3)\R_+,\pi_v\otimes E^{\sf v}_{\mu_v})\neq 0$ for some $q$. Similarly, we say that $\pi_\infty$ is cohomological with respect to $E^{\sf v}_\mu$, $\mu=(\mu_v)_{v\in S_\infty}$, if $\pi_v$ is cohomological with respect to $E^{\sf v}_{\mu_v}$ for all $v\in S_\infty$. Observe that $\pi_v$ being unitary implies that  $\mu_v=(\mu_{v,1}, 0, -\mu_{v,1})\in\Z^3$. Putting $\ell_v:= 2\mu_{v,1} +2$, we obtain that $\pi_v$ is necessarily isomorphic to
$$
\pi_v\cong {\rm Ind}_{P_{(2,1)}(\R)}^{\GL_3(\R)}[D(\ell_v) \otimes \sgn^{\varepsilon_v}_v]
$$
for a uniquely defined $\varepsilon_v=\varepsilon_v(\mu_v)\in\{0,1\}$. For this, see, for instance, \cite[Sect.\ 3.1]{mahnk} or \cite[Sect.\ 5.1]{ragsha}.

\subsubsection{} With this setup in place we obtain the following proposition.
\begin{prop}\label{prop:gl2gl3}
Let $\pi$ (resp., $\tau$) be a unitary cuspidal automorphic representation of $\GL_3(\A)$ (resp., $\GL_2(\A)$). Let $\pi$ be cohomological with respect to $E^{\sf v}_\mu$, $\mu=(\mu_v)_{v\in S_\infty}$, $\mu_v=(\mu_{v,1}, 0, -\mu_{v,1})$, and let $\tau$ be cohomological with respect to $E^{\sf v}_\lambda$, $\lambda=(\lambda_v)_{v\in S_\infty}$, $\lambda_v=(\lambda_{v,1}, -\lambda_{v,1})$. (Note that the unitarity of $\pi$ and $\tau$ forces the weights
$\mu$ and $\lambda$ to be self--dual and hence to be of the above form.) Put $\ell_v:= 2\mu_{v,1} +2$ and $\ell'_v:= 2\lambda_{v,1}+1$. Let $\omega_\tau$ be the central character of $\tau$.
Assume furthermore that
\begin{enumerate}
\item $\tau$ is not dihedral and $\pi$ is not a twist of \emph{Sym}$^2(\tau)$.
\item $\pi_\infty$ and $\tau_\infty$ are ``sufficiently disjoint'', i.e., $\ell_v>\ell'_v$ and $\ell_v\neq 2\ell'_v$ for all $v\in S_\infty$.
\item $\pi$ is essentially self-dual; say $\pi^{\sf v} \simeq \pi \otimes \chi_\pi$.
\end{enumerate}
Then $\Pi=\pi\boxtimes\tau$ is a cohomological cuspidal automorphic representation of $\GL_6(\A)$ which has an $(\eta,\psi)$-Shalika model for $\eta = \omega_\tau \chi_\pi^{-1}$ Further, the standard $L$-function $L(s,\Pi)$ of $\Pi$ is the Rankin-Selberg $L$-function $L(s,\pi\times\tau)$ of the pair $(\pi,\tau)$.
\end{prop}

\begin{proof}
For convenience, we divide the proof into four parts:

\smallskip

{\it $\Pi$ is cuspidal:} This follows from the cuspidality criterion \cite[Thm.\ 3.1(a)]{ramwong} of Ramakrishnan--Wang for the Kim--Shahidi transfer: The hypothesis that $\tau$ is not dihedral and that $\pi$ is not a twist of the Gelbart--Jacquet transfer of $\tau$ guarantees cuspidality of the
transfer $\Pi$.

\smallskip

{\it $\Pi$ is cohomological:} Recall from Sect.\ \ref{sect:cohreps} and from Sect.\ \ref{sect:81} above that for each $v\in S_\infty$ we have
$$
\pi_v\cong {\rm Ind}_{P_{(2,1)}(\R)}^{\GL_3(\R)}[D(\ell_v) \otimes \sgn^{\varepsilon_v}_v]\ \  {\rm and} \ \ \tau_v\cong D(\ell_v'),
$$
This shows that
$$
\Pi_v\cong {\rm Ind}_{P(\R)}^{\GL_6(\R)}[D(\ell_v+\ell'_v)  \otimes D(\ell_v-\ell'_v) \otimes D(\ell_v') ].
$$
Now one can check that $\Pi_\infty$ is regular algebraic, i.e., cohomological (cf.\  Rem.\  \ref{rem:clozel}) if $\ell_v>\ell'_v$ and $\ell_v\neq 2\ell'_v$ for all $v\in S_\infty$.

\smallskip

{\it $\Pi$ has a $(\eta, \psi)$-Shalika model:} We show that the $\eta^{-1}$-twisted partial exterior square $L$-function of $\Pi$ has a pole at $s=1$. This hinges on the following easy identity in linear algebra: Let $V$ and $W$ be
finite-dimensional vector spaces over some field then
$$
\wedge^2(V \otimes W)=\left({\rm Sym}^2(V) \otimes \wedge^2W \right) \oplus \left( \wedge^2V \otimes {\rm Sym}^2(W)\right).
$$
Applying this to a local unramified place $v$ of $F$, we see the following factorization of partial $L$-functions for any finite set of places  $S$ containing $S_{\Pi,\eta}$:
\begin{equation}
\label{eqn:wedge-2-factorize}
L^S(s, \Pi, \wedge^2 \otimes \eta^{-1}) =
L^S(s, \pi \otimes \tau, {\rm Sym}^2 \otimes \wedge^2 \otimes \eta^{-1})\cdot
L^S(s, \pi \otimes \tau, \wedge^2  \otimes  {\rm Sym}^2 \otimes \eta^{-1}).
\end{equation}
But $\wedge^2\tau$ is nothing but the central character $\omega_\tau$ of $\tau$, hence the first factor of the right hand side of
(\ref{eqn:wedge-2-factorize}) may be rewritten as
$$
L^S(s, \pi, {\rm Sym}^2 \otimes \omega_\tau \eta^{-1}) = L^S(s, \pi, {\rm Sym}^2 \otimes \chi_\pi),
$$
which has a pole at $s=1$, since we assumed that $\pi^{\sf v} = \pi \otimes \chi_\pi.$ Indeed, as $\pi$ is essentially self--dual,
$$L^S(s, \pi^{\sf v} \times \pi) = L^S(s, \pi, {\rm Sym}^2 \otimes \chi_\pi)\cdot L^S(s, \pi, \wedge^2 \otimes \chi_\pi)$$ has a pole at $s=1$. But since $\pi$ is on $\GL_3$,
we have
$$L^S(s, \pi, \wedge^2 \otimes \chi_\pi) = L^S(s, \pi^{\sf v} \otimes \omega_\pi\chi_\pi) =  L^S(s, \pi \otimes \omega_\pi\chi_\pi^2),$$ which is entire by the cuspidality of $\pi$. Hence, the pole at $s=1$ of $L^S(s, \pi^{\sf v} \times \pi)$ must come from $L^S(s, \pi, {\rm Sym}^2 \otimes \chi_\pi)$.

Now let us look at the second factor on the right hand side of (\ref{eqn:wedge-2-factorize}).
Since $\tau$ is not dihedral, the symmetric square transfer ${\rm Sym}^2(\tau)$ of $\tau$ is cuspidal by Gelbart-Jacquet \cite[Thm.\ 9.3]{gelbart-jacquet}. Moreover, since $\pi$ is on $\GL_3$, as above, we have $\wedge^2 \pi = \pi^{\sf v} \otimes \omega_\pi = \pi \otimes \chi_\pi\omega_\pi.$
Hence, the second factor is the same as
$$
L^S(s, \pi \times {\rm Sym}^2(\tau) \otimes \chi_\pi\omega_\pi\eta^{-1}) = L^S(s, \pi \times ({\rm Sym}^2(\tau) \otimes \beta))
$$
for the unitary character $\beta = \chi_\pi^2\omega_\pi\omega_\tau^{-1}$, i.e., it equals the partial Rankin-Selberg $L$-function for $\GL_3 \times \GL_3$ attached to the unitary cuspidal automorphic representations $\pi$ and ${\rm Sym}^2(\tau) \otimes \beta$. The second factor is therefore non-vanishing at $s=1$ by Shahidi \cite{shahidi-book},
Thm.\ on p.\ 462. We conclude that $L^S(s, \Pi, \wedge^2 \otimes \eta^{-1})$ has a pole at $s=1$. It follows from the equivalence of (i) and (iii) of Thm.\ \ref{thm:JS} that $\Pi$ has an $(\eta,\psi)$-Shalika model. Compare these considerations to Gotsbacher--Grobner, \cite{grobgots}, Sect.s 4.2 and 4.3.

\smallskip

{\it Equality of $L$-functions:} Finally, the equality of $L(s,\Pi)$ and $L(s,\pi\times\tau)$ is proved in \cite{kimshahidiANN}, Prop.\ 5.8.
\end{proof}

\begin{cor}\label{cor:gl2gl3}
Thm.\ \ref{thm:central-value} can be applied to any representation $\Pi=\pi\boxtimes\tau$ as in Prop.\ \ref{prop:gl2gl3}. In particular, if $\chi$ is a finite-order Hecke character of $F$, then for all $\tfrac12+m\in{\rm Crit}(\Pi)$, the following assertions hold:
\begin{enumerate}
\item For every $\sigma\in {\rm Aut}(\C)$,
$$
\sigma\left(
\frac{L(\tfrac 12+m,\pi_f \times \tau_f\otimes \chi_f)}{\omega^{(-1)^{m}\epsilon_{\chi}}(\Pi_f) \, \G(\chi_f)^3 \, \omega(\Pi_\infty,m)}\right) \ = \
\frac{L(\tfrac 12+m,{}^\sigma\!\pi_f \times {}^\sigma\!\tau_f\otimes {}^\sigma\!\chi_f)}{\omega^{(-1)^{m}\epsilon_\chi}({}^\sigma\Pi_f) \, \G({}^\sigma\!\chi_f)^3 \,\omega(\Pi_\infty,m)}.
$$
\item
$$
L(\tfrac 12+m,\pi_f\times\tau_f\otimes\chi_f) \ \sim_{\Q(\Pi,\eta, \chi)} \
\omega^{(-1)^{m}\epsilon_\chi}(\Pi_f) \, \G(\chi_f)^3  \, \omega(\Pi_\infty,m),
$$
where ``$\sim_{\Q(\Pi,\eta,\chi)}$'' means up to multiplication by an element in the number field $\Q(\Pi,\eta,\chi)$.
\end{enumerate}
\end{cor}

In Raghuram \cite{raghuram-imrn}, for $F=\Q,$ another type of periods $p^+(\pi_f)$, $p^-(\tau_f)$ and
$p_\infty(\mu,\lambda)$ was introduced for cuspidal automorphic representations $\pi$ and $\tau$ as above -- the latter one exists thanks to Kasten--Schmidt \cite[Sect.\ 4]{kasten-schmidt}. This was done by considering cohomology in bottom-degree. Using the results obtained in \cite{raghuram-imrn}, we get another corollary, which compares our top-degree (Shalika--)periods for $\Pi$ with the bottom-degree (Whittaker--)periods of $\pi$ and $\tau$.

\begin{cor}[Period Relations II]
\label{cor:gl2gl3-period-relations}
Let $F = \Q$ and  let $\pi$, $\mu,$ $\tau,$ $\lambda$ and $\Pi =\pi\boxtimes\tau$ be as in Prop.\ \ref{prop:gl2gl3}. We get the following relation:
$$
\omega^{\epsilon_0}(\Pi_f) \, \omega(\Pi_\infty) \ \sim_{\Q(\pi,\tau,\eta)} \
p^+(\pi_f) p^-(\tau_f) \G(\omega_{\tau_f}) p_\infty(\mu,\lambda),
$$
where $\Q(\pi,\tau,\eta)$ is the composition of the rationality fields of $\pi$, $\tau$ and $\eta$ and the rest of the notations are as in \cite{raghuram-imrn}.
\end{cor}

\begin{proof}
This follows directly from comparing the algebraicity results for $L(\tfrac12, \Pi) = L(\tfrac12, \pi \times \tau)$ given by Theorem~\ref{thm:central-value} for $L(\tfrac12, \Pi)$ and by \cite[Theorem 1.1]{raghuram-imrn} for $L(\tfrac12, \pi \times \tau).$
\end{proof}

\subsubsection{Rankin--Selberg $L$-functions for $\GL_n\times\GL_{n-1}$}\label{sect:glngln-1}
We would like to point out that -- similar to the case of Sym$^3$ -- the entire discussion in this section may be generalized, assuming Langlands Functoriality, to the situation where $\pi$ (resp., $\tau$) is a unitary essentially self--dual cuspidal automorphic representation of $\GL_n(\A)$ (resp., $\GL_{n-1}(\A)$) such that one of them is of symplectic type and the other of orthogonal type and for which the transfer $\Pi = \pi \boxtimes \tau$ is cuspidal as a representation of $\GL_{n(n-1)}(\A)$. The same remark applies to the corollary on period relations, if one additionally assumes the validity of Hypothesis 3.10 of \cite{raghuram-imrn}. However, as pointed out by Sun \cite{sun}, p.\ 4, this hypothesis may soon be proved to hold, applying similar techniques as used in the proof of \cite{sun} Thm. C.

\subsection{Degree four $L$-function of a Siegel modular form}
In this section we let $F=\Q$. Let $\Phi$ be a non-zero holomorphic cuspidal scalar-valued Siegel modular eigenform of degree $2$, weight $\ell$ and of full level, i.e., for the full modular group ${\rm Sp}_4(\Z).$ (Existence of such a $\Phi$ implies $\ell \geq 10$.) Let $\pi = \pi(\Phi)$ be the cuspidal automorphic representation of ${\rm GSp}_4(\A)$, associated to $\Phi$ as in \cite[Thm.\ 2]{asgari-schmidt}. Then $\pi$ has trivial central character, is unramified everywhere, and the representation $\pi_\infty$ at infinity is a holomorphic discrete series representation. The representation $\pi$ is not globally generic since $\pi_\infty$ is not generic, hence $\pi$ does not come under the purview of ``generic-transfer'' from ${\rm GSp}_4$ to $\GL_4$ of Asgari--Shahidi \cite[Prop.\ 7.8]{asgari-shahidi-generic}.

However, under the assumption that $\Phi$ is not of Saito-Kurokawa type (i.e., $\pi=\pi(\Phi)$ is not CAP with respect to the Siegel parabolic subgroup or the Borel subgroup of $\GSp$), Pitale--Saha--Schmidt \cite{pss} have recently proved the existence of the ``non-generic-transfer'' to a  representation $\Pi = \Pi(\Phi)$ of $\GL_4(\A)$. It follows from \cite{pss} that $\Pi = \Pi(\Phi)$ is a cuspidal automorphic representation whose exterior square $L$-function has a pole at $s=1$. Hence, by Theorem~\ref{thm:JS}, we get that $\Pi$ has a $(\triv, \psi)$-Shalika model. Next, the Langlands parameter of $\Pi_\infty$, described in \cite{pss}, is the representation
$$
{\rm Ind}_{\C^\times}^{W_\R}[(\tfrac{z}{\overline z})^{\frac12}] \oplus {\rm Ind}_{\C^\times}^{W_\R}[(\tfrac{z}{\overline z})^{\frac{2\ell-3}{2}}].
$$
It is easy to see then that $\Pi_\infty\cong {\rm Ind}_{P(\R)}^{\GL_4(\R)}[D(2\ell-3)\otimes D(1)]$ is cohomological with respect to $E^{\sf v}_\mu$ with $\mu=(\ell-3,0,0,-(\ell-3))$, cf.\  Sect.\ \ref{sect:cohreps}. All these observations collectively say that $\Pi = \Pi(\Phi)$ is a representation to which our main theorem on special values, Thm.\ \ref{thm:central-value}, applies. The standard $L$-function of $\Pi$ is the degree four spinor $L$-function of $\Phi$ and so we get a description of the critical values of twisted degree four $L$-functions of $\Phi$ in terms of the top-degree periods $\omega^\epsilon(\Pi_f)$ of the transferred representation. We record this as the following corollary to Thm.\ \ref{thm:central-value}.

\begin{cor}\label{cor:siegelmod}
Let $F=\Q$ and let $\Phi$ be a non-zero holomorphic cuspidal scalar-valued Siegel modular eigenform of degree $2$, weight $\ell$ and for the full modular group ${\rm Sp}_4(\Z)$.
Let $\Pi = \Pi(\Phi)$ be the cuspidal automorphic representation of $\GL_4(\A)$ attached to $\Phi$ by Pitale--Saha--Schmidt \cite{pss}. For any finite-order Hecke character $\chi$ of $\Q$ the following assertions hold:
\begin{enumerate}
\item For every $\sigma\in {\rm Aut}(\C)$,
$$
\sigma\left(
\frac{L(\tfrac 12,\Pi_f \otimes \chi_f)}{\omega^{-\epsilon_{\chi}}(\Pi_f) \, \G(\chi_f)^2 \, \omega(\Pi_\infty)}\right) \ = \
\frac{L(\tfrac 12,{}^\sigma\Pi_f \otimes {}^\sigma\!\chi_f)}{\omega^{-\epsilon_\chi}({}^\sigma\Pi_f) \, \G({}^\sigma\!\chi_f)^2 \, \omega(\Pi_\infty)}.
$$
\item
$$
L(\tfrac 12,\Pi_f\otimes\chi_f) \ \sim_{\Q(\Pi, \chi)} \
\omega^{-\epsilon_\chi}(\Pi_f) \, \G(\chi_f)^2  \, \omega(\Pi_\infty),
$$
where ``$\sim_{\Q(\Pi, \chi)}$'' means up to multiplication by an element in the number field $\Q(\Pi, \chi)$.
\end{enumerate}
\end{cor}

Note that $L(s,\Pi_f \otimes \chi_f)$ has only one critical point, namely $s = \tfrac12$. This follows from the Prop.\ \ref{prop:crit}, recalling that $\Pi$ is cohomological with respect to $E_\mu\cong E^{\sf v}_\mu$ with $\mu=(\ell-3,0,0,-(\ell-3))$.

\begin{rem}[Period relations III]
The critical values of degree four $L$-functions for ${\rm GSp}(4)$ have been studied by Harris~\cite{harris-occult}.
The periods appearing therein come via a comparison of rational structures on Bessel models and rational structures on coherent cohomology.
Using the $L$-values as an anchor, one may compare the whimsically titled ``occult'' periods of Harris with the top-degree periods in this paper in the
situation where the representation of ${\rm GSp}_4$ comes from a Siegel modular form $\Phi$ as considered above.
\end{rem}

\begin{rem}
With the current state of Langlands functoriality we can only deal with Siegel modular forms of genus $2$ and full level, however, it is clear that the entire discussion in this subsection can be generalized to give algebraicity results for the degree four
$L$-functions for holomorphic Hilbert-Siegel modular cusp forms of genus $2$ and arbitrary level. Further, although we did not work out the details, using Arthur's classification of the discrete spectrum for classical groups  (see \cite{arthur-book}), one should be able to get algebraicity results for spinor $L$-functions for certain representations of the split group
${\rm SO}(2n+1)$ over a totally real field.
\end{rem}

\subsection{Compatibility with Deligne's conjecture}
Given a critical motive $M$, a celebrated conjecture of Deligne~\cite[Conj.\ 2.8]{deligne} relates the critical values of its $L$-function $L(s,M)$ to certain periods that arise out of a comparison of the Betti and de~Rham realizations of the motive. One expects a cohomological cuspidal automorphic representation $\Pi$ to correspond to a motive $M(\Pi)$ and under this correspondence the standard $L$-function $L(s, \Pi)$ is the motivic $L$-function $L(s, M(\Pi))$ up to a shift in the $s$-variable; see Clozel \cite{clozel}, Sect.\ 4. However, with the current state of technology, it seems impossible to compare our periods $\omega^\epsilon(\Pi_f)$ with Deligne's periods $c^{\pm}(M(\Pi))$. Blasius \cite{blasius} and Panchishkin \cite{panchishkin} have studied the behaviour of Deligne's periods upon twisting the motive by a Dirichlet character (more generally by Artin motives). Using Deligne's conjecture, they then predict the behaviour of critical values of motivic $L$-functions upon twisting by Dirichlet characters. For a critical motive over $\Q$, assumed to be simple, and of rank $2n$ this prediction looks like
$$
L(m, M \otimes \chi_f) \ \sim_{\Q(M,\chi)} \  L(m,M) \mathcal{G}(\chi_f)^n.
$$
Observe that no periods need to be mentioned to make such a statement about $L$-values. Our
Thm.\ \ref{thm:central-value} is compatible with Deligne's conjecture in the sense that an analogous relation holds between critical values of $L(s, \Pi)$ and $L(s, \Pi \otimes \chi)$.

\begin{cor}
\label{cor:deligne}
Let $F$ be a totally real number field and $G=\GL_{2n}/F$, $n\geq 1$. Let $\Pi$ be a cuspidal automorphic representation of $G(\A)$, which is cohomological with respect to a highest weight representation $E^{\sf v}_\mu$ of $G_\infty$ and which admits an $(\eta,\psi)$-Shalika model. Let $\chi$ be a Hecke character of $F$ of finite order and let $\tfrac12 + m \in \emph{Crit}(\Pi)=\emph{Crit}(\Pi\otimes\chi)$. We have
$$
L(\tfrac 12+m,\Pi_f \otimes \chi_f) \ \sim_{\Q(\Pi, \eta, \chi)} \
L(\tfrac 12+m,\Pi_f) \mathcal{G}(\chi_f)^n.
$$
\end{cor}

\subsection{Compatibility with Gross's conjecture}
A conjecture due to Gross \cite[Conj.\ 2.7(ii)]{deligne} says that the order of vanishing of a motivic $L$-function at a critical point is independent of which conjugate of the motive we are looking at, i.e., if $M$ is critical, then ${\rm ord}_{s=0} L(s, \sigma, M)$ is independent of the embedding $\sigma : \Q(M) \to \C$. We are unable to say anything about the ``order'' of vanishing, however,  it follows  trivially from Thm.\ \ref{thm:central-value} that  the property of vanishing is independent of which particular conjugate of the representation we consider.

\begin{cor}
\label{cor:gross}
Let $F$ be a totally real number field and $G=\GL_{2n}/F$, $n\geq 1$. Let $\Pi$ be a cuspidal automorphic representation of $G(\A)$, which is cohomological with respect to a highest weight representation $E^{\sf v}_\mu$ of $G_\infty$ and which admits an $(\eta,\psi)$-Shalika model. Let $\chi$ be a finite-order Hecke character of $F$ and let $\tfrac12 + m \in$ \emph{Crit}$(\Pi)=$\emph{Crit}$(\Pi\otimes\chi)$. Then for $\sigma\in {\rm Aut}(\C)$,
$$
L(\tfrac 12+m,\Pi_f \otimes \chi_f) = 0  \ \iff \  L(\tfrac 12+m,{}^\sigma\Pi_f \otimes {}^\sigma\!\chi_f) = 0.
$$
\end{cor}

\bigskip

\appendix
\section*{Appendix: Arithmeticity for Shalika models}\label{appendix}
\begin{center}
{\it by Wee Teck Gan}
\end{center}
\vskip 5pt
The purpose of this appendix is to prove Thm.\ \ref{thm:arithmeticShalika}. More precisely, we show:

\begin{thm*}
Let $\Pi$ be a cohomological, cuspidal automorphic representation of $\GL_{2n}(\A)$ which admits an $(\eta, \psi)$-Shalika model, then for any $\sigma \in {\rm Aut}(\C)$, ${^{\sigma}}\Pi$ admits a $({^{\sigma}}\!\eta, \psi)$-Shalika model.
\end{thm*}

\begin{proof}
Since $\Pi$ has $(\eta, \psi)$-Shalika model, it follows by Thm.\ \ref{thm:JS} that $L^S(s, \Pi, \wedge^2 \otimes \eta^{-1})$ has a pole at $s=1$, and thus $ \Pi^{\sf v} \cong \Pi \otimes \eta^{-1}$. Further, the reader is reminded of
Lem.\ \ref{lem:eta} which says that all the archimedean components of $\eta$ are equal to
${\rm sgn}^{\sf w}|\ |^{\sf w}.$
Now, we note the following:

\begin{itemize}
\item By recent results of Asgari--Shahidi \cite{asgari-shahidi-generic, asgari-shahidi} and Hundley--Sayag \cite{hundleysayag}, $\Pi$ is a Langlands functorial lift of a cuspidal representation of $\GSpin_{2n+1}(\A)$ with central character $\eta$. See Prop.\ \ref{prop:selfdual}. Moreover, the lift is strong at the archimedean places, i.e., for each archimedean place, the L-parameter $\phi_v$ of $\Pi_v$ factors through the dual group ${\rm GSp}_{2n}(\C)$ of $\GSpin_{2n+1}$ with similitude character $\eta_v$.
\vskip 5pt

\item For any $\sigma \in {\rm Aut}(\C)$,
\[  {^{\sigma}}\Pi^{\sf v} \cong {^{\sigma}}\Pi \otimes  {^{\sigma}}\!\eta^{-1}, \]
and thus
\[  L^S(s, {^{\sigma}}\Pi \otimes  {^{\sigma}}\Pi \otimes {^{\sigma}}\!\eta^{-1}) = L^S(s, {^{\sigma}}\Pi, {\rm Sym}^2 \otimes {^{\sigma}}\!\eta^{-1}) \cdot
L^S(s, {^{\sigma}}\Pi, \wedge^2 \otimes {{^{\sigma}}\!\eta^{-1}})  \]
has a pole at $s = 1$.
 \end{itemize}

To prove the theorem, we need to show that the ${\rm Sym}^2$ $L$-function does not have a pole at $s = 1$.
Suppose for the sake of contradiction that $ L^S(s, {^{\sigma}}\Pi, {\rm Sym}^2 \otimes {^{\sigma}}\!\eta^{-1})$ has a pole at $s  =1$. Then by Asgari--Shahidi \cite{asgari-shahidi-generic, asgari-shahidi} and Hundley--Sayag \cite{hundleysayag}, one knows that $ {^{\sigma}}\Pi$ is a Langlands functorial lift from a cuspidal representation of $\GSpin_{2n}(\A)$ with central character ${^{\sigma}}\!\eta$, and this lift is strong at the archimedean places. Since the archimedean components of ${^{\sigma}}\Pi$ and ${^{\sigma}}\!\eta$ are, by definition, permutations of the archimedean components of $\Pi$ and $\eta$, we deduce that for all archimedean places $v$, the L-parameter $\phi_v$ of $\Pi_v$ factors through the dual group ${\rm GSO}_{2n}(\C)$ of $\GSpin_{2n}$ with similitude character $\eta_v$.

As a result, for each archimedean place $v$, the L-parameter $\phi_v$ of $\Pi_v$ preserves both a non-degenerate symmetric bilinear form $b_1$ and a non-degenerate skew-symmetric bilinear form $b_2$ on $\C^{2n}$ up to the same similitude character $\eta_v$. However, since $\Pi_v$ is cohomological, it follows from \eqref{eq:piinfty} that $\phi_v$ is a direct sum of ($2$-dimensional) irreducible representations $\phi_{i,v}$ of the Weil group $W_{F_v}$, and each $\phi_{i,v}$ is not a twist of another $\phi_{j,v}$. This shows that $b_1$ and $b_2$ must remain non-degenerate when restricted to each $\phi_{i,v}$. This gives two $W_{F_v}$-equivariant isomorphisms $\phi_{i,v}^{\sf v} \cong \phi_{i,v} \otimes \eta_v^{-1}$. Since $\phi_{i,v}$ is irreducible, this contradicts Schur's lemma.
\end{proof}
The reader should see Gan--Raghuram~\cite{gan-raghuram}, where the above result is put into a broader context of arithmeticity for periods of automorphic forms.

\vskip 10pt
\footnotesize
{\sc Wee Teck Gan: Department of Mathematics, National University of Singapore, 10 Lower Kent Ridge Road
Singapore 119076}
\\ {\it E-mail address:} {\tt matgwt@nus.edu.sg}

\end{document}